\documentclass{article}
\usepackage{arxiv}

\usepackage{amsmath}
\usepackage{graphicx,psfrag,epsf}
\usepackage{enumerate}
\usepackage{natbib}
\usepackage{moreverb}
\usepackage{amssymb}
\usepackage{amsthm}
\usepackage{multicol}
\usepackage{xcolor}
\usepackage{float}

\newtheorem{theorem}{Theorem}
\newtheorem{lemma}{Lemma}

\usepackage{url} 

\usepackage{chngcntr}
\usepackage{apptools}
\AtAppendix{\counterwithin{lemma}{section}}
\AtAppendix{\counterwithin{theorem}{section}}
\AtAppendix{\counterwithin{equation}{section}}
\AtAppendix{\counterwithin{theorem}{section}}

\title{\bf Large-Sample Properties of Blind Estimation of the Linear Discriminant Using Projection Pursuit}

\author{Una Radoji\v{c}i\'c\\
    Vienna University of Technology, {Austria}\\
    \texttt{una.radojicic@tuwien.ac.at}\\
\And
Klaus Nordhausen \\
  University of Jyv\"askyl\"a, {Finland}\\
  \texttt{klaus.k.nordhausen@jyu.fi} \\
  \And
Joni Virta  \\
 University of Turku, {Finland}\\
  \texttt{jomivi@utu.fi} \\

  }

\begin{document}

\maketitle

\bigskip
\begin{abstract}
We study the estimation of the linear discriminant with projection pursuit, a method that is blind in the sense that it does not use the class labels in the estimation. Our viewpoint is asymptotic and, as our main contribution, we derive central limit theorems for estimators based on three different projection indices, skewness, kurtosis and their convex combination. The results show that in each case the limiting covariance matrix is proportional to that of linear discriminant analysis (LDA), an unblind estimator of the discriminant. An extensive comparative study between the asymptotic variances reveals that projection pursuit is able to achieve efficiency equal to LDA when the groups are arbitrarily well-separated and their sizes are reasonably balanced. We conclude with a real data example and a simulation study investigating the validity of the obtained asymptotic formulas for finite samples.
\end{abstract}
\bigskip



\section{Introduction}\label{sec:introduction}

Classification and clustering are two central themes in modern data analysis and can be seen, respectively, as the ``unblind'' and ``blind'' versions of the same problem: In classification the group memberships, or labels, of the training data points are known and the objective is to use the training data to form a classification rule for future observations, some of the standard methods including, e.g., linear discriminant analysis, support vector machines and random forests, see \cite{friedman2001elements}. Whereas in clustering, no labels for the data points are known but we postulate that a reasonable grouping exists and aim to find it, with, e.g., $k$-means clustering or spectral clustering, see \cite{friedman2001elements,von2007tutorial}. 

In this paper, we work under a clustering context and the assumption that the data admit a natural grouping but that their labels are indeed unknown to us. In their seminal work, \cite{pena2001cluster} studied in this setting the use of \textit{projection pursuit} (PP), a general family of methods searching for a projection direction that maximizes the value of the so-called projection index, see, e.g., \cite{huber1985projection,BoltonKrzanowski2003,bickel2018projection,fischer2019repplab} and the references therein. Namely, denoting the within-class covariance matrix by $\boldsymbol{\Sigma}$ and the two group means by $\boldsymbol{\mu}_1, \boldsymbol{\mu}_2$, \cite{pena2001cluster} established that using kurtosis as the projection index in projection pursuit allows the ``blind'' estimation of the projection direction $\boldsymbol{\theta} := \boldsymbol{\Sigma}^{-1} (\boldsymbol{\mu}_2 - \boldsymbol{\mu}_1)$ that is used in linear discriminant analysis to construct the optimal Bayes classifier, in the full absence of any label information. In other words, projection pursuit essentially allows conducting LDA in a blind fashion to recover the subspace that optimally separates the two groups. Afterwards, various clustering methods can then be applied to the projected data to conduct efficient clustering.

While very interesting, the result of \cite{pena2001cluster} raises a natural question regarding the efficiency of the procedure. Namely, how much does one lose by not knowing the labels and relying on projection pursuit compared to using LDA to recover the same direction $\boldsymbol{\theta}$ when the group memberships are known? This is the main question we study in the current paper, working, for simplicity, under the assumption of two-group normal mixtures. Our approach is asymptotic in nature and we perform the comparison through the limiting covariance matrices of the estimators in question. In particular, we show that the limiting covariance matrices of projection pursuit and LDA are proportional, allowing us to conduct the comparisons simply through the corresponding constants of proportionality.

Asymptotic results for general projection indices have been derived earlier in the context of \textit{independent component analysis} (ICA), see, e.g., \cite{ollila2009deflation,dermoune2013fastica,miettinen2015fourth,virta2016projection}. In ICA, one assumes that the observed random $p$-vector $\textbf{x}$ is an \textit{independent component model}, i.e., there exists a full rank $p \times p$ matrix $\boldsymbol{\Gamma}$ such that $\boldsymbol{\Gamma} \{ \textbf{x} - \mathrm{E}(\textbf{x}) \}$ has independent components. The IC model is a rather wide family of distributions and, in particular, contains our model of choice, the multivariate normal mixture (this, apparently novel, result is given as Lemma \ref{lem:ic_mixture_intersection} in the supplementary material). This connection implies that the results of the current paper are intimately related to \cite{virta2016projection} who considered (in the context of ICA) the same projection indices as we do here. However, we remark that our contributions surpass those of \cite{virta2016projection} in two critical regards: 1) \cite{virta2016projection} derived only the asymptotic variances of the ICA parameters (ignoring their covariances), whereas we give the full limiting distribution of the estimated projection direction. Besides completing the asymptotic story, knowledge of the full distribution is crucial with respect to the comparison of PP and LDA as it reveals that the limiting covariance matrix of PP is exactly proportional to the limiting covariance of LDA, see Theorems \ref{theo:asnorm_unblind}--\ref{theo:asnorm_blind_3} later on. 2)  From a technical viewpoint, the derivation of the convergence rates of the estimators was in Virta et al. (2016) left implicit and our proofs provide a rigorous treatment of this. In particular, to guarantee well-defined Taylor expansions of the objective functions, we need to establish the almost sure convergence of the projection pursuit estimates and, as far as we are aware, such results have not been given previously either in ICA or PP-literature.

While kurtosis is the most popular choice for the projection index in projection pursuit, also several alternatives are commonly used. In particular, skewness is a somewhat standard choice, see, for example, \cite{loperfido2018skewness}, and was shown in \cite{loperfido2013skewness} to have the same property of being able to find the optimal projection direction without the label information as possessed by kurtosis. As such, we study also skewness-based projection pursuit in the current work. However, as shown by \cite{pena2001cluster, loperfido2013skewness}, for both kurtosis and skewness there exist particular values of the mixing proportion under which the two indices are unable to recover the optimal projection direction (for example, skewness fails to produce a consistent estimate of $\boldsymbol{\theta}$ when the two groups have equal proportions). These drawbacks can be mitigated by combining both cumulants into a single projection index, in a form of a weighted linear combination. Our results then show that, with a proper choice of weighting, a rather efficient blind competitor for the LDA-based unblind estimator can be obtained. Indeed, in the extreme case where the groups are arbitrarily well-separated and their sizes are reasonably balanced, projection pursuit is able to achieve equal efficiency with LDA. As remarked in the previous paragraph, the asymptotic properties of the hybrid index have been studied also earlier, in the context of independent component analysis, in \cite{virta2016projection}.

We note that despite the theoretical guarantees of projection pursuit, the most common blind method for revealing clusters is still arguably PCA, see, e.g., \cite{jolliffe2002principal}. However, it is also well known that PCA does not, in general, yield a consistent estimator of the linear discriminant direction. A standard example demonstrating this is the extreme case where the within-group covariance matrix $\boldsymbol{\Sigma}$ is heavily concentrated on a direction orthogonal to the difference of the group means $\boldsymbol{\mu}_2 - \boldsymbol{\mu}_1$. In such a case, the projections of the two group means onto the first principal component direction overlap, making clustering based on the direction impossible. Hence, due to its unreliability in estimating the linear discriminant, PCA cannot really be seen as a blind estimator of the separating direction and, as such, we do not include it in the comparisons in the current paper. However, we have still included, for completeness, equivalent asymptotic results for PCA as we state for the other methods, and these are given in Appendix \ref{sec:PCA}.

The rest of the manuscript is organized as follows. In Section 2 we derive the asymptotic behavior of three estimators of the linear discriminant direction: LDA and kurtosis- and skewness based projection pursuit. A short comparison of the results is also presented. In Section \ref{sec:combination}, we give the corresponding results for projection pursuit based on a weighted combination of skewness and kurtosis and conduct a more extensive set of asymptotic comparisons between all considered methods. Simulation studies exploring both the finite-sample performance of the methods and the applicability of our asymptotic results to practice are given in Section \ref{sec:simulations}, while the performance and the applicability of presented methods to a real data example, as well as the comparison to the PCA are given in Section~\ref{sec:example}. Finally, we conclude with some discussion in Section \ref{sec:discussion}. All proofs of the technical results are postponed to Appendix
\ref{sec:proofs}.

\section{Estimation of the linear discriminant}\label{sec:estimation}

Let $(\Omega, \mathcal{F}, \mathbb{P}) $ be a probability space. Throughout the following, we assume that the $(p + 1)$-dimensional pair $(\textbf{x}, y)$ obeys the following model:
\begin{align}\label{eq:xy_model}
    y  \sim \mathrm{Ber}(\alpha_1) \quad \mbox{and} \quad \textbf{x} \mid y \sim \mathcal{N}_p \{ y  \boldsymbol{\mu}_1 + (1 - y) \boldsymbol{\mu}_2, \boldsymbol{\Sigma} \},
\end{align}
for $0 < \alpha_1 < 1$, $\boldsymbol{\mu}_1, \boldsymbol{\mu}_2 \in \mathbb{R}^p$, $\boldsymbol{\mu}_1 \neq  \boldsymbol{\mu}_2$, and a full rank $\boldsymbol{\Sigma} \in \mathbb{R}^{p \times p} $. The marginal distribution of $\textbf{x}$ is then the multivariate normal mixture,
\begin{align*}
    \textbf{x} \sim \alpha_1 \mathcal{N}_p(\boldsymbol{\mu}_1, \boldsymbol{\Sigma}) + \alpha_2 \mathcal{N}_p(\boldsymbol{\mu}_2, \boldsymbol{\Sigma}),
\end{align*}
where $\alpha_2 := 1 - \alpha_1$. Under model \eqref{eq:xy_model}, the classification of $\textbf{x}$ is usually based on its projection onto the linear discriminant direction  $\boldsymbol{\theta} = \boldsymbol{\Sigma}^{-1} (\boldsymbol{\mu}_2 - \boldsymbol{\mu}_1) $. This projection direction is optimal in the sense that the optimal Bayes classifier (having the minimal misclassification rate out of all classifiers) depends on the data only through the projection $\boldsymbol{\theta}'\textbf{x}$, see, e.g., \cite{mardia1979multivariate}.

Our objective throughout the paper is the estimation of the standardized projection direction $ \boldsymbol{\theta}/\| \boldsymbol{\theta} \|$ (the scale of the projection direction is irrelevant, meaning that the unit length constraint is without loss of generality). As described in Section \ref{sec:introduction}, we will consider two types of estimators, blind ones which use only the random vector $\textbf{x}$ (a sample from its distribution) in the estimation, and an unblind one which bases the estimation on the full pair $(\textbf{x}, y)$. The unblind method is allowed more information in the estimation and is, naturally, expected to provide a more efficient estimator, a fact that is verified by our comparisons later on.

\subsection{Unblind estimation of the linear discriminant}

If we have a sample $(\textbf{x}_1, y_1), \ldots , (\textbf{x}_n, y_n) $ from the distribution of the full pair $(\textbf{x}, y)$ available, the standard estimator of $\boldsymbol{\Sigma}^{-1} (\boldsymbol{\mu}_2 - \boldsymbol{\mu}_1)$ is the plug-in estimator (which is also its MLE, up to the scaling of the pooled covariance matrix) used in standard LDA. That is, using the notation,
$$
\bar{\textbf{x}}_{n1} := \frac{1}{\sum_{i=1}^n y_i} \sum_{i=1}^n y_i \textbf{x}_i, \quad \bar{\textbf{x}}_{n2} := \frac{1}{\sum_{i=1}^n ( 1 - y_i ) } \sum_{i=1}^n (1 - y_i) \textbf{x}_i,
$$
$$
\textbf{S}_n := \frac{1}{n - 2} \left\{ \sum_{i=1}^n y_i (\textbf{x}_i - \bar{\textbf{x}}_1) (\textbf{x}_i - \bar{\textbf{x}}_1)' + \sum_{i=1}^n (1 - y_i) (\textbf{x}_i - \bar{\textbf{x}}_2) (\textbf{x}_i - \bar{\textbf{x}}_2)' \right\},
$$
we consider the estimator,
\begin{align*}
    \textbf{w}_n := \textbf{S}_n^{-1}(\bar{\textbf{x}}_{n2} - \bar{\textbf{x}}_{n1}).
\end{align*}
Asymptotic results for LDA are very standard in the literature, see for example \cite{anderson2003introduction}. However, these results are usually given in the case of fixed group sizes, whereas in our model the group sizes are determined by the indicator variables $y_1, \ldots, y_n$ and are, as such, random. Hence, as far as we know, the following theorem is, if not particularly groundbreaking in its conclusions, a novel one.

\begin{theorem}\label{theo:asnorm_unblind}
    Under model \eqref{eq:xy_model}, we have, as $n \rightarrow \infty$,
    \begin{align*}
        \sqrt{n}(\textbf{w}_n/\| \textbf{w}_n \| - \boldsymbol{\theta}/\| \boldsymbol{\theta} \|) \rightsquigarrow \mathcal{N}_p(\textbf{0}, \boldsymbol{\Psi}_U),
    \end{align*}
    where
    \begin{align*}
        \boldsymbol{\Psi}_U := \left( \frac{1 + \beta \tau}{\| \boldsymbol{\theta} \|^2 \beta} \right)  \left( \textbf{I}_p - \frac{\boldsymbol{\theta} \boldsymbol{\theta}'}{\| \boldsymbol{\theta} \|^2} \right) \boldsymbol{\Sigma}^{-1} \left( \textbf{I}_p - \frac{\boldsymbol{\theta} \boldsymbol{\theta}'}{\| \boldsymbol{\theta} \|^2} \right),
    \end{align*}
    $\beta := \alpha_1 \alpha_2$ and $\tau := (\boldsymbol{\mu}_2 - \boldsymbol{\mu}_1)' \boldsymbol{\Sigma}^{-1} (\boldsymbol{\mu}_2 - \boldsymbol{\mu}_1)$.
\end{theorem}

The form of the limiting covariance matrix in Theorem \ref{theo:asnorm_unblind} is rather simple and inspection of the proof of the result reveals that the involved projection matrices onto the orthogonal complement of the direction $\boldsymbol{\theta}/\| \boldsymbol{\theta} \|$ are simply consequences of the standardization of the estimator to unit length. Note also that the scalar factor in front can be written as $ 1/(\| \boldsymbol{\theta} \|^2 \beta) + (\boldsymbol{\theta}/\| \boldsymbol{\theta} \|)' \boldsymbol{\Sigma} (\boldsymbol{\theta}/\| \boldsymbol{\theta} \|) $, the two summands of which have the following rough interpretations: If the groups are imbalanced, $\beta$ is small, making the first summand large and inflating the asymptotic variance. Similarly, if the data exhibit a large amount of variation in the direction of the optimal discriminant direction, i.e., $(\boldsymbol{\theta}/\| \boldsymbol{\theta} \|)' \boldsymbol{\Sigma} (\boldsymbol{\theta}/\| \boldsymbol{\theta} \|)$ is large, the second term increases the magnitude of the asymptotic variance.

\subsection{Blind estimation of the linear discriminant}

\subsubsection*{Kurtosis-based projection pursuit}

Let $\delta_1 := 1/2 - 1/\sqrt{12}$, $\delta_2 := 1/2 + 1/\sqrt{12}$ and $\tilde{\textbf{x}} := \textbf{x} - \mathrm{E}(\textbf{x}) $. The kurtosis $\kappa: \mathbb{S}^{p-1} \to \mathbb{R}$ of the projection of $\textbf{x}$ on a given direction $\textbf{u} \in \mathbb{S}^{p-1}$ is then defined as,
\begin{align*}
    \kappa(\textbf{u}) = \frac{ \mathrm{E} \{ ( \textbf{u}' \tilde{\textbf{x}} )^4 \} }{ [ \mathrm{E} \{ ( \textbf{u}' \tilde{\textbf{x}} )^2 \} ]^2 }.
\end{align*}
The fact that projection pursuit based on kurtosis is Fisher consistent for the linear discriminant under normal mixtures was first shown in \cite[Corollary~2]{pena2001cluster}. However, the successful use of their result in practice requires knowing something about the mixing proportion $\alpha_1$. Namely, if $\alpha_1 \in ( \delta_1, \delta_2 )$ then the linear discriminant $\boldsymbol{\theta}/\|\boldsymbol{\theta}\|$ is found as the minimizer of $\kappa$, whereas if $\alpha_1 \in (0, \delta_1) \cup (\delta_2, 1)$ then  $\boldsymbol{\theta}/\|\boldsymbol{\theta}\|$ is found as the maximizer of $\kappa$. Naturally, as a workaround, one could in practice always search for both the minimizer and the maximizer of $\kappa$ but, even in this case, it might be non-trivial to recognize the linear discriminant amongst the two. Thus, to obtain a truly blind estimator, we propose instead using the squared \textit{excess kurtosis} $ \{ \kappa(\textbf{u}) - 3 \}^2 $ as an objective function. Indeed, the next lemma reveals that the squared excess kurtosis yields a Fisher consistent estimate of the linear discriminant, apart from the degenerate cases $\alpha_1 \in \{ \delta_1, \delta_2 \}$ where excess kurtosis vanishes, without the need to choose between minimization and maximization.

\begin{lemma}\label{lem:kurtosis_fisher}
    Given model \eqref{eq:xy_model},
    \begin{itemize}
        \item[1)] if $ \alpha_1 \notin \{ \delta_1, \delta_2 \} $, then the function $\textbf{u} \mapsto \{ \kappa(\textbf{u}) - 3 \}^2$ is uniquely maximized by $\pm \boldsymbol{\theta}/\| \boldsymbol{\theta} \|$,
        \item[2)] if $ \alpha_1 \in \{ \delta_1, \delta_2 \} $, then $\{ \kappa(\textbf{u}) - 3 \}^2 = 0$ for all $\textbf{u} \in \mathbb{S}^{p - 1}$.
    \end{itemize}
\end{lemma}

Moving next to study the asymptotic properties of $\kappa$, let $\textbf{x}_1, \ldots, \textbf{x}_n$ be a random sample from the marginal distribution of $\textbf{x}$ in the model \eqref{eq:xy_model}. The sample counterpart of $\kappa$ is
\begin{align*}
    \kappa_{n}: \mathbb{S}^{p - 1} \to \mathbb{R}, \quad \kappa_{n} (\textbf{u}) = \frac{ (1/n) \sum_{i=1}^n ( \textbf{u}' \tilde{\textbf{x}}_i )^4 }{ \{ (1/n) \sum_{i=1}^n ( \textbf{u}' \tilde{\textbf{x}}_i )^2  \}^2 },
\end{align*}
where $\tilde{\textbf{x}}_i := \textbf{x}_i - \bar{\textbf{x}}$. If $n \geq p $ the denominator of the random function $\kappa_n$ is a.s. positive, making $\kappa_n$ well-defined and an estimator for $\boldsymbol{\theta}/\| \boldsymbol{\theta} \|$ is then obtained as any maximizer of $ \textbf{u} \mapsto \{ \kappa_{n}(\textbf{u}) - 3 \}^2$ (note that if $n \geq p$, a maximizer exists almost surely due to the compacity of $\mathbb{S}^{p-1}$). The following theorem shows that any sequence of such maximizers has a limiting normal distribution. Note that the need to include the ``corrective'' signs $s_n$ in Theorem \ref{theo:asnorm_blind} stems from the sign-invariance of the objective function (which also causes the existence of two maximizers in Lemma \ref{lem:kurtosis_fisher}).

\begin{theorem}\label{theo:asnorm_blind}
    Given model \eqref{eq:xy_model}, assume that $ \alpha_1 \notin \{ \delta_1, \delta_2 \} $ and let $\textbf{u}_n$ be any sequence of maximizers of $\textbf{u} \mapsto \{ \kappa_n(\textbf{u}) - 3 \}^2$. Then, there exists a sequence of signs $s_n \in \{ -1 , 1 \}$ such that, as $n \rightarrow \infty$,
    \begin{itemize}
        \item[1)] $s_n \textbf{u}_n \rightarrow \boldsymbol{\theta}/\| \boldsymbol{\theta} \|$, almost surely.
        \item[2)] $ \sqrt{n}(s_n \textbf{u}_n - \boldsymbol{\theta}/\| \boldsymbol{\theta} \|) \rightsquigarrow \mathcal{N}_p(\textbf{0}, \boldsymbol{\Psi}_{\kappa}) $, where
        \begin{align*}
            \boldsymbol{\Psi}_{\kappa} := C_\kappa \left( \frac{1 + \beta \tau}{\| \boldsymbol{\theta} \|^2 \beta} \right) \left( \textbf{I}_p - \frac{\boldsymbol{\theta} \boldsymbol{\theta}'}{\| \boldsymbol{\theta} \|^2} \right) \boldsymbol{\Sigma}^{-1} \left( \textbf{I}_p - \frac{\boldsymbol{\theta} \boldsymbol{\theta}'}{\| \boldsymbol{\theta} \|^2} \right),
        \end{align*}
        and
        \begin{align*}
            C_\kappa := \frac{6 + 24 \beta \tau + 9 \beta (1 - 2 \beta) \tau^2 + \beta (1 - 3 \beta) \tau^3}{\beta \tau^3 (1 - 6 \beta)^2},
        \end{align*}
        with $ \beta $ and $ \tau $ as in Theorem \ref{theo:asnorm_unblind}.
    \end{itemize}
\end{theorem}

The limiting covariance matrices in Theorems \ref{theo:asnorm_unblind} and \ref{theo:asnorm_blind} are proportional, the only difference being the factor $C_\kappa$. This makes their comparisons in Subsection \ref{subsec:comparison} particularly straightforward. However, even without the formal comparisons, it is evident that the kurtosis-based estimator has a clear flaw in that it fails to be consistent for the mixing proportions $ \alpha_1 \in \{ \delta_1, \delta_2 \}$ (for these values of $\alpha_1$, we have $1-6 \beta = 0$ in the denominator of $C_\kappa$ in Theorem \ref{theo:asnorm_blind}). And even though these are only two points in the continuum $(0, 1)$, the continuity of $C_\kappa$ in $\alpha_1$ outside of these points implies that the estimator is highly inefficient for values of $\alpha_1$ near $ \delta_1 $ or $ \delta_2 $. Hence, we will next discuss an alternative estimator that is consistent when $ \alpha_1 \in \{ \delta_1, \delta_2 \} $ (at the price of lacking consistency in another point).

\subsubsection*{Skewness-based projection pursuit}

To complement the kurtosis-based projection pursuit, we next consider skewness-based projection pursuit. Note that, despite its dependency on lower moments, this form of PP is less studied in the literature (see the references in Section \ref{sec:introduction}).

The skewness of the projection of $\textbf{x}$ on a given direction $\textbf{u} \in \mathbb{S}^{p-1}$ is measured by the objective function $\gamma: \mathbb{S}^{p-1} \to \mathbb{R}$ defined as,
\begin{align*}
    \gamma(\textbf{u}) = \frac{ \mathrm{E} \{ ( \textbf{u}' \tilde{\textbf{x}} )^3 \} }{ [ \mathrm{E} \{ ( \textbf{u}' \tilde{\textbf{x}} )^2 \} ]^{3/2} },
\end{align*}
where again $\tilde{\textbf{x}} = \textbf{x} - \mathrm{E}(\textbf{x}) $. Similarly to kurtosis, also with skewness it is more convenient to work with its squared value. The next lemma presents the Fisher consistency of the corresponding estimator and reveals that the mixing proportion $1/2$ plays the role of the proportions $\delta_1, \delta_2$ for skewness. The reason for this is intuitively clear as, under the choice $\alpha_1 = 1/2$, the normal mixture is perfectly symmetrical, explaining the vanishing of the skewness. The result appeared originally as Proposition 1 in \cite{loperfido2013skewness} but we give, for completeness, a proof in Appendix \ref{sec:proofs}.

\begin{lemma}\label{lem:skewness_fisher}
    Given model \eqref{eq:xy_model},
    \begin{itemize}
        \item[1)] if $ \alpha_1 \neq 1/2 $, then the function $\textbf{u} \mapsto \gamma(\textbf{u})^2$ is uniquely maximized by $\pm \boldsymbol{\theta}/\| \boldsymbol{\theta} \|$,
        \item[2)] if $ \alpha_1 = 1/2 $, then $\gamma(\textbf{u})^2 = 0$ for all $\textbf{u} \in \mathbb{S}^{p - 1}$.
    \end{itemize}
\end{lemma}

Finally, we derive in Theorem \ref{theo:asnorm_blind_2} below the strong consistency and limiting distribution of the corresponding sample estimator, obtained through the maximization of the square of the sample skewness, defined as,
\begin{align*}
    \gamma_{n}: \mathbb{S}^{p - 1} \to \mathbb{R}, \quad \gamma_{n} (\textbf{u}) = \frac{ (1/n) \sum_{i=1}^n ( \textbf{u}' \tilde{\textbf{x}}_i )^3 }{ \{ (1/n) \sum_{i=1}^n ( \textbf{u}' \tilde{\textbf{x}}_i )^2  \}^{3/2} }.
\end{align*}

\begin{theorem}\label{theo:asnorm_blind_2}
    Given model \eqref{eq:xy_model}, assume that $\alpha_1 \neq 1/2$ and let $\textbf{u}_n$ be any sequence of maximizers of $\textbf{u} \mapsto \gamma_n(\textbf{u})^2$. Then, there exists a sequence of signs $s_n \in \{ -1 , 1 \}$ such that, as $n \rightarrow \infty$,
    \begin{itemize}
        \item[1)] $s_n \textbf{u}_n \rightarrow \boldsymbol{\theta}/\| \boldsymbol{\theta} \|$, almost surely.
        \item[2)] $ \sqrt{n}(s_n \textbf{u}_n - \boldsymbol{\theta}/\| \boldsymbol{\theta} \|) \rightsquigarrow \mathcal{N}_p(\textbf{0}, \boldsymbol{\Psi}_{\gamma}) $, where
        \begin{align*}
            \boldsymbol{\Psi}_{\gamma} := C_\gamma \left( \frac{1 + \beta \tau}{\| \boldsymbol{\theta} \|^2 \beta} \right) \left( \textbf{I}_p - \frac{\boldsymbol{\theta} \boldsymbol{\theta}'}{\| \boldsymbol{\theta} \|^2} \right) \boldsymbol{\Sigma}^{-1} \left( \textbf{I}_p - \frac{\boldsymbol{\theta} \boldsymbol{\theta}'}{\| \boldsymbol{\theta} \|^2} \right),
        \end{align*}
        and
        \begin{align*}
            C_\gamma := \frac{ 2 + 6 \beta \tau + \beta \tau^2}{\beta \tau^2 (1 - 4 \beta)}
        \end{align*}
        with $\beta, \tau$ as in Theorem \ref{theo:asnorm_blind}.
    \end{itemize}
\end{theorem}

Interestingly, also the limiting covariance of the skewness-based estimator is proportional to that of LDA, meaning that the main object of interest in the result is the factor $C_\gamma$. These factors will be compared in the next section to make statements about the relative efficiencies of the estimators under various scenarios.

\subsection{Asymptotic comparison of the three estimators}\label{subsec:comparison}

Theorems \ref{theo:asnorm_unblind}, \ref{theo:asnorm_blind} and \ref{theo:asnorm_blind_2} show that the limiting distributions of the unblind and blind estimators of $\boldsymbol{\theta}/\| \boldsymbol{\theta} \|$ all have proportional covariance matrices. Thus, their efficiencies may be compared simply through the corresponding constants of proportionality which depend on the problem parameters only through the mixing proportion ($\beta = \alpha_1 \alpha_2$) and the degree of separation between the two groups, as measured by the (squared) Mahalanobis distance $\tau$. The relative asymptotic efficiencies (pair-wise ratios of the constants) of the blind estimators vs. the unblind estimator (LDA) are simply $C_\kappa^{-1}$ and $C_\gamma^{-1}$ where the values of the constants are given in Theorems \ref{theo:asnorm_blind} and \ref{theo:asnorm_blind_2}. Especially the former expression is somewhat complicated for arbitrary $\beta$ and $\tau$ but both simplify greatly if we consider the case where the Mahalanobis distance $\tau$ is large. That is, letting $\tau \rightarrow \infty$, the relative asymptotic efficiencies are simply
\begin{align*}
    \mathrm{Eff}_\kappa = \frac{(1 - 6 \beta)^2}{1 - 3 \beta} \quad \mbox{and} \quad \mathrm{Eff}_\gamma = 1 - 4 \beta.
\end{align*}

\begin{figure}[ht]
\centering
    \includegraphics[width = 1\textwidth]{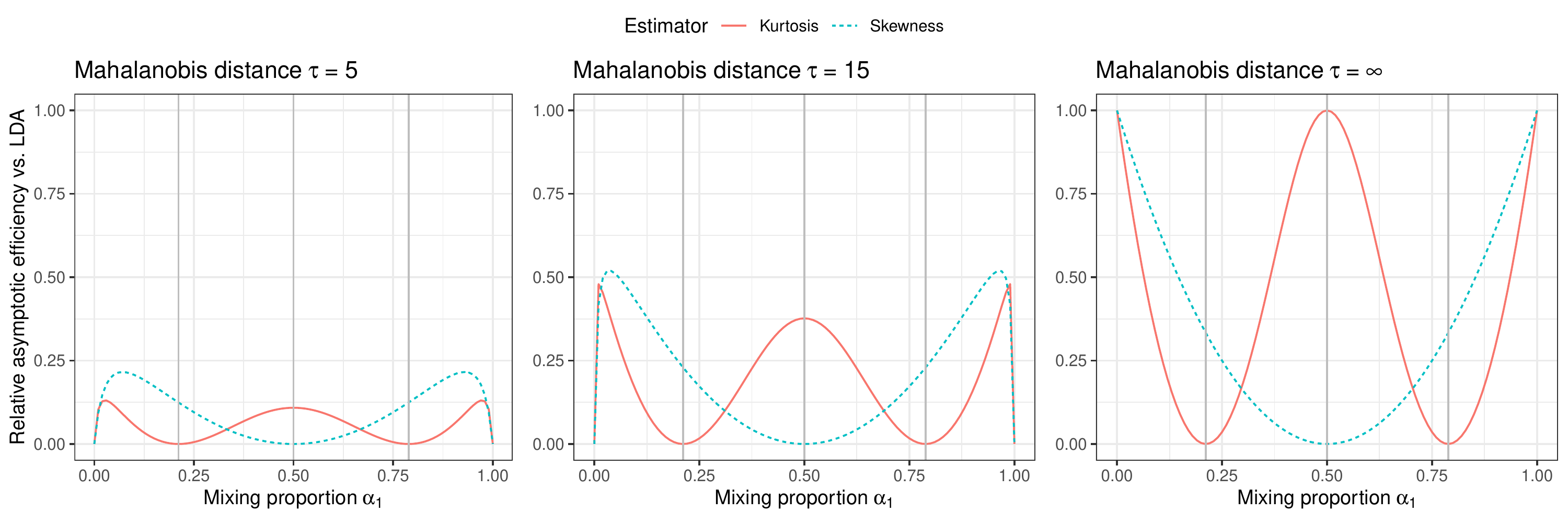}
    \caption{ The relative asymptotic efficiency of the blind vs. unblind estimation of $\boldsymbol{\theta}/\| \boldsymbol{\theta} \|$ under different Mahalanobis distances $\tau$ between the two group means. 
 The three gray vertical reference lines mark those values of $\alpha_1$ where kurtosis ($1/2 \pm 1/\sqrt{12}$) or skewness ($1/2$) of all projections is a constant.}
    \label{fig:as_rel_eff}
\end{figure}
Figure~\ref{fig:as_rel_eff} plots the relative efficiencies as a function of the first mixing proportion $\alpha_1$ in the cases $\tau = 5, 15$ and $\tau \rightarrow \infty$. The plots verify that, for any practical value of the Mahalanobis distance, LDA is always asymptotically highly superior to both blind methods. However, in the extreme case where the two groups are well-separated to an arbitrarily large degree, we see, in particular, that the kurtosis estimator is asymptotically equally efficient to LDA in the balanced case $\alpha_1 = 1/2$, and also in the limits $\alpha_1 \rightarrow 0$ and $\alpha_1 \rightarrow 1$ (although, in these cases the actual asymptotic covariance matrices themselves grow without bounds).

\begin{figure}[ht]
\centering
    \includegraphics[width=0.6\textwidth]{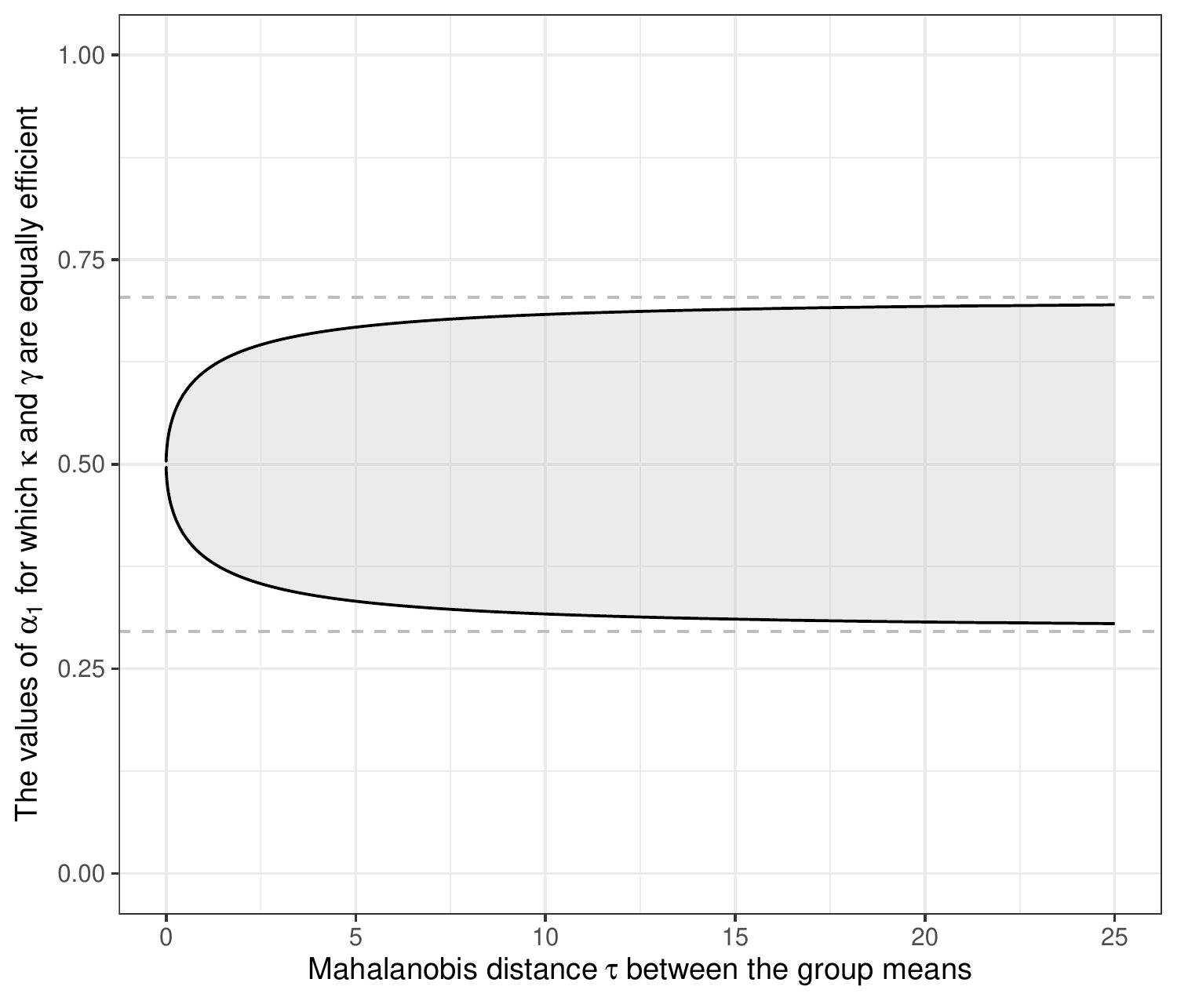}
    \caption{The two solid curves trace the values of the mixing proportion $\alpha_1$ ($y$-axis) for which $\kappa$ and $\gamma$ are asymptotically equally efficient as a function of the Mahalanobis distance between the two group means ($x$-axis). The shaded region represents the values of $\alpha_1$ for which $\kappa$ is the superior choice over $\gamma$. The two horizontal dashed lines indicate the limits $ 1/2 \pm 1/\sqrt{24}$ of the two curves.}
    \label{fig:as_rel_same}
\end{figure}

Figure~\ref{fig:as_rel_eff} also shows that, depending on the Mahalanobis distance, around the point $\alpha_1 \approx 0.30$ there is a mixing proportion for which $\kappa$ and $\gamma$ are asymptotically equally efficient. Figure~\ref{fig:as_rel_same} plots these proportions (and their mirror images on the upper half of the region $(0, 1)$) as a function of the Mahalanobis distance. The plot reveals that the region of mixing proportions for which $\kappa$ is asymptotically superior choice to $\gamma$ (the gray inner region) gets wider as the groups get more well-separated, finally approaching the region $ 1/2 \pm 1/\sqrt{24} $ in the limit $\tau \rightarrow \infty$ (the two horizontal dashed lines).



\section{Convex combination of skewness and kurtosis}\label{sec:combination}

\subsection{Theoretical properties}

Based on Figure~\ref{fig:as_rel_eff}, the objective functions $\kappa$ and $\gamma$ produce, especially for well-separated groups, fairly efficient estimators of the linear discriminant in the absence of any grouping information. However, for this, it is crucial to have at least an approximate idea of the mixing proportion $\alpha_1$ in order to choose the more efficient of the two objective functions and to avoid the points where a particular estimator becomes completely inefficient (for example, if the groups are close to being balanced, one wants to use kurtosis as skewness contains no information when $\gamma = 1/2$, see Lemma \ref{lem:skewness_fisher} and Figure~\ref{fig:as_rel_eff}). As the mixing proportion is rarely known in practice, this makes the procedure difficult to implement.

A natural way to overcome this weakness is to, instead of choosing between $\gamma$ and $\kappa$, use them both simultaneously, through a objective function $\eta:\mathbb{S}^{p - 1} \to \mathbb{R}$ that is a convex combination of the two squared cumulants,
\begin{align*}
     \eta(\textbf{u})=\eta(\textbf{u};w_1) := w_1 \gamma(\textbf{u})^2 + w_2 \{ \kappa(\textbf{u}) - 3 \}^2,
\end{align*}
where $w_1, w_2 \geq 0$, $w_1 + w_2 = 1$. Naturally, the cases $w_1 = 0$ and $w_2 = 0$ simply correspond to the two individual objective functions and, hence, we will in the following assume that $w_1, w_2 > 0$. The next result, following straightforwardly from Lemmas \ref{lem:kurtosis_fisher} and \ref{lem:skewness_fisher}, shows that $\eta$ indeed allows the completely blind recovery of the linear discriminant regardless of the mixing proportion $\alpha_1 \in (0, 1)$.

\begin{lemma}\label{lem:convex_fisher}
    Given model \eqref{eq:xy_model}, $\eta$ is uniquely maximized by $ \pm \boldsymbol{\theta}/\| \boldsymbol{\theta} \|$.
\end{lemma}

The sample version of the hybrid objective function is $\eta_n: \mathbb{S}^{p - 1} \to \mathbb{R}$, defined as $\eta_n(\textbf{u})=\eta_n(\textbf{u}; w_1):= w_1 \gamma_n(\textbf{u})^2 + w_2  \{ \kappa_n(\textbf{u}) - 3 \}^2$, and we next give the limiting behavior of its maximizer. Unsurprisingly, the resulting limiting covariance matrix is up to a multiplicative constant equal to the previous ones.

\begin{theorem}\label{theo:asnorm_blind_3}
    Given model \eqref{eq:xy_model}, let $\textbf{u}_n$ be any sequence of maximizers of $\eta_n$. Then, there exists a sequence of signs $s_n \in \{ -1 , 1 \}$ such that, as $n \rightarrow \infty$,
    \begin{itemize}
        \item[1)] $s_n \textbf{u}_n \rightarrow \boldsymbol{\theta}/\| \boldsymbol{\theta} \|$, almost surely.
        \item[2)] $ \sqrt{n}(s_n \textbf{u}_n - \boldsymbol{\theta}/\| \boldsymbol{\theta} \|) \rightsquigarrow \mathcal{N}_p(\textbf{0}, \boldsymbol{\Psi}_{\eta}) $, where
        \begin{align*}
            \boldsymbol{\Psi}_{\eta} := C_\eta \left( \frac{1 + \beta \tau}{\| \boldsymbol{\theta} \|^2 \beta} \right) \left( \textbf{I}_p - \frac{\boldsymbol{\theta} \boldsymbol{\theta}'}{\| \boldsymbol{\theta} \|^2} \right) \boldsymbol{\Sigma}^{-1} \left( \textbf{I}_p - \frac{\boldsymbol{\theta} \boldsymbol{\theta}'}{\| \boldsymbol{\theta} \|^2} \right),
        \end{align*}
        and
        \begin{align*}
        C_\eta =& \frac{(1+\beta\tau)(1-4\beta)\left( 9 w_1^2 (1 + \beta \tau)(2 + 6 \beta \tau + \beta \tau^2) + 24 w_1 w_2 \tau^{2}\beta (1 - 6 \beta) (6 + \tau)\right)}{ \beta \tau^{2} \{ 3 w_1 (1 + \beta \tau) (1 - 4\beta) + 4 w_2 \tau (1 - 6 \beta)^2 \}^2 }\\
        &+ \frac{16 w_2^2 \tau (1 - 6 \beta)^2 \Delta }{ \beta \tau^{2} \{ 3 w_1 (1 + \beta \tau) (1 - 4\beta) + 4 w_2 \tau (1 - 6 \beta)^2 \}^2 },
    \end{align*}
        where $\Delta := 6 + 24 \beta \tau + 9 \beta (1 - 2 \beta) \tau^2 + \beta (1 - 3 \beta) \tau^3 $ and $\beta, \tau$ are as in Theorem \ref{theo:asnorm_blind}.
    \end{itemize}
\end{theorem}

The constant of proportionality $C_\eta$ in Theorem \ref{theo:asnorm_blind_3} is again rather complicated but simplifies in the limit $\tau \rightarrow \infty$ to the more manageable, if not intuitive, form,
\begin{align*}
\frac{9 w_1^2 \beta^2 (1 - 4\beta) + 24 w_1 w_2 \beta (1 - 4\beta) (1 - 6 \beta) + 16 w_2^2 (1 - 6 \beta)^2 (1 - 3 \beta) }{ \{ 3 w_1 \beta (1 - 4\beta) + 4 w_2 (1 - 6 \beta)^2  \}^2 }.
\end{align*}

\subsection{Asymptotic comparisons}

We next investigate how the efficiency of the hybrid estimator compares to its competitors. Figure~\ref{fig:as_fig_3} shows the relative asymptotic efficiency of the hybrid estimator vs. LDA as a function of the mixing proportion $\alpha_1$ for the same values of $\tau$ as in Figure~\ref{fig:as_rel_eff} and for $w_1 \in \{ 0, 0.2, 0.4, 0.6, 0.8, 1 \}$. Note that the extreme cases $w_1 = 0$ and $w_1 = 1$ are equivalent to using the individual objective functions $\textbf{u} \mapsto \{ \kappa(\textbf{u}) - 3 \}^2$ and $\textbf{u} \mapsto \gamma(\textbf{u})^2$, respectively. The curves show somewhat erratic behavior around the points $\delta_1, \delta_2$ where kurtosis vanishes but otherwise seem to convey a clear message: inside the interval $(\delta_1, \delta_2)$ the hybrid estimator is, in general, a superior choice over the indivdual estimators, whereas outside of the interval the choice $w_1 = 1$ (corresponding to using skewness only) is preferable over the hybrid estimator.

\begin{figure}[ht]
\centering
    \includegraphics[width = 1\textwidth]{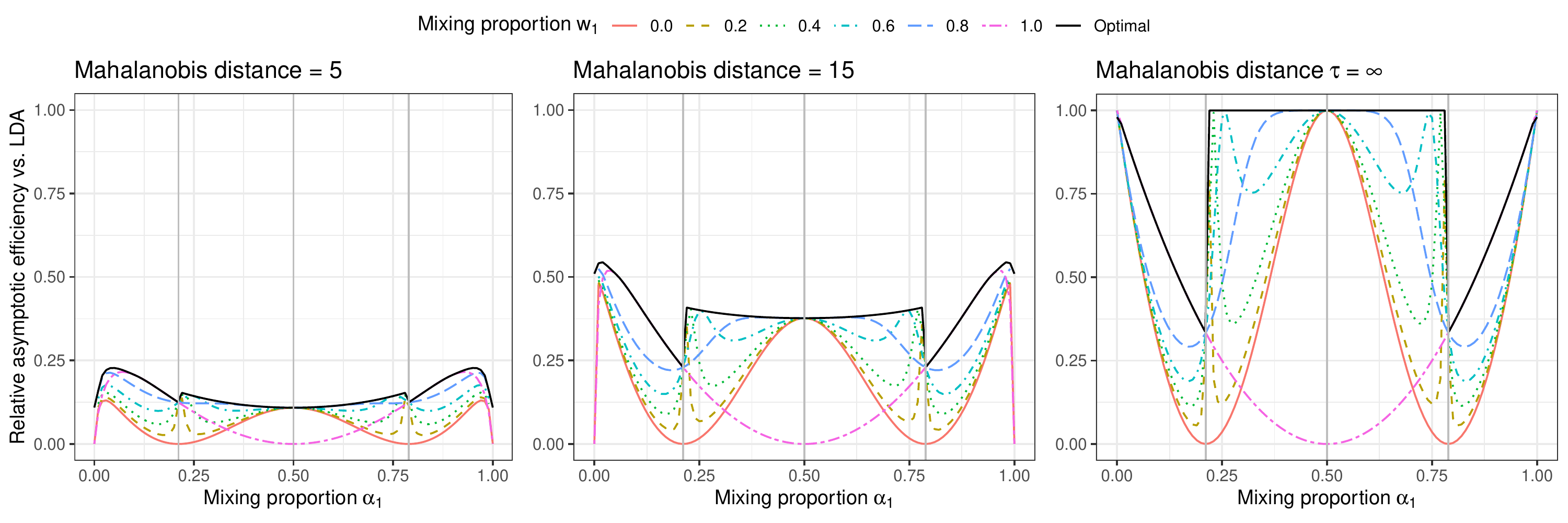}
    \caption{The relative asymptotic efficiency of the hybrid estimator vs. LDA under different Mahalanobis distances $\tau$ between the two group means. The colour and type of the lines denotes the weighting parameter $w_1$.
     The solid black line titled ``optimal'' traces the efficiencies obtained using the optimal choice of weighting for a given pair $(\alpha_1, \tau)$.}
    \label{fig:as_fig_3}
\end{figure}


To obtain a ``universal'' value of $w_1$ that yields (in some sense) on average the most efficient estimator over all $\alpha_1$, we compute with numerical integration the ``average efficiency'' $A(w_1, \tau)$ of the weight $w_1$ for a given value of $\tau$ as the area between the $x$-axis and the corresponding efficiency curve. For example $A(0, 15)$ is the area under the red solid curve in the middle panel of Figure~\ref{fig:as_fig_3}. Figure~\ref{fig:as_fig_4} then plots the weights $w_1$ yielding the maximal value of $A(w_1, \tau)$ as a function of the Mahalanobis distance $\tau$ and reveals that, regardless of the separation of the groups, one should optimally consider weights only in the range around $\tau \in (0.725, 0.825)$. This conclusion is rather predictable as kurtosis is based on a higher moment than skewness, meaning that the latter should be given a larger weight in order to obtain a ``balanced'' combination. As a further interesting observation, when $\tau \rightarrow 0$, i.e., when the mixture model approaches the multivariate normal model, the limit of the optimal weight seems to approach the value 0.8, which is the exact weighting used in the Jarque-Bera test statistic for testing normality, $(n/6) \gamma_n^2 + (n/24) (\kappa_n - 3)^2$ \citep{JarqueBera1980}. Moreover, the same weighting was also recommended by \cite{jones1987projection} as an approximation to an entropy-based index.

\begin{figure}[ht]
    \centering
    \begin{minipage}{0.45\textwidth}
  \centering
    \includegraphics[width=0.9\textwidth]{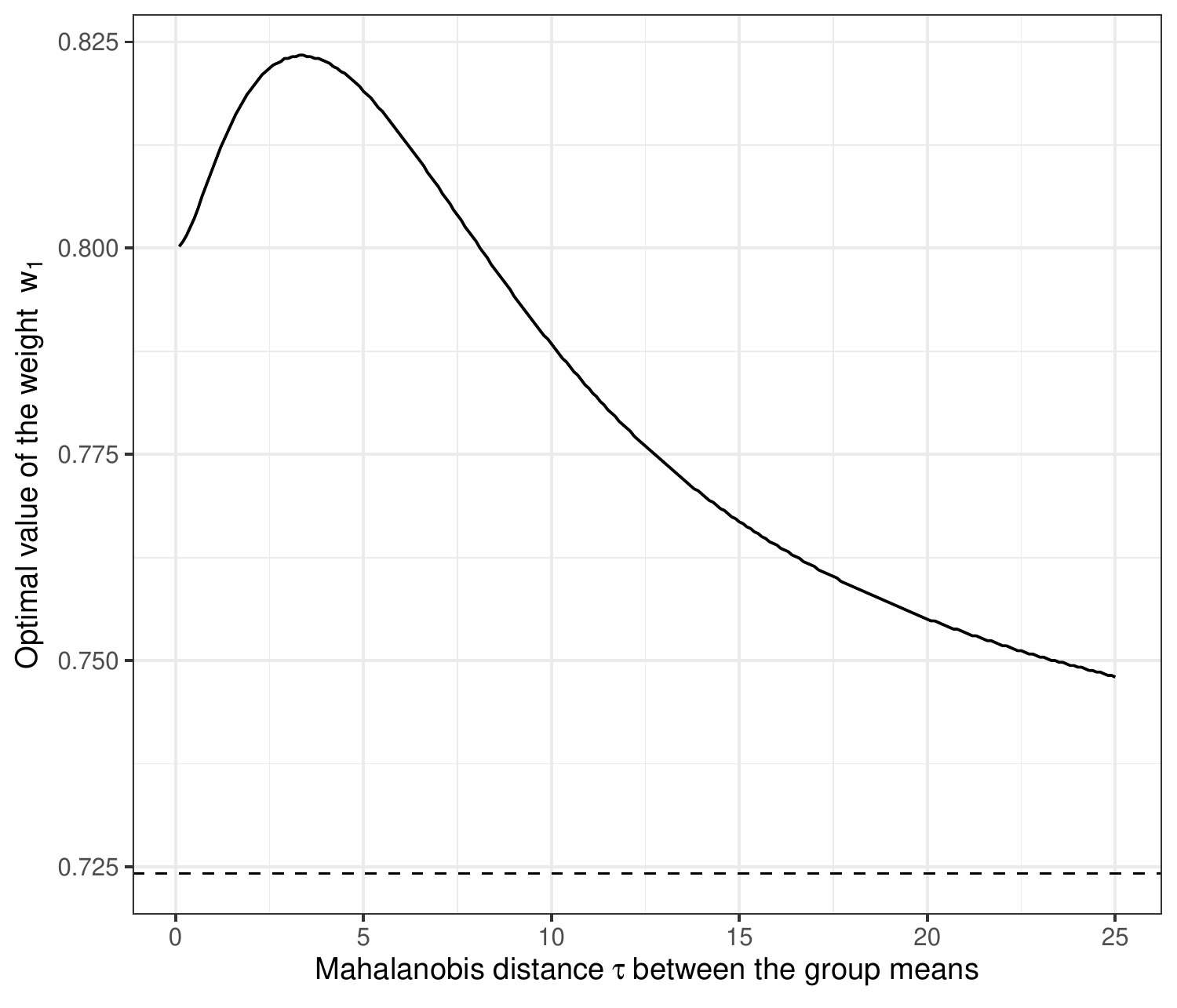}
    \caption{
    The weight $w_1$ yielding the maximal area under the corresponding efficiency curve as a function of the Mahalanonis distance $\tau$. The horizontal dashed line indicates the limit of the curve as $\tau \rightarrow \infty$ (approximately 0.7242).}
    \label{fig:as_fig_4}
    \end{minipage}\hfill
    \begin{minipage}{0.45\textwidth}
  \centering
    \includegraphics[width=1\textwidth]{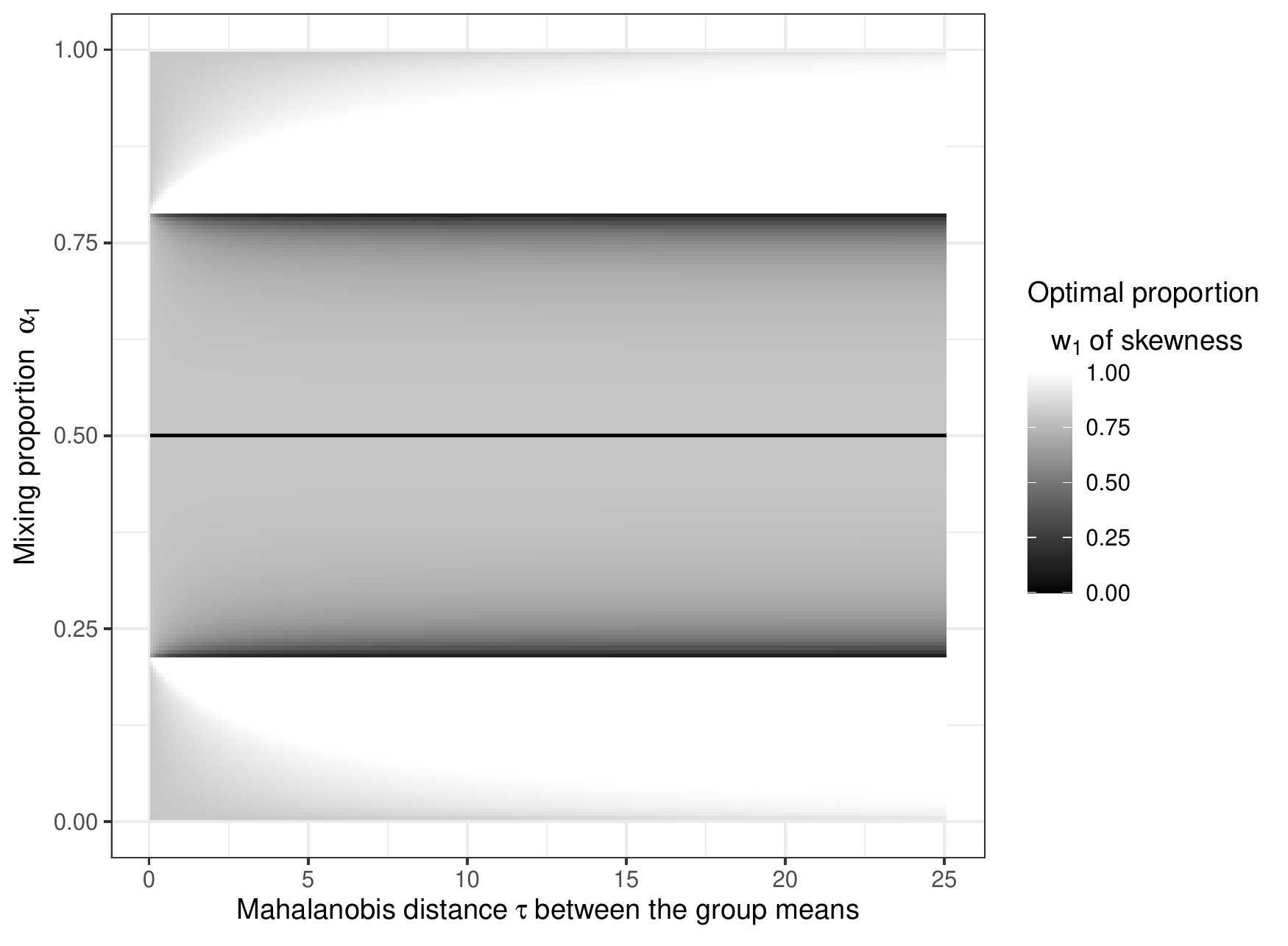}
    \caption{The heatmap shows as a function of $(\alpha_1, \tau)$ the weighting $w_1$ yielding the highest relative asymptotic efficiency compared to LDA. The three discontinuities correspond to the lines $\alpha_1 = \delta_1$, $\alpha_1 = 1/2$ and $\alpha_1 = \delta_2$.}
    \label{fig:as_fig_5}
    \end{minipage}
\end{figure}

Whereas Figure~\ref{fig:as_fig_4} aims to obtain a single universally useful value of $w_1$, in the optimal situation one would always use the particular weighting yielding the highest relative asymptotic efficiency for a given combination of mixing proportion and Mahalanobis distance. The optimal weights are plotted as a function of $(\alpha_1, \tau)$ in the heatmap of Figure~\ref{fig:as_fig_5}, the most striking features of which are the discontinuities at the horizontal lines $\alpha_1 = \delta_1$, $\alpha_1 = 1/2$ and $\alpha_1 = \delta_2$. These are caused by the fact that the coefficient $C_\eta$ in Theorem \ref{theo:asnorm_blind_3} becomes a constant function of  $w_1$ at each of these values of $\alpha_1$ (where either the skewness or excess kurtosis vanishes). Consequently, there is no unique maximizer $w_1$ at these points and to emphasize their nature we have chosen to color them in Figure~\ref{fig:as_fig_5} with the corresponding extreme color (e.g., black for $\alpha_1 = 1/2$ where skewness carries no information). However, more puzzling are the differing limits when approaching the points $\delta_1, \delta_2$ from below and above. Essentially, when one approaches either of these points from inside the interval, the highest efficiency is obtained by focusing all weight on kurtosis which seems very counter-intuitive as in the limit kurtosis carries no information at all. This behaviour is visualized still in more detail in Figure~\ref{fig:as_fig_6} which plots the relative asymptotic efficiency $C_\eta^{-1}$ as a function of $w_1$ for $\tau = 5$ and $\alpha_1 - \delta_1 =: \varepsilon \in \{ 0.001, 0.002, 0.005, 0.010 \}$. Clearly, the weight achieving the maximal efficiency indeed approaches zero as $\varepsilon \rightarrow 0$. Algebraically, it is easy to see what is happening: For $\tau\to\infty$ and $\beta\approx 1/6$, the approximation of $C_\eta$, obtained by ignoring the terms of order $(1 - 6 \beta)^2$ and higher is  $C_\eta\approx  \frac{1}{1-4\beta}+\frac{8}{3}\frac{1-w_1}{w_1}\frac{1-6\beta}{\beta(1-4\beta)}$. For $\alpha_1\to \delta_1$ from the inside of the interval we have $1-6\beta<0$, which yields that $C_\eta$ is minimized for $w_1\to 0$. Similarly, for $\alpha_1\to \delta_1$ from the outside of the interval we have $1-6\beta>0$, which yields that $C_\eta\geq 0$ and shows that it is minimized for $w_1=1$. Also, as $\beta\to 1/6$, no matter from which side, $C_\eta$ converges to constant in $w_1$, implying that there is no discontinuity in the efficiency value itself. No such behaviour is observed for $\alpha_1\to 0.5$ and the reason for this is that $1-4\beta\geq 0$, implying that no sign change occurs when passing the critical value $\alpha_1 = 1/2$.


\begin{figure}[ht]
\centering
    \includegraphics[width=0.6\textwidth]{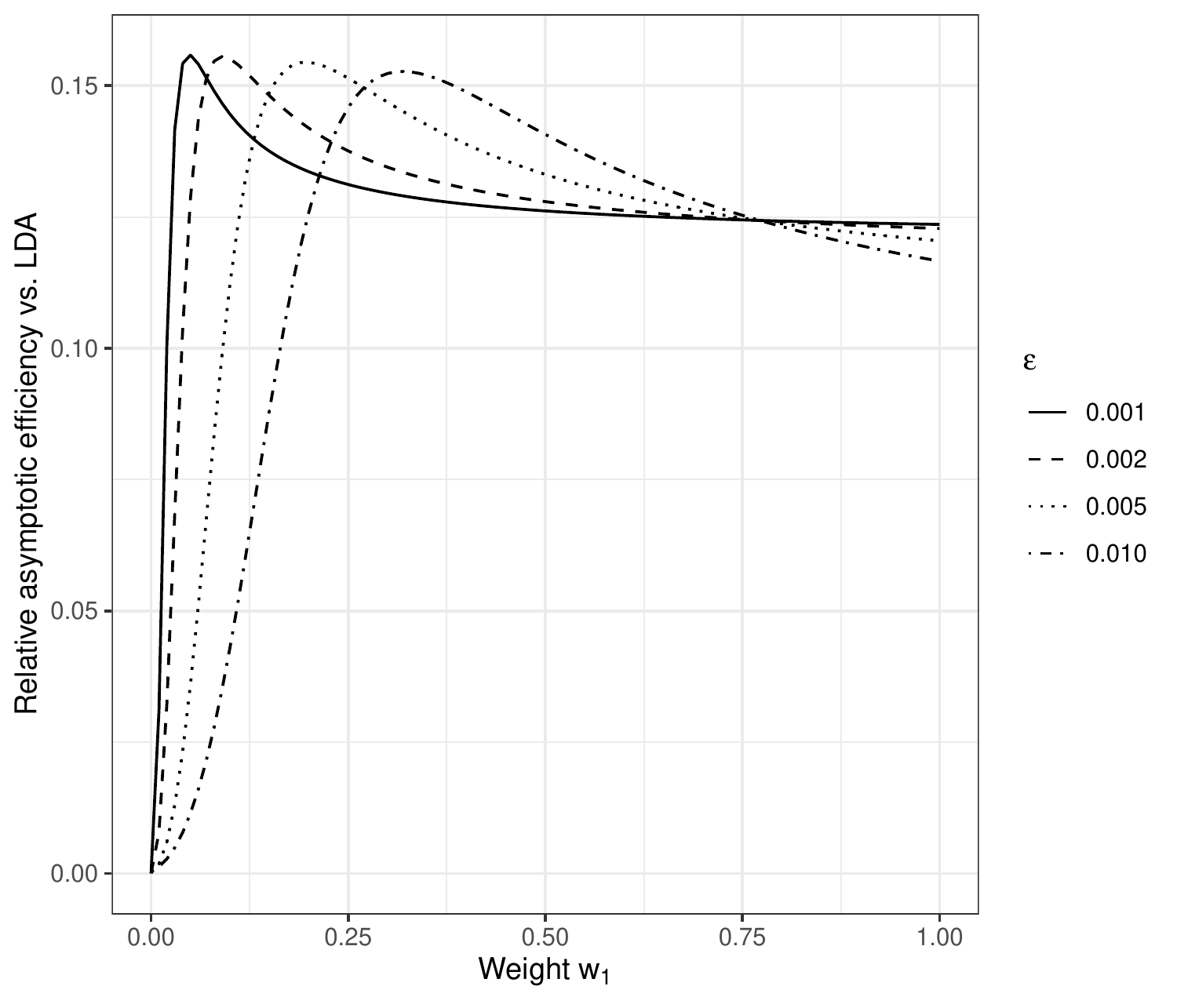}
    \caption{
    Relative asymptotic efficiencies of the hybrid estimator as a function of $w_1$ when $\tau = 5$ and $\alpha_1 = \delta_1 + \varepsilon$, where $\varepsilon \in \{ 0.001, 0.002, 0.005, 0.010 \}$. The value of the weight $w_1$ achieving the maximal efficiency can be seen to approach zero when $\varepsilon \rightarrow 0$.}
    \label{fig:as_fig_6}
\end{figure}

The efficiencies achieved by the optimal weighting are shown by the solid black line in Figure~\ref{fig:as_fig_5} and indicate that the hybrid estimator is able to reach satisfying levels of efficiency particularly when the mixing proportion lies in the interval $(\delta_1, \delta_2)$. Indeed, in the limit $\tau \rightarrow \infty$, within the interval there always exists a weighting that reaches efficiency equal to LDA, as evidenced by the right-most panel of Figure~\ref{fig:as_fig_5}. On the other hand, outside of the interval $(\delta_1, \delta_2)$, LDA is still, even in the limit $\tau \rightarrow \infty$, a superior choice. We conjecture that the reason for this critical difference in behaviour inside and outside of the interval is that when the value of $\alpha_1$ is extreme, one of the groups is small, making the pin-pointing of the optimal direction difficult in general, but even more so for the blind methods which have no class information available. However, it is not clear why the particular points $\delta_1, \delta_2$ serve as the cut-off values for this behavior.

Finally, note that the discontinuities make the use of the optimal choice of weighting somewhat difficult in practice, as, if one's prior information/guess on the value of the  mixing proportion $\alpha_1$ is even slightly off, relying on the seemingly optimal choice can in the worst case lead to relative efficiency close to zero. Moreover, recall that the previous experiments were asymptotical in nature and do not necessary reflect the behaviour of the method under sample sizes encountered in practical situations. Hence, we suggest using a ``safe'' universal value of $w_1$, most preferably falling in the interval $(0.725, 0.825)$ identified in conjunction with Figure~\ref{fig:as_fig_4}. For example, Figure~\ref{fig:as_fig_3} shows that the value $w_1 = 0.80$ delivers, for finite $\tau$, performance not far behind the optimal choice for any $\alpha_1$. However, if one is reasonably certain about the value of $\alpha_1$ (which, optimally, is far away from $\delta_1, \delta_2$) and has $n$ sufficiently large, resorting to the optimal choice is, of course, also possible.

\section{Simulations}\label{sec:simulations}

The three projection pursuit estimators considered here have been discussed in the context of ICA in detail in \citet{virta2016projection} where also fixed point algorithms for their computations are described. For our purpose here we can use their deflation-based algorithms when only one direction is to be extracted. Projection pursuit is considered notoriously prone to local optima and therefore it is known that good initial values for such algorithms are crucial. \citet{virta2016projection} suggest to use initial values based on a simple ICA method called FOBI \citep{Cardoso89sourceseparation}. This is also suitable in our context as the normal mixture is a sub-model of the ICA model, see Lemma \ref{lem:ic_mixture_intersection} in the supplementary material. The algorithms of \citet{virta2016projection} are implemented in the package ICtest \citep{ICtest}, which we will use in the following together with R 3.6.1 \citep{R}. Further details about software used are contained in the appendix.

Let $\textbf{u}_n$ be any of the estimators of $\boldsymbol{\theta}/\| \boldsymbol{\theta} \|$ discussed in Section \ref{sec:estimation}.
The accuracy of the estimator can in simulations be measured through the inner product $\textbf{u}_n'\boldsymbol{\theta}/\| \boldsymbol{\theta} \|$, which, by the Cauchy-Schwarz inequality, achieves the absolute value one if and only if the two vectors are parallel. In the continuation, we call the presented inner product ``Maximal similarity index'' (MSI). The following lemma presents the limiting distribution of this performance measure.

\begin{lemma}\label{lem:inner_product}
    Let the unit length vector $\textbf{u}_n$ satisfy $\sqrt{n}(s_n \textbf{u}_n - \boldsymbol{\theta}/\| \boldsymbol{\theta} \|) \rightsquigarrow \mathcal{N}_p(\textbf{0}, \boldsymbol{\Psi})$ for some sequence of signs $s_n \in \{ -1, 1 \}$ and limiting covariance matrix $\boldsymbol{\Psi}$. Then, as $n \rightarrow \infty$,
    \begin{align}\label{eq:inner_product_limiting}
        2 n ( 1 - s_n \textbf{u}_n' \boldsymbol{\theta}/\| \boldsymbol{\theta} \|) \rightsquigarrow \textbf{z}' \boldsymbol{\Psi} \textbf{z},
    \end{align}
    where the random vector $\textbf{z}$ obeys the $p$-variate standard normal distribution. Moreover, the expected value of the right-hand side of \eqref{eq:inner_product_limiting} is $\mathrm{tr}(\boldsymbol{\Psi})$.
\end{lemma}

Note that the sign correction in Lemma \ref{lem:inner_product} can be incorporated in practice by choosing the sign $s_n$ such that the quantity $s_n \textbf{u}_n' \boldsymbol{\theta}/\| \boldsymbol{\theta} \|$ is positive. In the simulations we will evaluate the performances of methods through the left-hand side of \eqref{eq:inner_product_limiting}. By Lemma \ref{lem:inner_product} the average of this criterion over several replicates should be close to the trace of the limiting covariance matrix of the corresponding estimator, for sample size $n$ large enough. Hence, the simulations also serve to ``verify'' our asymptotic results.

In the following simulations four projection pursuit (PP) directions have been calculated: kurtosis based (obtained by maximization of $(\kappa_n-3)^2$), skewness based (obtained by maximization of $\gamma_n^2$), ``safe'' hybrid estimator (obtained by maximization of $\eta_n$ for $w_1=0.8$) and ``optimal'' hybrid estimator (obtained by maximization of $\eta_n$ for $w_1=w_1(\alpha_1,\tau)$ which maximizes the relative asymptotic efficiency of the hybrid estimator w.r.t. LDA). Sign $s_n$ is chosen such that $s_n\textbf{u}_n'\boldsymbol\theta/||\boldsymbol\theta||\geq 0$. 
As discussed, the performances of the four presented PP directions $\textbf{u}_n$ are evaluated using the maximal similarity index $\textbf{u}_n'\boldsymbol\theta/||\boldsymbol\theta||$ from above.\\

For the first simulation setting, the means of maximal similarity indices $s_n\textbf{u}_n'\boldsymbol\theta/||\boldsymbol\theta||$ in the simulations are obtained using $1000$ random samples in each setting. In each $(\tau, \alpha_1, n)$-setting, where the Mahalanobis distance between the group means $\tau=1,2,\dots,20$, mixing proportion $\alpha_1=0.05,0.1,\dots,0.45,0.5$ and sample size $n=500,1000,2000,4000,8000,$ $16000,32000$, $m=1000$ random samples are generated from a $10$-dimensional normal mixture $\alpha_1 \mathcal N_{10}(\textbf{0},\boldsymbol\Sigma)+(1-\alpha_1)\mathcal N_{10}(\boldsymbol\mu_2,\boldsymbol\Sigma)$, where $\boldsymbol\Sigma$ is a covariance matrix with autoregression AR(1) structure with $\rho=0.6$, and $\boldsymbol \mu$ is in each setting chosen randomly such that $\boldsymbol\mu'\boldsymbol\Sigma^{-1}\boldsymbol\mu=\tau$.

The heatmaps of the MSI-values in Figure~\ref{fig:as_fig_7} show that for moderate sample sizes $(n\geq 2000)$, both hybrid estimators estimate the optimal LDA direction very well. It is also visible that kurtosis and skewness based PP directions perform very badly when $\alpha_1$ is near their corresponding discontinuity points. The hybrid estimators suffer from the same problem when $\alpha_1$ is near $1/2-1/\sqrt{12}$. The hybrid estimator with optimal $w_1$ is performing worse when $\alpha_1$ is approaching $1/2-1/\sqrt{12}$ from the inside the interval  $(1/2-1/\sqrt{12},1/2+1/\sqrt{12})$. 
 It is important to recall, that the criterion for choosing the optimal weight $w_1$ is an asymptotic one and thus might not perform well in small sample settings. Furthermore, in order to calculate the optimal weight $w_1$ for the hybrid estimator, one needs to know both the Mahalanobis distance $\tau$ between the group means and the mixing proportion $\alpha_1$, which is a rather unrealistic requirement in practice. Luckily, the hybrid estimator with the ``safe'' weight $w_1=0.8$ shows a very good performance in this simulation study, and is therefore recommended in cases where knowledge of $\alpha_1$ and $\tau$ is lacking.
Another observation based on this simulation is that for small sample sizes skewness based PP seems to be preferable. This might be due to the fact that moments of order three are easier to estimate than moments of order four.


\begin{figure}[ht]
\centering
    \includegraphics[width=0.95\textwidth]{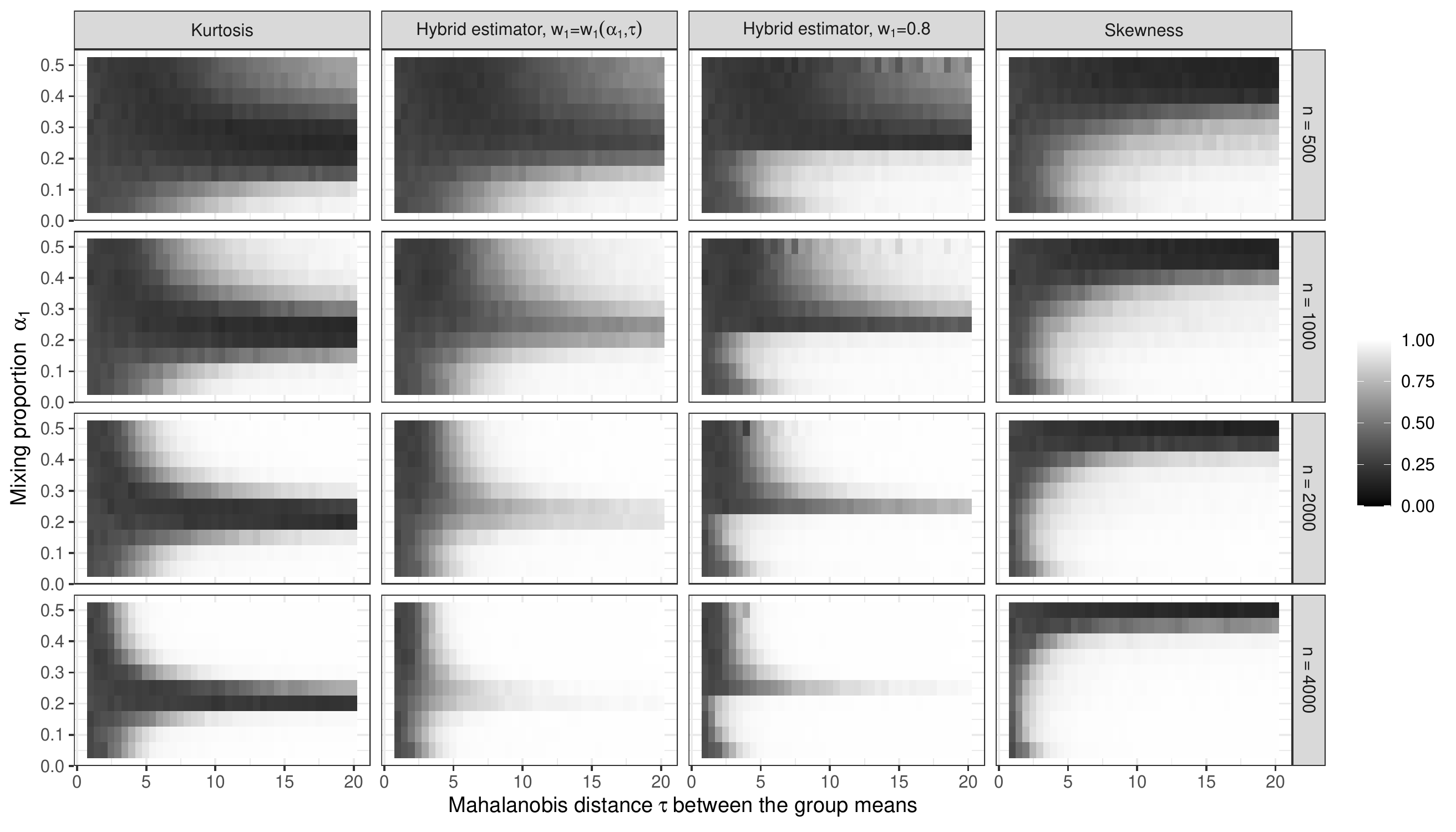}\\
    \caption{
    Average values of the MSI  $s_n\textbf{u}_n'\boldsymbol\theta/||\boldsymbol\theta||$ as a function of Mahalanobis distance $\tau$ between the group means and mixing proportion $\alpha_1$, where $\textbf{u}_n$ is one of the four estimators discussed above. 
    }
    \label{fig:as_fig_7}
\end{figure}

The corresponding heatmap of the standard deviation of MSI can be found in the Appendix, Figure~\ref{fig:as_fig_10}, and shows that for sample size and distance between the group means moderately large, deviation of the MSI to the to the corresponding mean, which is very close to the optimal value of $1$, is negligible, for most of the values of the mixing proportion $\alpha_1$. Heatmaps of mean and standard deviation of the MSI in Figures~\ref{fig:as_fig_14} and \ref{fig:as_fig_15} show that for large sample sizes $(n=8000,16000,32000)$ MSI is virtually $1$.

\begin{figure}[ht]
\centering
    \includegraphics[width=0.95\textwidth]{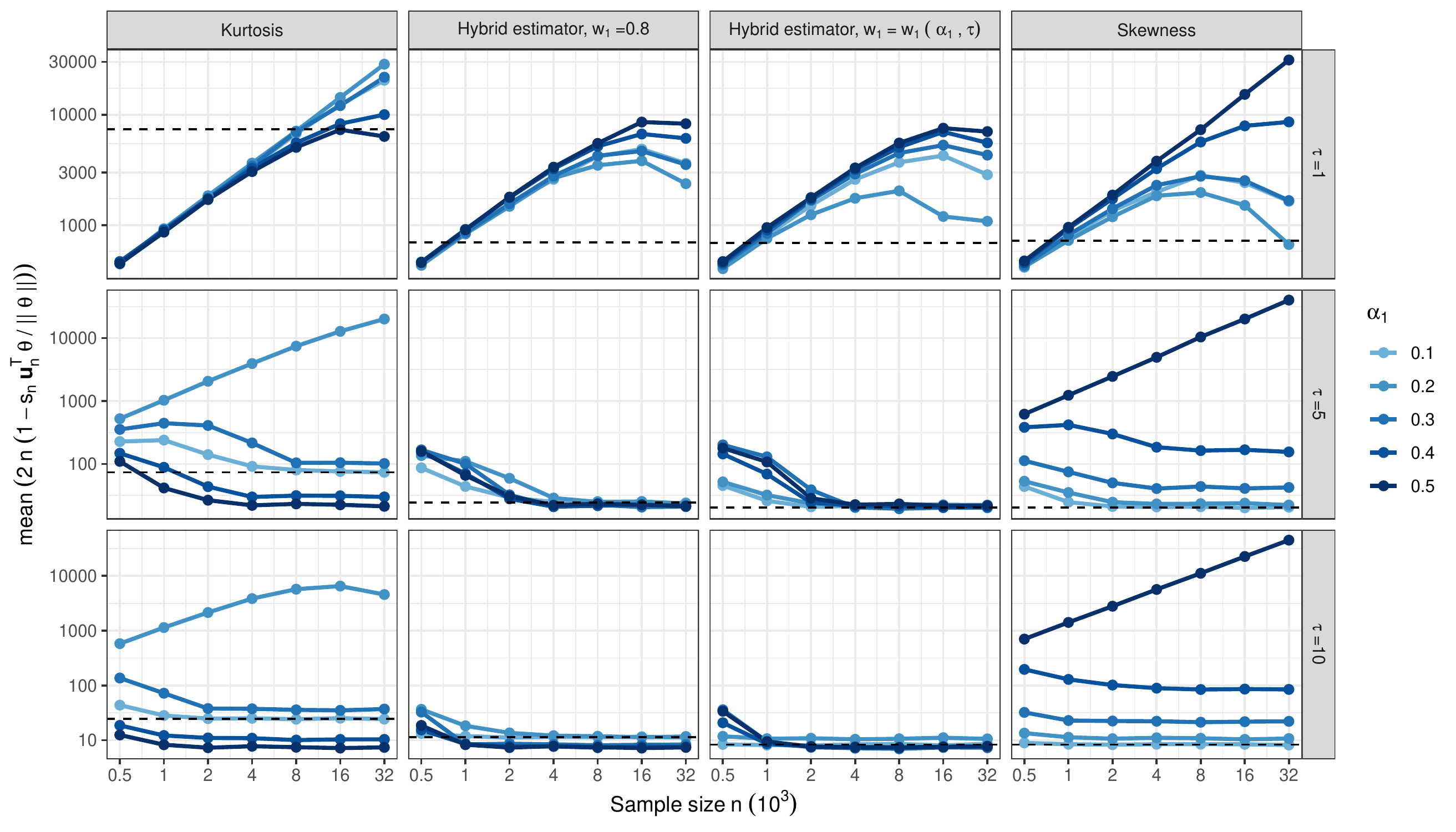}\\
    \caption{Average values of the MSI $s_n\textbf{u}_n'\boldsymbol\theta/||\boldsymbol\theta||$ as a function of $n$, for mixing proportion   $\alpha_1\in\{0.1,0.2,0.3,0.4,0.5\}$ and Mahalanobis distance between the group means  $\tau\in\{1,5,10\}$, where $\textbf{u}_n$ is one of the four estimators discussed above. 
    }
    \label{fig:as_fig_8}
\end{figure}

In the next simulation the theoretic results are to be confirmed by exploiting the results of Lemma~\ref{lem:inner_product}. For that purpose we select three values for $\tau \in \{1,5,10\}$, to represent hardly, moderately and clearly separated clusters, respectively. Then we simulate, for sample sizes $n=500,1000,2000,4000,8000,16000,32000$, from a three-variate Gaussian mixture model as specified above with $\boldsymbol\Sigma=\textbf{I}_3+\textbf{1}_3$, where $\textbf{1}_3$ is matrix of ones, $\boldsymbol\mu_1=\textbf 0$ and $\boldsymbol\mu_2=\boldsymbol\mu_2(\tau)$ is for each $\tau$ chosen such that Mahalanobis distance between the means is equal to $\tau$, i.e., $\boldsymbol\mu_2(1)=(0.68,-0.55, 0.6)',\, \boldsymbol\mu_2(5)=(0.81,-2.24, -0.36)',\,\boldsymbol\mu_2(10)=(3.06,1.6,-1.11)'$. For $\alpha_1=0.1,0.2,\dots,0.5$ we compute then for sample sizes $n=500,1000,2000,4000,8000,$ $16000,32000$ the means of the $2n(1-s_n \boldsymbol{u}^\top_n)\boldsymbol\theta/\|\boldsymbol\theta\|$ based on 2000 repetitions. The question is then whether those averages for the presented methods stabilize for the cases they are expected to work. In order not to clutter the figure we computed only for $\alpha_1=0.1$ trace $\mathrm{tr}(\boldsymbol{\Psi})$, for the corresponding matrix $\boldsymbol\Psi$. Figure~\ref{fig:as_fig_8} shows the results of this simulation and confirms the corresponding theoretic findings from above. The less separated the clusters are the more difficult the estimation and even for $n=32000$ observations there is no stabilization visible. But the more separated the two clusters are, the faster the stabilization. Also the closer $\alpha_1$ is to the critical value $1/2-1/\sqrt{12}$, the worse is kurtosis based PP. Skewness based PP similarly is better the more skewed the distribution and clearly does not work in the symmetric case. Both hybrid estimators show an excellent performance in this setting. It is also clearly visible that the empirical lines for the mixing proportion $\alpha_1=0.1$  correspond to the theoretically computed dashed lines given that the groups are separated enough and we assume for $\tau=1$ the line would be reached for much larger sample sizes. Though we show the theoretic lines only for one mixing proportion the behaviour is similar for all others naturally with the exception of skewness not working in the symmetric case.\\

Principal component analysis (PCA) can also be seen as a projection method where the variance is maximized. PCA is arguably the most popular dimension reduction method and often used before clustering. While skewness and kurtosis can be related to mixtures the variance does not have the same connection with the discriminant direction as the other cumulants. A theoretic consideration when PCA can be used to estimate the discriminant is in Appendix~\ref{sec:PCA}. Here we show an example where PCA fails.


\begin{figure}[ht]
\centering
    \includegraphics[width=0.7\textwidth]{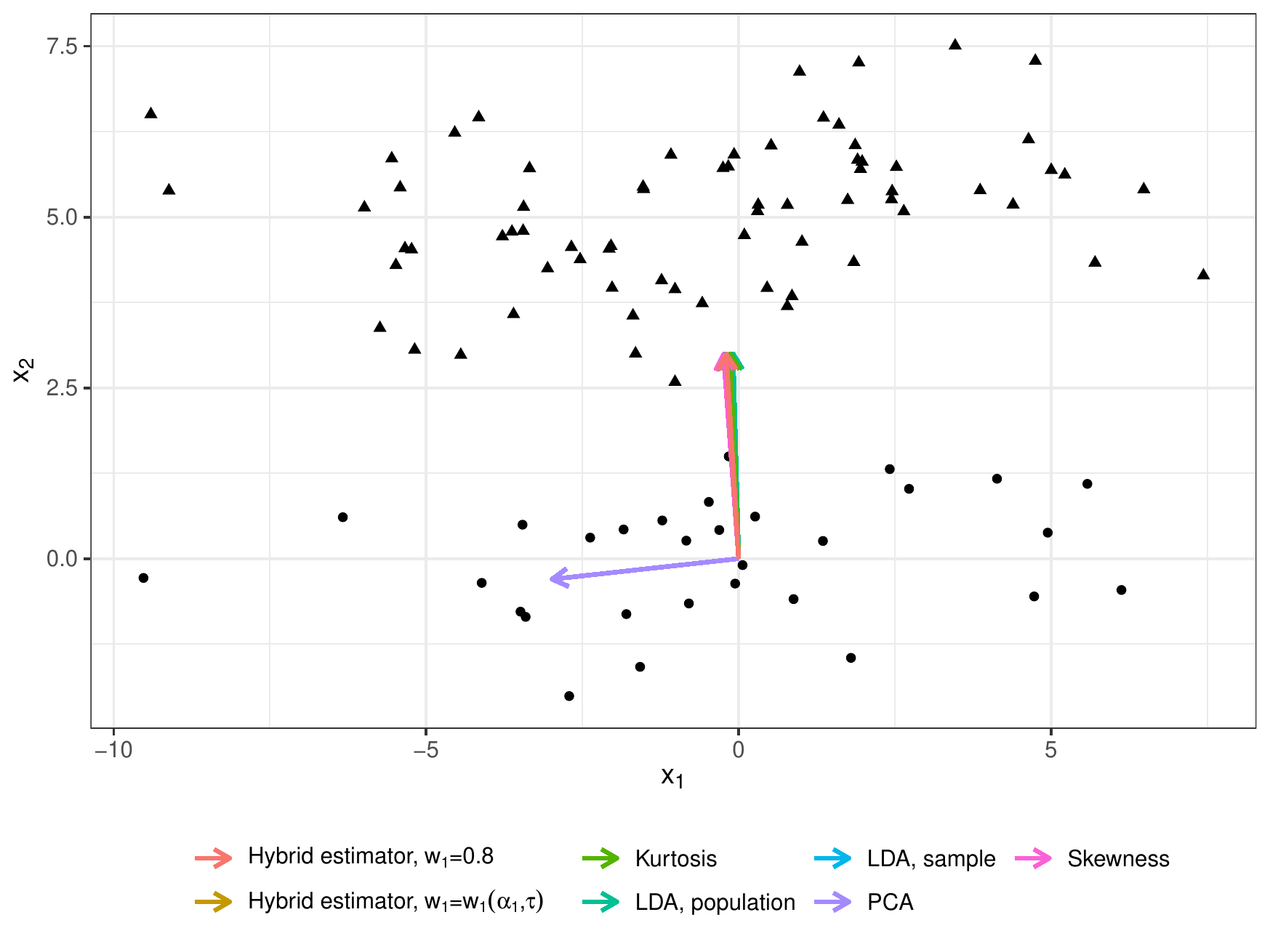}
    \caption{PP direction $\textbf{u}_n$ based on LDA (here denoted as ``LDA, sample''), PCA, and one of the four estimators discussed above. 
    }
    \label{fig:as_fig_9}
\end{figure}

Figure~\ref{fig:as_fig_9} visualizes a sample of size $n=100$ from our Gaussian mixture model  with $$
\boldsymbol\mu_1=(0,0)',\quad\boldsymbol\mu_2=(0,5)',\quad\boldsymbol\Sigma=\begin{pmatrix}10&0.3\\
    0.3&1\end{pmatrix}$$
    and mixing proportion $\alpha_1 = 0.3$. The figure contains then the  direction $\boldsymbol\theta/\|\boldsymbol\theta\|$ of the population LDA as well as its estimate together with our four PP methods considered in this section and the direction of the first principal component. As it is clearly visible here, there is no big difference between the methods except for PCA which points in a direction which contains no information for the separation of the two groups.

\section{Real data example}\label{sec:example}

To compare the hybrid estimator to PCA in a real data set we consider the finance data set available in the R package Rmixmod \citep{Rmixmod}, which consists of 889 records of companies where, based on four numeric summary statistics, it should be decided if the company is financially healthy or not, where the information is provided in the data set. The scatter plot matrix is given in the Appendix as Figure~\ref{fig:as_fig_12} and shows no clear clusters. As a reference we compute for the data set LDA and then compare this supervised estimate via the estimate $s_n\textbf{u}_n\boldsymbol\theta_n/\|\boldsymbol\theta_n\|$ of the MSI  to our hybrid estimator for different weights and to PCA. Figure~\ref{fig:as_fig_11} shows obtained MSI values for the discussed estimators. The figure clearly shows that as long as enough weight is given to kurtosis, the hybrid estimator based PP clearly outperforms PCA, while the performance is poor if skewness gets too much weight. This is not surprising as the amount of healthy (457) and bankrupt (432) companies is almost equal. The weight of 0.8 gives again a good performance. Nevertheless, even though for most values of $w_1$, and especially for suggested $w_1\in[0.7,0.8]$, hybrid estimators clearly outperform PCA,  achieved MSI values of around $0.5$ are not that good. Such performance can be explained with the low sample size and that the cluster centers are not that far apart, as for example is visualized in the Appendix in Figure~\ref{fig:as_fig_13}.

\begin{figure}[ht]
\centering
    \includegraphics[width=0.7\textwidth]{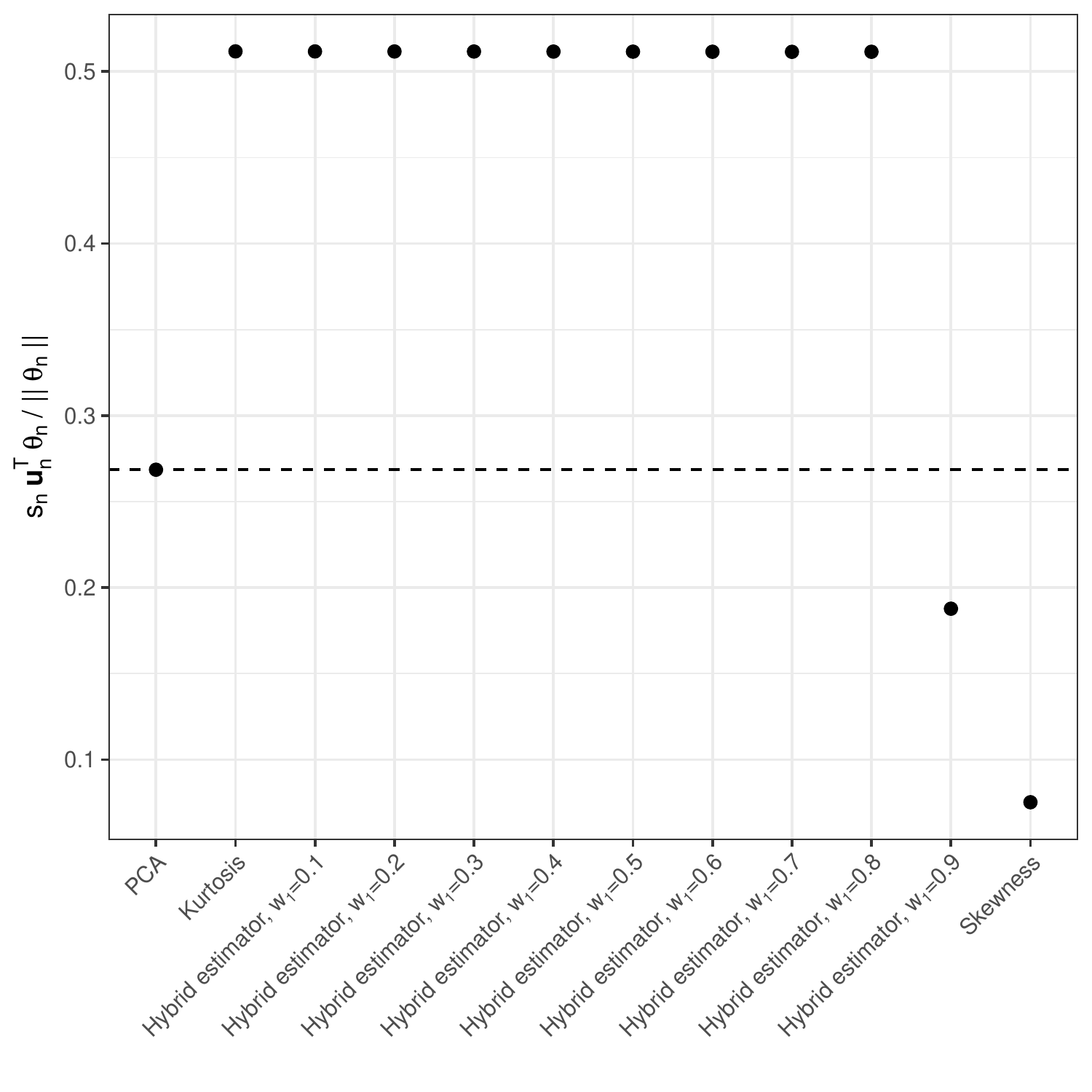}
    \caption{MSI $s_n\textbf{u}_n \boldsymbol\theta_n/\|\boldsymbol \theta_n\|$ for the finance data, where $\textbf{u}_n$ is the direction based on PCA and PP estimators for $w_1=0,0.1,0.2,\dots,0.9,1$. $\boldsymbol{\theta_n}$ is the direction based on LDA.}
    \label{fig:as_fig_11}
\end{figure}

\section{Discussion}\label{sec:discussion}

In this paper, we conducted an asymptotic comparison of two popular estimators of the linear discriminant direction, LDA and projection pursuit based on skewness and kurtosis. For the latter, we proposed using the weighted combination of kurtosis and skewness as the projection index (giving the individual cumulants as special cases). Both the theoretical results and simulations indicate that, with a suitable choice of weighting, such projection pursuit achieves reasonably good performance compared to LDA (e.g., around 15\% relative asymptotic efficiency if the Mahalanobis distance between the groups is 5, see Figure~\ref{fig:as_fig_3}), considering it operates in complete absence of any label information. Moreover, in the extreme case of balanced and infinitely well-separated groups, projection pursuit is able to reach asymptotic efficiency equal to LDA with an optimal choice of weighting.

The use of our optimal weighting results is  difficult in practice by the discontinuities around the mixing proportions $\delta_1, \delta_2$ observed in Section \ref{sec:combination}, see Figure~\ref{fig:as_fig_5}. As such, unless one is absolutely sure that the mixing proportion is not in these regions, our  recommendation is to use a universal choice of weighting, anything between $0.7$ and $0.8$ (as the weight for skewness) likely being a good choice.

At first we thought that the discontinuities, and the surprising recommendation to favor kurtosis just outside the interval $(\delta_1, \delta_2)$, might be caused by the uneven robustness properties of skewness and kurtosis in the objective function. Namely, being based on fourth moments, kurtosis is more affected by outliers than skewness (despite the standardization with second moments). Hence, we also considered using the ```balanced'' objective function,
\begin{align*}
    \eta^*(\textbf{u}) := w_1 \gamma(\textbf{u})^{8/3} + w_2 \{ \kappa(\textbf{u}) - 3 \}^2,
\end{align*}
in an attempt to put skewness and kurtosis on an equal footing. However, the asymptotic properties of $\eta^*$ (not shown here) turn out to be essentially the same as for $\eta$, including the discontinuities which are also observed for it. Note also that the discontinuous behavior was observed also in \cite{virta2016projection}, where the normal mixture model was studied using independent component analysis.

Similarly one could extend our considerations here to many other PP indices as well, which often are modifications of skewness or kurtosis (see e.g. \citet{HouWentzell2014}) or otherwise motivated to be useful in clustering or structure detection, see for example \citet{CookBujaCabrera1993,fischer2019repplab} and references therein for alternative indices. These indices are however often computationally expensive and therefore much less popular than skewness and kurtosis.

Finally, besides projection pursuit, there exist also other blind estimators of the linear discriminant. For example, it is known that the linear discriminant can be reconstructed using invariant coordinate selection (ICS) \citep{TylerCritchleyDuembgenOja2009} where two scatter matrices are jointly diagonalized.  Especially when using the regular covariance matrix and the scatter matrix of fourth moments in this context as, for example, suggested in \citep{AlashwaliKent2016,pena2010eigenvectors}, would allow a theoretic comparison. Actually this combination corresponds to FOBI mentioned in Section~\ref{sec:simulations}. Another prospective line of work is the extension of our asymptotic results to mixtures of elliptical distributions, as \cite{pena2001cluster} indeed showed that projection pursuit yields a Fisher consistent estimator of the linear discriminant also in the case of general elliptical families. Similarly, another possible extension is to the case of multiple groups instead of just two or to groups with unequal covariance matrices.
\newpage
\appendix

\section{Equivalent results for PCA}\label{sec:PCA}

While PCA manages to capture the linear discriminant direction only under very specific conditions, and cannot thus be reasonably seen as a ``blind'' estimator of it, we still give for it in the following, for completeness, equivalent results to the ones in Sections \ref{sec:estimation} and \ref{sec:combination}. The first result, detailing conditions required for the Fisher consistency of PCA, is \textit{qualitatively} well-known in the literature (see, e.g., Section 9.1 in \cite{jolliffe2002principal}), but, as far as we know, the exact eigenvalue bound is novel.

\begin{lemma}\label{lem:PCA_fisher}
    Given model \eqref{eq:xy_model}, the following two are equivalent:
    \begin{enumerate}
        \item[i)] The vector $ \textbf{h} := \boldsymbol{\mu}_2 - \boldsymbol{\mu}_1 $ is an eigenvector of $\boldsymbol{\Sigma}$ and, denoting the corresponding eigenvalue with $\phi$, the second-to-largest eigenvalue $\phi_2\{ \mathrm{Cov}(\textbf{x}) \}$ of $\mathrm{Cov}(\textbf{x})$ satisfies
    \begin{align*}
        \phi_2\{ \mathrm{Cov}(\textbf{x}) \} < \phi ( 1 + \beta \tau),
    \end{align*}
    where $\beta$ and $\tau$ are as in Theorem \ref{theo:asnorm_blind}.
    \item[ii)] The unique leading unit length eigenvectors of $\mathrm{Cov}(\textbf{x})$ are $\pm \boldsymbol{\theta}/\| \boldsymbol{\theta} \|$.
    \end{enumerate}
\end{lemma}

Lemma \ref{lem:PCA_fisher} states that for the first PC to recover the discriminant direction, it is necessary that the difference $\boldsymbol{\mu}_2 - \boldsymbol{\mu}_1$ between the group means is an eigenvector of $\boldsymbol{\Sigma}$. However, it is not necessary for it to be the leading eigenvector but instead, roughly, the more well-separated the groups are (large Mahalanobis distance $\tau$) and the more balanced the groups are (large $\beta$), the smaller the corresponding eigenvalue can be relative to the rest of the spectrum. Note also that in the spherical case, $\boldsymbol{\Sigma} \propto \textbf{I}_p$, the first part of condition i) in Lemma \ref{lem:PCA_fisher} is trivially satisfied.

Asymptotic results for PCA are also well-known, see, e.g., \cite{anderson1963asymptotic,davis1977asymptotic}, and the following theorem details the strong consistency and the limiting normality of the first PC in our particular scenario. For completeness, we provide a proof.

\begin{theorem}\label{theo:asnorm_PCA}
    Given model \eqref{eq:xy_model}, assume that the condition i) (or, equivalently, ii)) holds and let $\textbf{u}_n$ be any sequence of leading unit-length eigenvectors of the sample covariance matrix $\textbf{C}_n$ of $\textbf{x}_1, \ldots , \textbf{x}_n$. Then, there exists a sequence of signs $s_n \in \{ -1 , 1 \}$ such that, as $n \rightarrow \infty$,
    \begin{itemize}
        \item[1)] $s_n \textbf{u}_n \rightarrow \boldsymbol{\theta}/\| \boldsymbol{\theta} \|$, almost surely.
        \item[2)] $ \sqrt{n}(s_n \textbf{u}_n - \boldsymbol{\theta}/\| \boldsymbol{\theta} \|) \rightsquigarrow \mathcal{N}_p(\textbf{0}, \boldsymbol{\Psi}_{\mathrm{PCA}}) $, where
        \begin{align*}
            \boldsymbol{\Psi}_{\mathrm{PCA}} := \left( \frac{1 + \beta \tau}{\| \boldsymbol{\theta} \|^2} \right) \textbf{M}^\dagger \left[ \tau (\boldsymbol{\Sigma} - \tau^{-1} \textbf{h} \textbf{h}') + (1 + \beta \tau) \{ \kappa(\boldsymbol{\theta}) - 1 \} \textbf{h} \textbf{h}' \right]  \textbf{M}^\dagger,
        \end{align*}
        where $\textbf{M} := \boldsymbol{\Sigma} + \beta \textbf{h} \textbf{h}' - \lambda_1 \textbf{I}_p$, $\lambda_1$ is the eigenvalue of $\mathrm{Cov}(\textbf{x})$ corresponding to the eigenvector $ \boldsymbol{\theta}/\| \boldsymbol{\theta} \| $, $\textbf{M}^{\dagger}$ denotes the Moore-Penrose pseudoinverse of $\textbf{M}$ and $\kappa (\boldsymbol{\theta})$ is the kurtosis of $\textbf{x}$ in the direction $\boldsymbol{\theta}$.
    \end{itemize}
\end{theorem}

It is evident from part 2) of Theorem \ref{theo:asnorm_PCA} that the limiting covariance matrix of the PCA-based estimator is not proportional to the four others derived in Theorems \ref{theo:asnorm_unblind}, \ref{theo:asnorm_blind}, \ref{theo:asnorm_blind_2} and \ref{theo:asnorm_blind_3}. However, proportionality is reached in the special case where the group covariance matrix is spherical, $\boldsymbol{\Sigma} = \sigma^2 \textbf{I}_p$, for some $\sigma^2 > 0$. In this case, $\textbf{h}/\| \textbf{h} \| = \boldsymbol{\theta}/\| \boldsymbol{\theta} \|$, the Moore-Penrose pseudoinverses in Theorem \ref{theo:asnorm_PCA} equal $ (\boldsymbol{\Sigma} + \beta \textbf{h} \textbf{h}' - \phi \textbf{I}_p)^\dagger = -1/(\beta \| \textbf{h} \|^2) (\textbf{I}_p - \textbf{h} \textbf{h}'/\| \textbf{h} \|^2)$ and the limiting covariance matrix $\boldsymbol{\Psi}_{\mathrm{PCA}}$ can be expressed as
\begin{align*}
    \boldsymbol{\Psi}_{\mathrm{PCA}} = \frac{1}{\tau \beta} \left( \frac{1 + \beta \tau}{\| \boldsymbol{\theta} \|^2 \beta} \right) \left( \textbf{I}_p - \frac{\boldsymbol{\theta} \boldsymbol{\theta}'}{\| \boldsymbol{\theta} \|^2} \right) \boldsymbol{\Sigma}^{-1} \left( \textbf{I}_p - \frac{\boldsymbol{\theta} \boldsymbol{\theta}'}{\| \boldsymbol{\theta} \|^2} \right).
\end{align*}
Comparison to Theorem \ref{theo:asnorm_unblind} now reveals that the relative asymptotic efficiency of PCA vs. LDA equals $\tau \beta$, showing, in particular, that in the balanced case with $\alpha_1 = \alpha_2$ PCA surpasses LDA in asymptotic efficiency as soon as the Mahalanobis distance between the groups is greater than 4. Moreover, in the limit $\tau \rightarrow \infty$, PCA is infinitely more efficient than LDA regardless of the mixing proportion. This counterintuitive result is, of course, not something one should rely on in practice, as the conditions required to achieve the situation are being very restrictive.

\section{Proofs of technical results}\label{sec:proofs}

\begin{lemma}\label{lem:ic_mixture_intersection}
	Let $\textbf{x} \sim \alpha_1 \mathcal{N}_p(\boldsymbol{\mu}_1, \boldsymbol{\Sigma}) + \alpha_2 \mathcal{N}_p(\boldsymbol{\mu}_2, \boldsymbol{\Sigma})$, where $\alpha_1, \alpha_2 > 0$, $\alpha_1 + \alpha_2 = 1$, $\boldsymbol{\mu}_1, \boldsymbol{\mu}_2 \in \mathbb{R}^p$, $\boldsymbol{\mu}_1 \neq \boldsymbol{\mu}_2$, and $\boldsymbol{\Sigma} \in \mathbb{R}^{p \times p} $ is full rank. Then $\textbf{x}$ is an independent component model, i.e., there exists an invertible matrix $\boldsymbol{\Gamma} \in \mathbb{R}^{p \times p}$ such that $\boldsymbol{\Gamma} \{ \textbf{x} - \mathrm{E}(\textbf{x}) \}$ has independent components.
\end{lemma}

\begin{proof}[Proof of Lemma \ref{lem:ic_mixture_intersection}]
	We have
	\begin{align*}
	\textbf{x} - \mathrm{E}(\textbf{x}) \sim \alpha_1 \mathcal{N}_p(-\alpha_2 \textbf{h}, \boldsymbol{\Sigma}) + \alpha_2 \mathcal{N}_p(\alpha_1 \textbf{h}, \boldsymbol{\Sigma}),
	\end{align*}
	where $\textbf{h} := \boldsymbol{\mu}_2 - \boldsymbol{\mu}_1$. Let then $ \boldsymbol{\Gamma} := \textbf{U}' \boldsymbol{\Sigma}^{-1/2} $ where $\textbf{U}$ is an orthogonal matrix satisfying $ \textbf{U}' \boldsymbol{\Sigma}^{-1/2} \textbf{h} \propto \textbf{e}_1$ and $\textbf{e}_1$ is the first canonical basis vector of $\mathbb{R}^p$ (such an $\textbf{U}$ always exists as $\boldsymbol{\Sigma}$ is full rank and $\textbf{h} \neq \textbf{0}$). Now,
	\begin{align*}
	\boldsymbol{\Gamma} \{ \textbf{x} - \mathrm{E}(\textbf{x}) \} \sim \alpha_1 \mathcal{N}_p(-\alpha_2 b \textbf{e}_1, \textbf{I}_p) + \alpha_2 \mathcal{N}_p(\alpha_1 b \textbf{e}_1, \textbf{I}_p),
	\end{align*}
	for some $b \neq 0$. The result now follows by writing out the density function of $\boldsymbol{\Gamma}\{ \textbf{x} - \mathrm{E}(\textbf{x}) \}$ and observing that it factors into a product of the density of a univariate Gaussian mixture and the densities of $p - 1$ univariate Gaussians with zero means.
\end{proof}

\begin{proof}[Proof of Theorem \ref{theo:asnorm_unblind}]

The estimator $\textbf{w}$ is translation invariant, meaning that we may, without loss of generality, assume that $\mathrm{E}(\textbf{x}) = \textbf{0}$. Under this, the model \eqref{eq:xy_model} takes the form
\begin{align}
    y  \sim \mathrm{Ber}(\alpha_1) \quad \mbox{and} \quad \textbf{x} \mid y \sim \mathcal{N}_p \{ -y \alpha_2 \textbf{h} + (1 - y) \alpha_1 \textbf{h}, \boldsymbol{\Sigma} \},
\end{align}
where $\textbf{h} := \boldsymbol{\mu}_2 - \boldsymbol{\mu}_1$. Concurrently, $\textbf{x} \sim \alpha_1 \mathcal{N}_p(-\alpha_2 \textbf{h}, \boldsymbol{\Sigma}) + \alpha_2 \mathcal{N}_p(\alpha_1 \textbf{h}, \boldsymbol{\Sigma})$ and $\mathrm{Cov}(\textbf{x}) = \boldsymbol{\Sigma} + \beta \textbf{h} \textbf{h}'$.

We begin by deriving asymptotic linearizations for $\bar{\textbf{x}}_{n2} - \bar{\textbf{x}}_{n1}$. Let in the following $\beta := \alpha_1 \alpha_2$. By LLN, $\bar{y}_n \rightarrow_p \alpha_1$ and $(1/n) \sum_i y_i \textbf{x}_i \rightarrow_p -\beta \textbf{h}$. Hence, the relation $\bar{y}_n \bar{\textbf{x}}_{n1} = (1/n) \sum_i y_i \textbf{x}_i$ shows that $\bar{\textbf{x}}_{n1} \rightarrow_p -\alpha_2 \textbf{h}$. We further have the expansion,
\begin{align*}
    \sqrt{n} \left( \frac{1}{n}  \sum_{i=1}^n y_i \textbf{x}_i + \beta \textbf{h} \right) = \sqrt{n} \left( \bar{y}_n - \alpha_1 \right) \bar{\textbf{x}}_{n1} + \alpha_1 \sqrt{n} (\bar{\textbf{x}}_{n1} + \alpha_2 \textbf{h}),
\end{align*}
which, by CLT, shows that $\bar{\textbf{x}}_{n1}$ is asymptotically normal,
\begin{align*}
    \alpha_1 \sqrt{n} (\bar{\textbf{x}}_{n1} + \alpha_2 \textbf{h}) = \sqrt{n} \left( \frac{1}{n}  \sum_i y_i \textbf{x}_i + \beta \textbf{h} \right) + \sqrt{n} ( \bar{y}_n - \alpha_1 ) \alpha_2 \textbf{h} + o_p(1).
\end{align*}
One can similarly show that,
\begin{align*}
    \alpha_2 \sqrt{n} (\bar{\textbf{x}}_{n2} - \alpha_1 \textbf{h}) = \sqrt{n} \left\{ \frac{1}{n}  \sum_i (1 - y_i) \textbf{x}_i - \beta \textbf{h} \right\} + \sqrt{n} (\bar{y}_n - \alpha_1 ) \alpha_1 \textbf{h} + o_p(1).
\end{align*}
Defining $\textbf{a}_{n1} := \sqrt{n} \{ (1/n) \sum_i y_i \textbf{x}_i + \beta \textbf{h} \}$, $\textbf{a}_{n2} := \sqrt{n} \{ (1/n) \sum_i (1 - y_i) \textbf{x}_i - \beta \textbf{h} \}$ and $a_{n3} := \sqrt{n} ( \bar{y}_n - \alpha_1 )$, the previous two can be written as $ \alpha_1 \sqrt{n} (\bar{\textbf{x}}_{n1} + \alpha_2 \textbf{h}) = \textbf{a}_{n1} + a_{n3}  \alpha_2 \textbf{h} + o_p(1) $ and $ \alpha_2 \sqrt{n} (\bar{\textbf{x}}_{n2} - \alpha_1 \textbf{h}) = \textbf{a}_{n2} + a_{n3}  \alpha_1 \textbf{h} + o_p(1)  $. The two in combination yield the desired linearization,
\begin{align*}
    \beta \sqrt{n}(\bar{\textbf{x}}_{n2} - \bar{\textbf{x}}_{n1} - \textbf{h}) &= \alpha_1 \textbf{a}_{n2} - \alpha_2 \textbf{a}_{n1} + (\alpha_1 - \alpha_2) a_{n3} \textbf{h} + o_p(1).
\end{align*}

We then derive a similar expansion for the pooled covariance matrix $\textbf{S}_n$. It is straightforwardly seen that $\sum_{i=1}^n y_i (\textbf{x}_i - \bar{\textbf{x}}_{n1}) (\textbf{x}_i - \bar{\textbf{x}}_{n1})' = \sum_i y_i \textbf{x}_i \textbf{x}_i' - \sum_i y_i \bar{\textbf{x}}_1 \bar{\textbf{x}}_1'$. This together with the equivalent formula for the second group yields
\begin{align*}
    \textbf{S}_n = \frac{1}{n - 2} \left\{  \sum_{i=1}^n \textbf{x}_i \textbf{x}_i' - \sum_{i=1}^n y_i \bar{\textbf{x}}_{n1} \bar{\textbf{x}}_{n1}' - \sum_{i=1}^n (1 - y_i) \bar{\textbf{x}}_{n2} \bar{\textbf{x}}_{n2}' \right\}.
\end{align*}
Since $\sqrt{n}( \textbf{S}_n - \boldsymbol{\Sigma}) = \sqrt{n}[\{ (n - 2)/n \} \textbf{S}_n - \boldsymbol{\Sigma}] + o_p(1) $, we have the expansion,
\begin{align*}
    \sqrt{n}( \textbf{S}_n - \boldsymbol{\Sigma}) &= \sqrt{n} \left\{ \frac{1}{n} \sum_{i=1}^n \textbf{x}_i \textbf{x}_i' - (\boldsymbol{\Sigma} + \beta \textbf{h} \textbf{h}') \right\} - \sqrt{n} \left( \frac{1}{n} \sum_{i=1}^n y_i \bar{\textbf{x}}_{n1} \bar{\textbf{x}}_{n1}' - \alpha_2 \beta \textbf{h} \textbf{h}' \right)\\
    &- \sqrt{n} \left\{ \frac{1}{n} \sum_{i=1}^n (1 - y_i) \bar{\textbf{x}}_{n2} \bar{\textbf{x}}_{n2}' - \alpha_1 \beta \textbf{h} \textbf{h}' \right\} + o_p(1).
\end{align*}
The second term above expands as,
\begin{align*}
    \sqrt{n} \left( \frac{1}{n} \sum_{i=1}^n y_i \bar{\textbf{x}}_{n1} \bar{\textbf{x}}_{n1}' - \alpha_2 \beta \textbf{h} \textbf{h}' \right) =& \alpha_2^2 a_{n3} \textbf{h} \textbf{h}' - \beta \sqrt{n} (\bar{\textbf{x}}_{n1} + \alpha_2 \textbf{h}) \textbf{h}' - \beta \textbf{h} \sqrt{n} (\bar{\textbf{x}}_{n1} + \alpha_2 \textbf{h})'\\
    &+ o_p(1),
\end{align*}
and the third as,
\begin{align*}
    \sqrt{n} \left\{ \frac{1}{n} \sum_{i=1}^n (1 - y_i) \bar{\textbf{x}}_{n2} \bar{\textbf{x}}_{n2}' - \alpha_1 \beta \textbf{h} \textbf{h}' \right\} =& -a_{n3} \alpha_1^2 \textbf{h} \textbf{h}' + \beta \sqrt{n} (\bar{\textbf{x}}_{n2} - \alpha_1 \textbf{h}) \textbf{h}'\\
    &+ \beta \textbf{h} \sqrt{n} (\bar{\textbf{x}}_{n2} - \alpha_1 \textbf{h})' + o_p(1).
\end{align*}
Denoting then $\textbf{A}_{n4} := \sqrt{n} \{ (1/n) \sum_i \textbf{x}_i \textbf{x}_i' - (\boldsymbol{\Sigma} + \beta \textbf{h} \textbf{h}') \} $, the linearizations derived earlier for $\bar{\textbf{x}}_{n1}$ and $\bar{\textbf{x}}_{n2}$ allow us to write,
\begin{align*}
    \sqrt{n}( \textbf{S}_n - \boldsymbol{\Sigma}) &= \textbf{A}_{n4} + a_{n3} (\alpha_2 - \alpha_1) \textbf{h} \textbf{h}' + \alpha_2 \textbf{a}_{n1} \textbf{h}' + \alpha_2 \textbf{h} \textbf{a}_{n1}'  - \alpha_1 \textbf{a}_{n2} \textbf{h}' - \alpha_1 \textbf{h} \textbf{a}_{n2}' + o_p(1).
\end{align*}
The above in particular shows that $\textbf{S}$ is asymptotically normal. Hence, the relation
\begin{align*}
    \textbf{0} = \sqrt{n}(\textbf{S}_n \textbf{S}_n^{-1} - \textbf{I}_p) = \sqrt{n}(\textbf{S}_n  - \boldsymbol{\Sigma})\textbf{S}_n^{-1} + \boldsymbol{\Sigma} \sqrt{n}( \textbf{S}_n^{-1} - \boldsymbol{\Sigma}^{-1}),
\end{align*}
gives $ \sqrt{n}( \textbf{S}_n^{-1} - \boldsymbol{\Sigma}^{-1}) = -\boldsymbol{\Sigma}^{-1} \sqrt{n}(\textbf{S}_n - \boldsymbol{\Sigma}) \boldsymbol{\Sigma}^{-1} + o_p(1) $.

We are now equipped to derive the limiting distribution of the optimal direction $\textbf{w}_n = \textbf{S}_n^{-1} (\bar{\textbf{x}}_{n2} - \bar{\textbf{x}}_{n1})$. Recalling that $\boldsymbol{\theta} = \boldsymbol{\Sigma}^{-1} \textbf{h}$, we have, by the calculus of $o_p(1)$ and $O_p(1)$ sequences,
\begin{align*}
    \sqrt{n}(\textbf{w}_n - \boldsymbol{\theta}) &= \sqrt{n}( \textbf{S}_n^{-1} - \boldsymbol{\Sigma}^{-1}) \textbf{h} + \boldsymbol{\Sigma}^{-1} \sqrt{n}(\bar{\textbf{x}}_{n2} - \bar{\textbf{x}}_{n1} - \textbf{h}) + o_p(1) \\
    &= -\boldsymbol{\Sigma}^{-1} \sqrt{n}(\textbf{S}_n  - \boldsymbol{\Sigma}) \boldsymbol{\theta} + \boldsymbol{\Sigma}^{-1} \sqrt{n}(\bar{\textbf{x}}_{n2} - \bar{\textbf{x}}_{n1} - \textbf{h}) + o_p(1).
\end{align*}
Hence,
\begin{align*}
    \beta \boldsymbol{\Sigma} \sqrt{n}(\textbf{w}_n - \boldsymbol{\theta}) &= -\beta \textbf{A}_{n4} \boldsymbol{\theta} - a_{n3} \beta (\alpha_2 - \alpha_1) \textbf{h} \textbf{h}'\boldsymbol{\theta} - \alpha_2 \beta \textbf{a}_{n1} \textbf{h}'\boldsymbol{\theta}\\
    &- \alpha_2 \beta \textbf{h} \textbf{a}_{n1}'\boldsymbol{\theta} + \alpha_1 \beta \textbf{a}_{n2} \textbf{h}'\boldsymbol{\theta} + \alpha_1 \beta \textbf{h} \textbf{a}_{n2}'\boldsymbol{\theta} \\
    &+ \alpha_1 \textbf{a}_{n2} - \alpha_2 \textbf{a}_{n1} + a_{n3} (\alpha_1 - \alpha_2) \textbf{h} + o_p(1)\\
    &= -\beta \textbf{A}_{n4} \boldsymbol{\theta} - a_{n3} (\alpha_2 - \alpha_1) (\beta \textbf{h}' \boldsymbol{\theta} + 1) \textbf{h} - \alpha_2 (\beta \textbf{h}'\boldsymbol{\theta} + 1) \textbf{a}_{n1} \\
    &+ \alpha_1 (\beta \textbf{h}'\boldsymbol{\theta} + 1) \textbf{a}_{n2} - \alpha_2 \beta \textbf{h} \textbf{a}_{n1}'\boldsymbol{\theta} + \alpha_1 \beta \textbf{h} \textbf{a}_{n2}'\boldsymbol{\theta} + o_p(1).
\end{align*}
By the definitions of $\textbf{a}_{n1}, \textbf{a}_{n2}, a_{n3}$ and $\textbf{A}_{n4}$ and CLT, the limiting covariance matrix of $\beta \boldsymbol{\Sigma} \sqrt{n}(\textbf{w} - \boldsymbol{\theta})$ is
\begin{align*}
    \mathrm{Cov} \{ -\beta (\boldsymbol{\theta}'\textbf{x}) \textbf{x} - (\alpha_2 - \alpha_1) \Delta y \textbf{h} - \alpha_2 \Delta y \textbf{x} + \alpha_1 \Delta (1 - y) \textbf{x} - \alpha_2 \beta y (\boldsymbol{\theta}' \textbf{x}) \textbf{h} + \alpha_1 \beta (1 - y) (\boldsymbol{\theta}' \textbf{x}) \textbf{h} \},
\end{align*}
where $\Delta := \beta \lambda + 1$ and $\lambda := \boldsymbol{\theta}'\textbf{h}$. The covariance matrix is a sum of a total of 36 terms, which we next compute one-by-one. We use the notation $\gamma := \alpha_1^3 + \alpha_2^3$.
\begin{multicols}{2}
\begin{itemize}
    \item[(1, 1):] $\beta^2 ( 1 + 3 \beta \lambda + \beta (\gamma - \beta) \lambda^2) \textbf{h} \textbf{h}' + \beta^2 \lambda ( 1 + \beta \lambda) \boldsymbol{\Sigma}$.
    \item[(1, 2):] $\beta^2 (\alpha_2 - \alpha_1)^2 (\beta \lambda + 1) \lambda \textbf{h} \textbf{h}'$.
    \item[(1, 3):] $ - \beta^2 \alpha_2 \lambda (\beta \lambda + 1) \boldsymbol{\Sigma} - \beta^2 \alpha_2 (\beta \lambda + 1) \{1 + \lambda \alpha_2 (\alpha_2 - \alpha_1) \} \textbf{h} \textbf{h}'$.
    \item[(1, 4):] $ - \beta^2 \alpha_1 \lambda (\beta \lambda + 1) \boldsymbol{\Sigma} - \beta^2 \alpha_1 (\beta \lambda + 1) \{1 + \lambda \alpha_1 (\alpha_1 - \alpha_2) \} \textbf{h} \textbf{h}'$.
    \item[(1, 5):] $\beta^3 \alpha_2 \lambda \{ \lambda \alpha_2 (\alpha_1 - \alpha_2) - 2 \} \textbf{h} \textbf{h}'$.
    \item[(1, 6):] $\beta^3 \alpha_1 \lambda \{ \lambda \alpha_1 (\alpha_2 - \alpha_1) - 2 \} \textbf{h} \textbf{h}'$.
    \item[(2, 1):] $\beta^2 (\alpha_2 - \alpha_1)^2 (\beta \lambda + 1) \lambda \textbf{h} \textbf{h}'$.
    \item[(2, 2):] $(\alpha_2 - \alpha_1)^2 (\beta \lambda + 1)^2 \beta \textbf{h} \textbf{h}'$.
    \item[(2, 3):] $-(\alpha_2 - \alpha_1) (\beta \lambda + 1)^2 \alpha_2^2 \beta \textbf{h} \textbf{h}'$.
    \item[(2, 4):] $(\alpha_2 - \alpha_1) \alpha_1^2 (\beta \lambda + 1)^2 \beta \textbf{h} \textbf{h}' $.
    \item[(2, 5):] $-(\alpha_2 - \alpha_1) (\beta \lambda + 1) \alpha_2^2 \beta^2 \lambda \textbf{h} \textbf{h}'$.
    \item[(2, 6):] $(\alpha_2 - \alpha_1)(\beta \lambda + 1) \alpha_1^2 \beta^2 \lambda \textbf{h} \textbf{h}'$.
    \item[(3, 1):] $ - \beta^2 \alpha_2 \lambda (\beta \lambda + 1) \boldsymbol{\Sigma} - \beta^2 \alpha_2 (\beta \lambda + 1) \{1 + \lambda \alpha_2 (\alpha_2 - \alpha_1) \} \textbf{h} \textbf{h}'$.
    \item[(3, 2):] $-(\alpha_2 - \alpha_1) (\beta \lambda + 1)^2 \alpha_2^2 \beta \textbf{h} \textbf{h}'$.
    \item[(3, 3):] $\alpha_2 \beta (\beta \lambda + 1)^2 \boldsymbol{\Sigma} + \alpha_2^4 \beta (\beta \lambda + 1)^2 \textbf{h} \textbf{h}'$.
    \item[(3, 4):] $-\beta^3 (\beta \lambda + 1)^2 \textbf{h} \textbf{h}'$.
    \item[(3, 5):] $\alpha_2 (\beta \lambda + 1) \beta^2 (1 + \lambda \alpha_2^3) \textbf{h} \textbf{h}'$.
    \item[(3, 6):] $-\beta^4 \lambda (\beta \lambda + 1) \textbf{h} \textbf{h}'$.
    \item[(4, 1):] $ - \beta^2 \alpha_1 \lambda (\beta \lambda + 1) \boldsymbol{\Sigma} - \beta^2 \alpha_1 (\beta \lambda + 1) \{1 + \lambda \alpha_1 (\alpha_1 - \alpha_2) \} \textbf{h} \textbf{h}'$.
    \item[(4, 2):] $(\alpha_2 - \alpha_1) \alpha_1^2 (\beta \lambda + 1)^2 \beta \textbf{h} \textbf{h}' $.
    \item[(4, 3):] $-\beta^3 (\beta \lambda + 1)^2 \textbf{h} \textbf{h}'$.
    \item[(4, 4):] $\alpha_1 \beta (\beta \lambda + 1)^2 \boldsymbol{\Sigma} + \alpha_1^4 \beta (\beta \lambda + 1)^2 \textbf{h} \textbf{h}'$.
    \item[(4, 5):] $-\beta^4 \lambda (\beta \lambda + 1) \textbf{h} \textbf{h}'$.
    \item[(4, 6):] $\alpha_1 (\beta \lambda + 1) \beta^2 (1 + \lambda \alpha_1^3) \textbf{h} \textbf{h}'$.
    \item[(5, 1):] $\beta^3 \alpha_2 \lambda \{ \lambda \alpha_2 (\alpha_1 - \alpha_2) - 2 \} \textbf{h} \textbf{h}'$.
    \item[(5, 2):] $-(\alpha_2 - \alpha_1) (\beta \lambda + 1) \alpha_2^2 \beta^2 \lambda \textbf{h} \textbf{h}'$.
    \item[(5, 3):] $\alpha_2 (\beta \lambda + 1) \beta^2 (1 + \lambda \alpha_2^3) \textbf{h} \textbf{h}'$.
    \item[(5, 4):] $-\beta^4 \lambda (\beta \lambda + 1) \textbf{h} \textbf{h}'$.
    \item[(5, 5):] $\alpha_2 \beta^3 \lambda (1 + \lambda \alpha_2^3) \textbf{h} \textbf{h}'$.
    \item[(5, 6):] $-\beta^5 \lambda^2 \textbf{h} \textbf{h}' $.
    \item[(6, 1):] $\beta^3 \alpha_1 \lambda \{ \lambda \alpha_1 (\alpha_2 - \alpha_1) - 2 \} \textbf{h} \textbf{h}'$.
    \item[(6, 2):] $(\alpha_2 - \alpha_1)(\beta \lambda + 1) \alpha_1^2 \beta^2 \lambda \textbf{h} \textbf{h}'$.
    \item[(6, 3):] $-\beta^4 \lambda (\beta \lambda + 1) \textbf{h} \textbf{h}'$.
    \item[(6, 4):] $\alpha_1 (\beta \lambda + 1) \beta^2 (1 + \lambda \alpha_1^3) \textbf{h} \textbf{h}'$.
    \item[(6, 5):] $-\beta^5 \lambda^2 \textbf{h} \textbf{h}' $.
    \item[(6, 6):] $\alpha_1 \beta^3 \lambda (1 + \lambda \alpha_1^3) \textbf{h} \textbf{h}'$.
\end{itemize}
\end{multicols}
Summing the previous terms, we obtain $\beta ( 1 + \beta \lambda ) \boldsymbol{\Sigma} + \beta^2 \textbf{h} \textbf{h}'$. Hence, the the limiting covariance of $ \sqrt{n}(\textbf{w}_n - \boldsymbol{\theta})$ is $ (\boldsymbol{\theta}' \textbf{h} + 1/\beta) \boldsymbol{\Sigma}^{-1} + \boldsymbol{\theta} \boldsymbol{\theta}' $.

Finally, the Jacobian of the map $\boldsymbol{\theta} \mapsto \boldsymbol{\theta}/\| \boldsymbol{\theta} \|$ is $(\| \boldsymbol{\theta} \|^2 \textbf{I}_p - \boldsymbol{\theta} \boldsymbol{\theta}')/\| \boldsymbol{\theta} \|^3$ and the delta method then implies that the scaled direction $\textbf{w}/\| \textbf{w} \|$ has the limiting covariance matrix,
\begin{align*}
    \boldsymbol{\Psi}_U &:= (\| \boldsymbol{\theta} \|^2 \textbf{I}_p - \boldsymbol{\theta} \boldsymbol{\theta}')/\| \boldsymbol{\theta} \|^3 \{ (\boldsymbol{\theta}' \textbf{h} + 1/\beta) \boldsymbol{\Sigma}^{-1} + \boldsymbol{\theta} \boldsymbol{\theta}' \} (\| \boldsymbol{\theta} \|^2 \textbf{I}_p - \boldsymbol{\theta} \boldsymbol{\theta}')/\| \boldsymbol{\theta} \|^3\\
    &= \left( \frac{\boldsymbol{\theta}'\boldsymbol{\Sigma} \boldsymbol{\theta}}{\| \boldsymbol{\theta} \|^2}  + \frac{1}{\beta \| \boldsymbol{\theta} \|^2} \right) \left( \textbf{I}_p - \frac{\boldsymbol{\theta} \boldsymbol{\theta}'}{\| \boldsymbol{\theta} \|^2} \right) \boldsymbol{\Sigma}^{-1} \left( \textbf{I}_p - \frac{\boldsymbol{\theta} \boldsymbol{\theta}'}{\| \boldsymbol{\theta} \|^2} \right).
\end{align*}

\end{proof}

Before proving results regarding the blind estimators, we establish two auxiliary lemmas.

\begin{lemma}\label{lem:full_rank_1_inversion}
    Let $\textbf{A} = \sum_{j=2}^p \lambda_j \textbf{w}_j \textbf{w}_j' + \textbf{w}_1 \textbf{w}_1' \textbf{C} \in \mathbb{R}^{p \times p}$, where $\lambda_2, \ldots , \lambda_p \in \mathbb{R}$, $\textbf{w}_1, \ldots, \textbf{w}_p$ constitute an orthonormal set of vectors and $\textbf{C} \in \mathbb{R}^{p \times p}$ is a symmetric positive definite matrix. Then $\textbf{A}$ is invertible and
    \begin{align*}
    \textbf{A}^{-1} = \textbf{B}^\dagger + (\textbf{w}_1' \textbf{C} \textbf{w}_1)^{-1} \textbf{w}_1 \textbf{w}_1' (\textbf{I}_p - \textbf{C} \textbf{B}^\dagger),
\end{align*}
where $\textbf{B}^\dagger = \sum_{j=2}^p \lambda^{-1}_j \textbf{w}_j \textbf{w}_j$ is the Moore-Penrose pseudoinverse of the matrix $\textbf{B} := \sum_{j=2}^p \lambda_j \textbf{w}_j \textbf{w}_j$.
\end{lemma}

\begin{proof}[Proof of Lemma \ref{lem:full_rank_1_inversion}]
    Observe first that $\textbf{B} \textbf{B}^\dagger = \textbf{B}^\dagger \textbf{B} = \textbf{I}_p - \textbf{w}_1 \textbf{w}_1'$. Then, we compute the product of the two matrices to be,
    \begin{align*}
        & (\textbf{B} +  \textbf{w}_1 \textbf{w}_1' \textbf{C})\left\{ \textbf{B}^\dagger + (\textbf{w}_1' \textbf{C} \textbf{w}_1)^{-1} \textbf{w}_1 \textbf{w}_1' (\textbf{I}_p - \textbf{C} \textbf{B}^\dagger) \right\} \\
        =& \textbf{I}_p - \textbf{w}_1 \textbf{w}_1' + \textbf{w}_1 \textbf{w}_1' \textbf{C}  \textbf{B}^\dagger + \textbf{w}_1 \textbf{w}_1' (\textbf{I}_p - \textbf{C} \textbf{B}^\dagger)\\
        =& \textbf{I}_p.
    \end{align*}
    The opposite product can be verified to equal identity in a similar manner, proving the claim.
\end{proof}

\begin{lemma}\label{lem:normal_moments}
    Let $\textbf{z} \sim \mathcal{N}_p (c \textbf{v}, \textbf{I}_p)$, for some $c \in \mathbb{R}$ and $\textbf{v} \in \mathbb{R}^p$. Then $\mathrm{E}\{ (\textbf{v}'\textbf{z})^k \textbf{z} \} = \| \textbf{v} \|^{-2} \mathrm{E}\{ (\textbf{v}'\textbf{z})^{k + 1} \} \textbf{v}$ and $\mathrm{E}\{ (\textbf{v}'\textbf{z})^k \textbf{z} \textbf{z}' \} = \mathrm{E}\{ (\textbf{v}'\textbf{z})^{k} \} (\textbf{I}_p - \| \textbf{v} \|^{-2} \textbf{v} \textbf{v}') +  \| \textbf{v} \|^{-4} \mathrm{E}\{ (\textbf{v}'\textbf{z})^{k + 2} \} \textbf{v} \textbf{v}'$.
\end{lemma}

\begin{proof}[Proof of Lemma \ref{lem:normal_moments}]
    The conditional distribution of $\textbf{z}$ given $\textbf{v}'\textbf{z}$ is
    \begin{align*}
        \textbf{z} \mid \textbf{v}'\textbf{z} = s \sim \mathcal{N}(\| \textbf{v} \|^{-2} s \textbf{v}, \textbf{I}_p - \| \textbf{v} \|^{-2} \textbf{v} \textbf{v}').
    \end{align*}
    Thus,
    \begin{align*}
        \mathrm{E}\{ (\textbf{v}'\textbf{z})^k \textbf{z} \} = \mathrm{E}[ \mathrm{E}\{ (\textbf{v}'\textbf{z})^k \textbf{z} \mid \textbf{v}'\textbf{z} \} ] = \mathrm{E}[  (\textbf{v}'\textbf{z})^k \mathrm{E}\{ \textbf{z} \mid \textbf{v}'\textbf{z} \} ] = \| \textbf{v} \|^{-2} \mathrm{E}\{ (\textbf{v}'\textbf{z})^{k + 1} \} \textbf{v}.
    \end{align*}
    The second claim is shown analogously and by using the fact that $\mathrm{E}(\textbf{z} \textbf{z}' \mid \textbf{v}' \textbf{z}) = \mathrm{Cov}(\textbf{z} \mid \textbf{v}' \textbf{z}) + \mathrm{E}(\textbf{z} \mid \textbf{v}' \textbf{z}) \mathrm{E}(\textbf{z}' \mid \textbf{v}' \textbf{z})$.
\end{proof}

\begin{proof}[Proof of Lemma \ref{lem:kurtosis_fisher}]
    The distribution of the projection $ \textbf{u}' \tilde{\textbf{x}} $ is
    \begin{align*}
        \textbf{u}' \tilde{\textbf{x}} \sim \alpha_1 \mathcal{N}_p(-\alpha_2 t , g) + \alpha_2 \mathcal{N}_p(\alpha_1 t, g),
    \end{align*}
    where $t := \textbf{u}'\textbf{h}$, $g := \textbf{u}'\boldsymbol{\Sigma}\textbf{u}$ and $\textbf{h} := \boldsymbol{\mu}_2 - \boldsymbol{\mu}_1$. By the moment formulas of univariate normal distribution, $\mathrm{E} \{ ( \textbf{u}' \tilde{\textbf{x}} )^2 \} = g + \beta t^2$, where $\beta := \alpha_1 \alpha_2$. Similarly, $ \mathrm{E} \{ ( \textbf{u}' \tilde{\textbf{x}} )^4 \} = \beta (\alpha_1^3 + \alpha_2^3) t^4 + 6 \beta t^2 g + 3 g^2$ which can be further simplified by noting that $\alpha_1^3 + \alpha_2^3 = 1 - 3 \beta$. Hence,
    \begin{align}\label{eq:kappa_expression}
        \{ \kappa(\textbf{u}) - 3 \}^2 = \beta^2 (1 - 6 \beta)^2 \frac{f^4}{(1 + \beta f)^{4}},
    \end{align}
    where $f := t^2/g \geq 0$.

    If $\alpha_1 \in \{ \delta_1, \delta_2 \}$, then $1 - 6 \beta = 0$ making $\{ \kappa(\textbf{u}) - 3 \}^2 = 0$. Assume then that $\alpha_1 \notin \{ \delta_1, \delta_2 \}$, implying that $(1 - 6 \beta)^2 > 0$. The derivative of the map $x \mapsto x^4/( 1 + \beta x)^{4}$ is $4x^3/( 1 + \beta x)^{5}$, showing that the map is strictly increasing in $(0, \infty)$. Hence, $\{ \kappa(\textbf{u}) - 3 \}^2$ is maximal when $f$ is at its largest. Now,
    \begin{align*}
        f = \frac{t^2}{g} = \left\{ \left( \frac{\boldsymbol{\Sigma}^{1/2}\textbf{u}}{\| \boldsymbol{\Sigma}^{1/2}\textbf{u} \|} \right)' \boldsymbol{\Sigma}^{-1/2} \textbf{h} \right\}^2,
    \end{align*}
    showing that, by the Cauchy-Schwarz inequality, $f$ is maximal if and only if $\boldsymbol{\Sigma}^{1/2}\textbf{u} \propto \boldsymbol{\Sigma}^{-1/2} \textbf{h}$, i.e., when $\textbf{u} = \pm \boldsymbol{\theta}/\| \boldsymbol{\theta} \|$ (where $\boldsymbol{\theta} = \boldsymbol{\Sigma}^{-1} \textbf{h}$).
\end{proof}

\begin{proof}[Proof of Theorem \ref{theo:asnorm_blind}]
    The objective functions are translation invariant, meaning that we may, without loss of generality, assume that $\mathrm{E}(\textbf{x}) = \textbf{0}$. This makes the marginal distribution of $\textbf{x}$ be $\textbf{x} \sim \alpha_1 \mathcal{N}_p(-\alpha_2 \textbf{h}, \boldsymbol{\Sigma}) + \alpha_2 \mathcal{N}_p(\alpha_1 \textbf{h}, \boldsymbol{\Sigma})$, where $\textbf{h} := \boldsymbol{\mu}_2 - \boldsymbol{\mu}_1$.

    The strong consistency of the estimator can be shown in the usual way by establishing that the objective function is strongly uniformly convergent in the compact parameter set $\mathbb{S}^{p - 1}$ (or, more precisely, in its subset where the sign of the estimator is fixed), that is,
    \begin{align}\label{eq:uniform_consistency}
        \sup_{\textbf{u} \in \mathbb{S}^{p - 1}} | \{ \kappa_n(\textbf{u}) - 3 \}^2 - \{ \kappa(\textbf{u}) - 3 \}^2 | \rightarrow 0, \quad \mathrm{a.s.}
    \end{align}
    For simplicity, we give the proof of the uniform convergence only in Theorem \ref{theo:asnorm_blind_2}, in the context of skewness (having lower moments than kurtosis), and similar (but lengthier) argments can be used to show \eqref{eq:uniform_consistency}.

    To show the limiting normality, note that the Largrangian corresponding to the optimization problem is $\ell_n(\textbf{u}) = \{ \kappa_{n}(\textbf{u}) - 3 \}^2 + \lambda_n (\textbf{u}' \textbf{u} - 1)$ where $\lambda_n$ is the Lagrangian multiplier. Using some matrix calculus, the corresponding gradient is seen to be
    \begin{align*}
        \nabla \ell_n (\textbf{u}) = \frac{8}{\tilde{s}_{n2}(\textbf{u})^3} \{ \kappa_{n}(\textbf{u}) - 3 \} \{ \tilde{s}_{n2}(\textbf{u}) \tilde{\textbf{m}}_{n3}(\textbf{u}) - \tilde{s}_{n4}(\textbf{u}) \tilde{\textbf{m}}_{n1}(\textbf{u}) \} - 2 \lambda_n \textbf{u},
    \end{align*}
    where $\tilde{s}_{nk}(\textbf{u}) := (1/n) \sum_i (\textbf{u}' \tilde{\textbf{x}}_i)^k$ and $\tilde{\textbf{m}}_{nk}(\textbf{u}) := (1/n) \sum_i (\textbf{u}' \tilde{\textbf{x}}_i)^{k} \tilde{\textbf{x}}_i$. The gradient vanishes at the (sign-adjusted) sample maximum $s_n \textbf{u}_n$ and multiplication of the gradient from the left with $s_n \textbf{u}_n'$ thus yields that $ 0 = s_n \textbf{u}_n' \nabla \ell_n (s_n \textbf{u}_n) = - 2 \lambda_n $, showing that $\lambda_n = 0$.

    We next work on the level of individual probability elements $\omega \in \Omega$. By Lemma \ref{lem:kurtosis_fisher}, LLN and the strong consistency of $s_n \textbf{u}_n$, there exists a probability one set $\mathcal{H}$ such that $s_n \textbf{u}_n \rightarrow \textbf{u}_0$ and $ \kappa_{n}(\textbf{u}) - 3 \rightarrow t \neq 0$ for all $\omega \in \mathcal{H}$. Thus, for each $\omega \in \mathcal{H}$, the maximizer $u_n$ satisfies, for $n$ large enough, the estimating equation $ \tilde{s}_{n2}(s_n \textbf{u}_n) \tilde{\textbf{m}}_{n3}(s_n \textbf{u}_n) - \tilde{s}_{n4}(s_n \textbf{u}_n) \tilde{\textbf{m}}_{n1}(s_n \textbf{u}_n) = \textbf{0}$. Using Lagrangian multipliers we can similarly show that the population maximizer $\textbf{u}_0 := \boldsymbol{\theta}/\| \boldsymbol{\theta} \|$ satisfies $ s_2(\textbf{u}_0) \textbf{m}_3(\textbf{u}_0) - s_4(\textbf{u}_0) \textbf{m}_1(\textbf{u}_0) = 0$, where $s_k(\textbf{u}) = \mathrm{E}\{ (\textbf{u}' \textbf{x})^k \}$ and $\textbf{m}_k(\textbf{u}) = \mathrm{E}\{ (\textbf{u}' \textbf{x})^k \textbf{x} \}$.

    Let $g_{n\kappa}: \mathbb{R}^p \setminus \{ \textbf{0} \} \to \mathbb{R}^{}$ be such that $g_{n\kappa}(\textbf{u}) =  \tilde{s}_{n2}(\textbf{u}) \tilde{\textbf{m}}_{n3}(\textbf{u}) - \tilde{s}_{n4}(\textbf{u}) \tilde{\textbf{m}}_{n1}(\textbf{u}) $. For each $\omega \in \mathcal{H}$, we have, for $n$ large enough, the Taylor expansion
    \begin{align*}
        g_{n\kappa}(s_n \textbf{u}_n) = g_{n\kappa}(\textbf{u}_0) + \nabla g_{n\kappa}(\textbf{u}_0)(s_n \textbf{u}_n - \textbf{u}_0) + \{ (s_n \textbf{u}_n - \textbf{u}_0)' \times  \nabla' \nabla g_{n\kappa}(\tilde{\textbf{u}}_n) \} (s_n \textbf{u}_n - \textbf{u}_0),
    \end{align*}
    where $ \nabla' \nabla g_{n\kappa}(\tilde{\textbf{u}}_n)  $ is the third order tensor of second derivatives of $g$, the symbol $\times$ denotes the vector-by-tensor multiplication (producing a matrix) and $\tilde{\textbf{u}}_n$ satisfies $\| \tilde{\textbf{u}}_n - \textbf{u}_0 \| \leq \| \textbf{u}_n - \textbf{u}_0 \|$, implying that $ \tilde{\textbf{u}}_n \rightarrow \textbf{u}_0 $. Multiplying the expansion by $\sqrt{n}$ and using the fact that $g_{n\kappa}(s_n \textbf{u}_n) = \textbf{0}$ gives that
    \begin{align}\label{eq:reordered_taylor}
        \{ (s_n \textbf{u}_n - \textbf{u}_0)' \times  \nabla' \nabla g_{n\kappa}(\tilde{\textbf{u}}_n) + \nabla g_{n\kappa}(\textbf{u}_0) \}\sqrt{n}(s_n \textbf{u}_n - \textbf{u}_0) = \sqrt{n}  g_{n\kappa}(\textbf{u}_0).
    \end{align}
    Now, the elements of $ \nabla' \nabla g_{n\kappa}(\textbf{u})$ are polynomials of the sample moments of $\tilde{\textbf{x}}_i$ and the elements of $\textbf{u}$ implying that, by LLN, $\nabla' \nabla g_{n\kappa}(\tilde{\textbf{u}}_n)$ converges to a constant and $(s_n \textbf{u}_n - \textbf{u}_0)' \times  \nabla' \nabla g_{n\kappa}(\tilde{\textbf{u}}_n)$ converges to zero, for any $\omega \in \mathcal{H}$. Now, by the unit lengths of $s_n \textbf{u}_n$ and $\textbf{u}_0$, we have $c_0 \textbf{h} (s_n \textbf{u}_n + \textbf{u}_0)' \sqrt{n}(s_n \textbf{u}_n - \textbf{u}_0) = 0 $, where $c_0 := (1/2) \{ 3 s_2 ( \textbf{u}_0)^2 - s_4(\textbf{u}_0) \} \| \boldsymbol{\theta} \| (\textbf{h}' \boldsymbol{\Sigma}^{-1} \textbf{h})^{-1}$ (the inclusion of the constant $c_0$ simplifies things later on). Summing this with equation \eqref{eq:reordered_taylor} gives
    \begin{align*}
        \{ (s_n \textbf{u}_n - \textbf{u}_0)' \times  \nabla' \nabla g_{n\kappa}(\tilde{\textbf{u}}_n) + \nabla g_{n\kappa}(\textbf{u}_0) + c_0 \textbf{h} (s_n \textbf{u}_n + \textbf{u}_0)' \}\sqrt{n}(s_n \textbf{u}_n - \textbf{u}_0) = \sqrt{n}  g_{n\kappa}(\textbf{u}_0).
    \end{align*}
    Assume now for a moment that $\nabla g_{n\kappa}(\textbf{u}_0) + c_0 \textbf{h} (s_n \textbf{u}_n + \textbf{u}_0)'$ converges to a full-rank matrix $\textbf{G} \in \mathbb{R}^{p \times p}$. Then, for $n$ large enough, we have,
    \begin{align}
        \sqrt{n}(s_n \textbf{u}_n - \textbf{u}_0) = \{ (s_n \textbf{u}_n - \textbf{u}_0)' \times  \nabla' \nabla g_{n\kappa}(\tilde{\textbf{u}}_n) + \nabla g_{n\kappa}(\textbf{u}_0) + c_0 \textbf{h} (s_n \textbf{u}_n + \textbf{u}_0)' \}^{-1} \sqrt{n}  g_{n\kappa}(\textbf{u}_0).
    \end{align}
    Hence, assuming further that we have $\sqrt{n}  g_{n\kappa}(\textbf{u}_0) \rightsquigarrow \mathcal{N}_p(\textbf{0}, \boldsymbol{\Pi})$, then the limiting distribution of $s_n \textbf{u}_n$ is, by Slutsky's theorem,
    \begin{align}\label{eq:limiting_dist_kurtosis}
        \sqrt{n}(s_n \textbf{u}_n - \textbf{u}_0) \rightsquigarrow \mathcal{N}_p \{ \textbf{0}, \textbf{G}^{-1} \boldsymbol{\Pi} (\textbf{G}^{-1})' \}.
    \end{align}
    Thus, to complete the proof, we next derive expressions for $\textbf{G}$ and $\boldsymbol{\Pi}$ (and show that the former has indeed full rank).

    The Jacobian of $g_{n\kappa}$ is,
    \begin{align*}
        \nabla g_{n\kappa}(\textbf{u}_0) = 2 \tilde{\textbf{m}}_{n1}(\textbf{u}_0) \tilde{\textbf{m}}_{n3}(\textbf{u}_0)' + 3 \tilde{s}_{n2}(\textbf{u}_0) \tilde{\textbf{G}}_{n2}(\textbf{u}_0) - 4 \tilde{\textbf{m}}_{n3}(\textbf{u}_0) \tilde{\textbf{m}}_{n1}(\textbf{u}_0)' - \tilde{s}_{n4}(\textbf{u}_0) \tilde{\textbf{G}}_{n0}(\textbf{u}_0),
    \end{align*}
    where $\tilde{\textbf{G}}_{nk}(\textbf{u}) := (1/n) \sum_i (\textbf{u}' \tilde{\textbf{x}}_i)^k  \tilde{\textbf{x}}_i \tilde{\textbf{x}}_i'$. Thus, by LLN and using the population level estimating equation, ${s}_{2}(\textbf{u}_0) {\textbf{m}}_{3}(\textbf{u}_0) = {s}_{4}(\textbf{u}_0) {\textbf{m}}_{1}(\textbf{u}_0)$, we get
    \begin{align}\label{eq:hessian_lln}
        \nabla g_{n\kappa}(\textbf{u}_0) \rightarrow_p = -2 \frac{s_4(\textbf{u}_0)}{s_2(\textbf{u}_0)} \textbf{m}_{1}(\textbf{u}_0) \textbf{m}_{1}(\textbf{u}_0)' + 3 s_{2}(\textbf{u}_0) \textbf{G}_{2}(\textbf{u}_0) - s_{4}(\textbf{u}_0) \textbf{G}_{0}(\textbf{u}_0),
    \end{align}
    where $\textbf{G}_k(\textbf{u}) := \mathrm{E}\{ (\textbf{u}' \textbf{x})^k \textbf{x} \textbf{x}' \}$. Denote next $\tau := \textbf{h}' \boldsymbol{\Sigma}^{-1} \textbf{h}$.


    To compute the moments $\textbf{m}_{k}(\textbf{u}_0)$ and $\textbf{G}_{k}(\textbf{u}_0)$, we use Lemma \ref{lem:normal_moments}. The former satisfies $ \boldsymbol{\Sigma}^{-1/2} \textbf{m}_{k}(\textbf{u}_0) = \| \boldsymbol{\theta} \|^{-k} \mathrm{E}\{ (\textbf{v}' \textbf{z})^k \textbf{z} \}$, where $\textbf{v} := \boldsymbol{\Sigma}^{-1/2} \textbf{h}$ and $\textbf{z} \sim \alpha_1 \mathcal{N}_p(-\alpha_2 \textbf{h}, \textbf{I}_p) + \alpha_2 \mathcal{N}_p(\alpha_1 \textbf{h}, \textbf{I}_p)$. Denoting the components of the mixture by $\textbf{z}_1$ and $\textbf{z}_2$, we have, by the first part of Lemma \ref{lem:normal_moments}, for $\textbf{z}_1$ that $\mathrm{E}\{ (\textbf{v}' \textbf{z}_1)^k \textbf{z}_1 \} = \| \textbf{v} \|^{-2} \mathrm{E}\{ (\textbf{v}'\textbf{z}_1)^{k + 1} \} \textbf{v}$, and similarly for $\textbf{z}_2$. Hence, $\boldsymbol{\Sigma}^{-1/2} \textbf{m}_{k}(\textbf{u}_0) = \| \boldsymbol{\theta} \|^{-k} \| \textbf{v} \|^{-2} \mathrm{E}\{ (\textbf{v}'\textbf{z})^{k + 1} \} \textbf{v}$. Finally, since $s_{k}(\textbf{u}_0) = \| \boldsymbol{\theta} \|^{-k} \mathrm{E}\{ (\textbf{v}'\textbf{z})^{k} \}$, we get
    \begin{align}\label{eq:formula_m}
        \textbf{m}_{k}(\textbf{u}_0) = \| \boldsymbol{\theta} \| \| \textbf{v} \|^{-2} s_{k + 1}(\textbf{u}_0) \boldsymbol{\Sigma}^{1/2} \textbf{v} = \| \boldsymbol{\theta} \| \tau^{-1} s_{k + 1}(\textbf{u}_0) \textbf{h}.
    \end{align}

    For $\textbf{G}_{k}(\textbf{u}_0)$, we have, using the same notation, that $$
    \boldsymbol{\Sigma}^{-1/2} \textbf{G}_{k}(\textbf{u}_0) \boldsymbol{\Sigma}^{-1/2} = \| \boldsymbol{\theta} \|^{-k} \mathrm{E}\{ (\textbf{v}' \textbf{z})^k \textbf{z} \textbf{z}' \}.
    $$
    The second part of Lemma \ref{lem:normal_moments} then shows that
    \begin{align}\label{eq:formula_G}
    \begin{split}
         \textbf{G}_{k}(\textbf{u}_0) &= \| \boldsymbol{\theta} \|^{-k} \boldsymbol{\Sigma}^{1/2} [ \mathrm{E}\{ (\textbf{v}'\textbf{z})^{k} \} (\textbf{I}_p - \| \textbf{v} \|^{-2} \textbf{v} \textbf{v}') +  \| \textbf{v} \|^{-4} \mathrm{E}\{ (\textbf{v}'\textbf{z})^{k + 2} \} \textbf{v} \textbf{v}' ] \boldsymbol{\Sigma}^{1/2} \\
         &= \| \boldsymbol{\theta} \|^{-k} [ \| \boldsymbol{\theta} \|^k s_{k}(\textbf{u}_0) (\boldsymbol{\Sigma} - \tau^{-1} \textbf{h} \textbf{h}') + \tau^{-2} \| \boldsymbol{\theta} \|^{k+2} s_{k+2}(\textbf{u}_0) \textbf{h} \textbf{h}' ] \\
         &=  s_{k}(\textbf{u}_0) \boldsymbol{\Sigma} + \tau^{-1} \{ \tau^{-1} \| \boldsymbol{\theta} \|^{2} s_{k+2}(\textbf{u}_0) -  s_{k}(\textbf{u}_0) \} \textbf{h} \textbf{h}'.
    \end{split}
    \end{align}
    Plugging in the expressions to \eqref{eq:hessian_lln}, we get $ \nabla g_{n\kappa}(\textbf{u}_0) \rightarrow_p (3 {s}_2^2 - {s}_4) \boldsymbol{\Sigma}^{1/2} (\textbf{I}_p - \textbf{w} \textbf{w}') \boldsymbol{\Sigma}^{1/2} $, where $\textbf{w} := \boldsymbol{\Sigma}^{-1/2} \textbf{h}/ \| \boldsymbol{\Sigma}^{-1/2} \textbf{h} \|$ and $s_k \equiv s_{k}(\textbf{u}_0)$. Moreover, we also have $ c_0 \textbf{h} (s_n \textbf{u}_n + \textbf{u}_0)' \rightarrow_p (3 {s}_2^2 - {s}_4) \boldsymbol{\Sigma}^{1/2} \textbf{w} \textbf{w}' \boldsymbol{\Sigma}^{-1} \boldsymbol{\Sigma}^{1/2} $. Now $\textbf{G}$ is the sum of these two, giving,
    \begin{align*}
        \textbf{G} = (3 {s}_2^2 - {s}_4) \boldsymbol{\Sigma}^{1/2} (\textbf{I}_p - \textbf{w} \textbf{w}' + \textbf{w} \textbf{w}' \boldsymbol{\Sigma}^{-1}) \boldsymbol{\Sigma}^{1/2}.
    \end{align*}
    The invertibility of $\textbf{G}$ now follows from Lemma \ref{lem:full_rank_1_inversion}, which also gives
    \begin{align*}
        (\textbf{I}_p - \textbf{w} \textbf{w}' + \textbf{w} \textbf{w}' \boldsymbol{\Sigma}^{-1})^{-1} &= \textbf{I}_p + (\textbf{w}' \boldsymbol{\Sigma}^{-1} \textbf{w})^{-1} \textbf{w} \textbf{w}' (\textbf{I}_p - \boldsymbol{\Sigma}^{-1})\\
        &= \textbf{I}_p + \| \boldsymbol{\theta} \|^{-2} \boldsymbol{\Sigma}^{-1/2} \textbf{h} \textbf{h}' \boldsymbol{\Sigma}^{-1/2} (\textbf{I}_p - \boldsymbol{\Sigma}^{-1}).
    \end{align*}
    Finally, this makes the inverse of $\textbf{G}$ be,
    \begin{align*}
        \textbf{G}^{-1} = \frac{1}{3 {s}_2^2 - {s}_4} \left\{ \boldsymbol{\Sigma}^{-1} + \frac{1}{\| \boldsymbol{\theta} \|^2} \boldsymbol{\theta} \boldsymbol{\theta}' (\textbf{I}_p - \boldsymbol{\Sigma}^{-1})  \right\}.
    \end{align*}
    The fact that $3 {s}_2^2 - {s}_4 \neq 0$ follows from the formulas for $s_k$ given later in the proof.

    We next obtain the limiting distribution of $$
    \sqrt{n}  g_{n\kappa}(\textbf{u}_0) = \sqrt{n} \{ \tilde{s}_{n2}(\textbf{u}_0) \tilde{\textbf{m}}_{n3}(\textbf{u}_0) - \tilde{s}_{n4}(\textbf{u}_0) \tilde{\textbf{m}}_{n1}(\textbf{u}_0) \}.
    $$
    Define non-centered counterparts for the sample moments as ${s}_{nk}(\textbf{u}) := (1/n) \sum_i (\textbf{u}' {\textbf{x}}_i)^k$ and ${\textbf{m}}_{nk}(\textbf{u}) := (1/n) \sum_i (\textbf{u}' {\textbf{x}}_i)^{k} {\textbf{x}}_i$. Then, LLN together with the calculus of $o_p(1)$ and $O_p(1)$ sequences shows that $\tilde{s}_{n2}(\textbf{u}) = {s}_{n2}(\textbf{u}) + o_p(1/\sqrt{n})$ and $\tilde{\textbf{m}}_{n1}(\textbf{u}) = {\textbf{m}}_{n1}(\textbf{u}) + o_p(1/\sqrt{n})$. However, the same equivalence does not hold for the terms $\tilde{\textbf{m}}_{n3}(\textbf{u})$ and $\tilde{s}_{n4}(\textbf{u})$ but we instead have
    \begin{align*}
        \tilde{\textbf{m}}_{n3}(\textbf{u}) = {\textbf{m}}_{n3}(\textbf{u}) - 3 {s}_{n1}(\textbf{u}) \textbf{m}_{2}(\textbf{u}) - {s}_{3}(\textbf{u}) {\textbf{m}}_{n0}(\textbf{u}) + o_p(1/\sqrt{n}),
    \end{align*}
    and
    \begin{align*}
        \tilde{s}_{n4}(\textbf{u}) = {s}_{n4}(\textbf{u}) - 4 {s}_{3}(\textbf{u}) {s}_{n1}(\textbf{u}) + o_p(1/\sqrt{n}).
    \end{align*}
    Using these, we expand $\sqrt{n}  g_{n\kappa}(\textbf{u}_0)$ to be (dropping $\textbf{u}_0$ from the notation),
    \begin{align}\label{eq:g_kappa_expansion}
        \sqrt{n} g_{n\kappa} =& \sqrt{n} ( {s}_{n2} - s_2 ) \textbf{m}_3 + s_2 \sqrt{n} (\textbf{m}_{n3} - \textbf{m}_3) - \sqrt{n} ( {s}_{n4} - s_4 ) \textbf{m}_1 - s_4 \sqrt{n} (\textbf{m}_{n1} - \textbf{m}_1)\nonumber\\
        &+ (4 s_3 \textbf{m}_1 - 3 s_2 \textbf{m}_2) \sqrt{n} s_{n1} - s_2 s_3 \sqrt{n} \textbf{m}_{n0}.
    \end{align}
    Hence, by CLT, $\sqrt{n} g_{n\kappa}$ has a limiting normal distribution with the covariance matrix,
    \begin{align*}
        \boldsymbol{\Pi} = \mathrm{Cov}\{ (\textbf{u}_0'\textbf{x})^2 \textbf{m}_3 + s_2 (\textbf{u}_0'\textbf{x})^3 \textbf{x} -  (\textbf{u}_0'\textbf{x})^4 \textbf{m}_1 - s_4 (\textbf{u}_0'\textbf{x}) \textbf{x} + (4 s_3 \textbf{m}_1 - 3 s_2 \textbf{m}_2) (\textbf{u}_0'\textbf{x}) - s_2 s_3 \textbf{x} \}.
    \end{align*}
    This matrix consists of 36 terms, which we next present and simplify using \eqref{eq:formula_m} and \eqref{eq:formula_G}. We use the notation $\psi = \| \boldsymbol{\theta} \| \tau^{-1}$. Note that $s_1 = 0$, $\textbf{m}_0 = \textbf{0}$ and $\textbf{f} := 4 s_3 \textbf{m}_1 - 3 s_2 \textbf{m}_2 = \psi s_2 s_3 \textbf{h} $.
    \begin{multicols}{2}
    \begin{itemize}
        \item[(1, 1):] $(s_4 - s_2^2) \textbf{m}_3 \textbf{m}_3' = \psi^2 s_4^2 (s_4 - s_2^2) \textbf{h} \textbf{h}'$.
        \item[(1, 2):] $s_2 (\textbf{m}_3 \textbf{m}_5' - s_2 \textbf{m}_3 \textbf{m}_3') = \psi^2 s_2 s_4 (s_6 - s_2 s_4) \textbf{h} \textbf{h}'$.
        \item[(1, 3):] $-(s_6 - s_2 s_4) \textbf{m}_3 \textbf{m}_1' = - \psi^2 s_2 s_4 (s_6 - s_2 s_4) \textbf{h} \textbf{h}'$.
        \item[(1, 4):] $-s_4(\textbf{m}_3 \textbf{m}_3' - s_2 \textbf{m}_3 \textbf{m}_1') = - \psi^2 s_4^2 (s_4 - s_2^2) \textbf{h} \textbf{h}'$.
        \item[(1, 5):] $s_3 \textbf{m}_3 \textbf{f}' = \psi^2 s_2 s_3^2 s_4 \textbf{h} \textbf{h}' $.
        \item[(1, 6):] $-s_2 s_3 \textbf{m}_3 \textbf{m}_2' = -\psi^2 s_2 s_3^2 s_4 \textbf{h} \textbf{h}'$.
        \item[(2, 1):] $\psi^2 s_2 s_4 (s_6 - s_2 s_4) \textbf{h} \textbf{h}'$.
        \item[(2, 2):] $ s_2^2 (\textbf{G}_6 - \textbf{m}_3 \textbf{m}_3') = s_2^2 s_6 \boldsymbol{\Sigma} + s_2^2 \{ \psi^2 (s_8 - s_4^2) - \tau^{-1} s_6 \} \textbf{h} \textbf{h}'$.
        \item[(2, 3):] $-s_2 (\textbf{m}_7 \textbf{m}_1' - s_4 \textbf{m}_3 \textbf{m}_1') = -\psi^2 s_2^2 (s_8 - s_4^2) \textbf{h} \textbf{h}'$.
        \item[(2, 4):] $-s_2 s_4(\textbf{G}_4 - \textbf{m}_3 \textbf{m}_1') =  -s_2 s_4^2 \boldsymbol{\Sigma} - s_2 s_4 \{ \psi^2 (s_6 - s_2 s_4) - \tau^{-1} s_4 \} \textbf{h} \textbf{h}'$.
        \item[(2, 5):] $ s_2 \textbf{m}_4 \textbf{f}' = \psi^2 s_2^2 s_3 s_5 \textbf{h} \textbf{h}' $.
        \item[(2, 6):] $- s_2^2 s_3 \textbf{G}_3 = - s_2^2 s_3^2 \boldsymbol{\Sigma} - s_2^2 s_3 \{ \psi^2 s_5 - \tau^{-1} s_3 \} \textbf{h} \textbf{h}'$.
        \item[(3, 1):] $- \psi^2 s_2 s_4 (s_6 - s_2 s_4) \textbf{h} \textbf{h}'$.
        \item[(3, 2):] $- \psi^2 s_2^2 (s_8 - s_4^2) \textbf{h} \textbf{h}'$.
        \item[(3, 3):] $ (s_8 - s_4^2) \textbf{m}_1 \textbf{m}_1' = \psi^2 s_2^2 (s_8 - s_4^2) \textbf{h}\textbf{h}' $.
        \item[(3, 4):] $s_4 (\textbf{m}_1 \textbf{m}_5' - s_4 \textbf{m}_1 \textbf{m}_1') = \psi^2 s_2 s_4 (s_6 - s_2 s_4) \textbf{h} \textbf{h}'$.
        \item[(3, 5):] $-s_5 \textbf{m}_1 \textbf{f}' = - \psi^2 s_2^2 s_3 s_5 \textbf{h} \textbf{h}'$.
        \item[(3, 6):] $s_2 s_3 \textbf{m}_1 \textbf{m}_4' = \psi^2 s_2^2 s_3 s_5 \textbf{h} \textbf{h}' $.
        \item[(4, 1):] $ - \psi^2 s_4^2 (s_4 - s_2^2) \textbf{h} \textbf{h}'$.
        \item[(4, 2):] $ - s_2 s_4^2 \boldsymbol{\Sigma} - s_2 s_4 \{ \psi^2 (s_6 - s_4 s_2) - \tau^{-1} s_4 \} \textbf{h} \textbf{h}'$.
        \item[(4, 3):] $\psi^2 s_2 s_4 (s_6 - s_2 s_4) \textbf{h} \textbf{h}'$.
        \item[(4, 4):] $s_4^2 (\textbf{G}_2 - \textbf{m}_1 \textbf{m}_1') = s_2 s_4^2 \boldsymbol{\Sigma} + s_4^2 \{ \psi^2 (s_4 - s_2^2) - \tau^{-1} s_2 \} \textbf{h} \textbf{h}'$.
        \item[(4, 5):] $-s_4 \textbf{m}_2 \textbf{f}' = - \psi^2 s_2 s_3^2 s_4 \textbf{h} \textbf{h}'$.
        \item[(4, 6):] $s_2 s_3 s_4 \textbf{G}_1 = \psi^2 s_2 s_3^2 s_4 \textbf{h} \textbf{h}'$.
        \item[(5, 1):] $\psi^2 s_2 s_3^2 s_4 \textbf{h} \textbf{h}' $.
        \item[(5, 2):] $\psi^2 s_2^2 s_3 s_5 \textbf{h} \textbf{h}' $.
        \item[(5, 3):] $- \psi^2 s_2^2 s_3 s_5 \textbf{h} \textbf{h}'$.
        \item[(5, 4):] $- \psi^2 s_2 s_3^2 s_4 \textbf{h} \textbf{h}'$.
        \item[(5, 5):] $s_2 \textbf{f} \textbf{f}' = \psi^2 s_2^3 s_3^2 \textbf{h} \textbf{h}' $.
        \item[(5, 6):] $ -s_2 s_3 \textbf{f} \textbf{m}_1' = -\psi^2 s_2^3 s_3^2 \textbf{h} \textbf{h}' $.
        \item[(6, 1):] $-\psi^2 s_2 s_3^2 s_4 \textbf{h} \textbf{h}'$.
        \item[(6, 2):] $- s_2^2 s_3^2 \boldsymbol{\Sigma} - s_2^2 s_3 \{ \psi^2 s_5 - \tau^{-1} s_3 \} \textbf{h} \textbf{h}'$.
        \item[(6, 3):] $\psi^2 s_2^2 s_3 s_5 \textbf{h} \textbf{h}' $.
        \item[(6, 4):] $\psi^2 s_2 s_3^2 s_4 \textbf{h} \textbf{h}'$.
        \item[(6, 5):] $-\psi^2 s_2^3 s_3^2 \textbf{h} \textbf{h}' $.
        \item[(6, 6):] $s_2^2 s_3^2 \textbf{G}_0 = s_2^2 s_3^2 \boldsymbol{\Sigma} + s_2^2 s_3^2 (\psi^2 s_2 - \tau^{-1}) \textbf{h} \textbf{h}'$.
    \end{itemize}
    \end{multicols}
    Summation of the previous 36 terms results in $\boldsymbol{\Pi} = s_2 (s_2 s_6 - s_2 s_3^2 - s_4^2) (\boldsymbol{\Sigma} - \tau^{-1} \textbf{h} \textbf{h}')$. Thus, from the reasoning preceding \eqref{eq:limiting_dist_kurtosis}, we have that $\sqrt{n}(s_n \textbf{u}_n - \textbf{u}_0) $ has a limiting normal distribution and with the covariance matrix $\boldsymbol{\Psi}_\kappa = \textbf{G}^{-1} \boldsymbol{\Pi} (\textbf{G}^{-1})'$. Plugging now in the values of $\textbf{G}$ and $\boldsymbol{\Pi}$ and simplifying, we obtain,
    \begin{align}\label{eq:final_covariance_kurtosis}
        \boldsymbol{\Psi}_\kappa = \frac{s_2 (s_2 s_6 - s_2 s_3^2 - s_4^2)}{(3 {s}_2^2 - {s}_4)^2} \left( \textbf{I}_p - \frac{\boldsymbol{\theta} \boldsymbol{\theta}'}{\| \boldsymbol{\theta} \|^2} \right) \boldsymbol{\Sigma}^{-1} \left( \textbf{I}_p - \frac{\boldsymbol{\theta} \boldsymbol{\theta}'}{\| \boldsymbol{\theta} \|^2} \right).
    \end{align}
    Now, recall that $s_k \equiv s_k(\textbf{u}_0) = \mathrm{E}\{ (\textbf{u}_0' \textbf{x})^k \} = \| \boldsymbol{\theta} \|^{-k} \mathrm{E}\{ (\boldsymbol{\theta}' \textbf{x})^k \}$ where $\boldsymbol{\theta}' \textbf{x} \sim \alpha_1 \mathcal{N}_1(-\alpha_2 \tau, \tau) + \alpha_2 \mathcal{N}_1(\alpha_1 \tau, \tau)$ and $\tau = \boldsymbol{\theta}' \textbf{h} = \boldsymbol{\theta}' \boldsymbol{\Sigma} \boldsymbol{\theta}$. Using the moment formulas for univariate normal distribution we now obtain that
    \begin{align*}
        s_2 = \| \boldsymbol{\theta} \|^{-2} \tau (1 + \beta \tau), \,\, s_3 = \| \boldsymbol{\theta} \|^{-3} (\alpha_1 - \alpha_2) \beta \tau^3, \,\, s_4 = \| \boldsymbol{\theta} \|^{-4} \tau^2 \{\beta \tau^2(1 - 6 \beta) + 3(1 + \beta \tau)^2\},
    \end{align*}
    and
    \begin{align*}
        s_6 = \| \boldsymbol{\theta} \|^{-6} \tau^3 \{ \beta (1 - 5 \beta + 5 \beta^2)  \tau^3 + 15 \beta (1 - 3 \beta) \tau^2 + 45 \beta \tau + 15\}
    \end{align*}
    where $\beta := \alpha_1 \alpha_2$ and we have used the identities $ \alpha_1^3 + \alpha_2^3 = 1 - 3 \beta$ and $ \alpha_1^5 + \alpha_2^5 = 1 - 5\beta + 5\beta^2$. Plugging these in to \eqref{eq:final_covariance_kurtosis} and simplifying (using $(\alpha_1 - \alpha_2)^2 = 1 - 4 \beta$), shows that the constant in front is
    \begin{align*}
        \frac{(1 + \beta \tau)(6 + 24 \beta \tau + 9 \beta \tau^2 + \beta \tau^3 - 18 \beta^2 \tau^2 - 3 \beta^2 \tau^3)}{\tau^3 \beta^2 (6 \beta - 1)^2 \| \boldsymbol{\theta} \|^2}
    \end{align*}
\end{proof}

\begin{proof}[Proof of Lemma \ref{lem:skewness_fisher}]
    The distribution of the projection $ \textbf{u}' \tilde{\textbf{x}} $ is
    \begin{align*}
        \textbf{u}' \tilde{\textbf{x}} \sim \alpha_1 \mathcal{N}_p(-\alpha_2 t , g) + \alpha_2 \mathcal{N}_p(\alpha_1 t, g),
    \end{align*}
    where $t := \textbf{u}'\textbf{h}$, $g := \textbf{u}'\boldsymbol{\Sigma}\textbf{u}$ and $\textbf{h} := \boldsymbol{\mu}_2 - \boldsymbol{\mu}_1$. By the moment formulas of univariate normal distribution, $\mathrm{E} \{ ( \textbf{u}' \tilde{\textbf{x}} )^2 \} = g + \beta t^2$, where $\beta := \alpha_1 \alpha_2$. Similarly, $ \mathrm{E} \{ ( \textbf{u}' \tilde{\textbf{x}} )^3 \} = (\alpha_1 - \alpha_2) \beta t^3 $. Hence,
    \begin{align*}
        \gamma(\textbf{u}) = \beta^2 (1 - 4 \beta) \frac{f^3}{(1 + \beta f)^{3}},
    \end{align*}
    where $f := t^2/g \geq 0$. Now, if $\alpha_1 = \alpha_2 = 1/2$, then clearly $\gamma(\textbf{u})^2 = 0$. If $\alpha_1 \neq 1/2$, the derivative of the map $x \mapsto x^3/( 1 + \beta x)^{3}$ is $3x^2/( 1 + \beta x)^{4}$, showing that the map is strictly increasing outside of the origin. The conclusion now follows as in the proof of Lemma \ref{lem:kurtosis_fisher}.

\end{proof}

\begin{proof}[Proof of Theorem \ref{theo:asnorm_blind_2}]
    The strong consistency follows as soon as we show the strong uniform consistency,
    \begin{align}\label{eq:uniform_consistency_2}
        \sup_{\textbf{u} \in \mathbb{S}^{p - 1}} | \gamma_n(\textbf{u})^2 -  \gamma(\textbf{u})^2 | \rightarrow 0, \quad \mathrm{a.s.}.
    \end{align}
    By Theorem 2 and Lemma 1 in \cite{andrews1992generic}, \eqref{eq:uniform_consistency_2} holds if, 1) the parameter space is compact, 2) we have $\gamma_n(\textbf{u})^2 \rightarrow \gamma(\textbf{u})^2$, a.s., for all $\textbf{u} \in \mathbb{S}^{p - 1}$ (this holds by LLN and the continuous mapping theorem), 3) $\gamma^2$ is uniformly continuous in $\textbf{u}$ and, 4) $\gamma_n^2$ is Lipschitz continuous in the sense that $| \gamma_n(\textbf{u}_1)^2 -  \gamma_n(\textbf{u}_2)^2 | \leq K_n \| \textbf{u}_1 - \textbf{u}_2 \|$ for all $\textbf{u}_1, \textbf{u}_2 \in \mathbb{S}^{p - 1}$ and some random variable $K_n$ converging almost surely to a constant.

    We now verify condition 4) above. Using the notation of the proof of Theorem \ref{theo:asnorm_blind}, we have
    \begin{align*}
        | \gamma_n(\textbf{u}_1)^2 -  \gamma_n(\textbf{u}_2)^2 | =& \left| \frac{\tilde{s}_{n3}^2(\textbf{u}_1)}{\tilde{s}_{n2}^3(\textbf{u}_1)} - \frac{\tilde{s}_{n3}^2(\textbf{u}_2)}{\tilde{s}_{n2}^3(\textbf{u}_2)} \right|\\
        &\leq \frac{| \tilde{s}_{n3}^2(\textbf{u}_1) - \tilde{s}_{n3}^2(\textbf{u}_2) | \tilde{s}_{n2}^3(\textbf{u}_2) - \tilde{s}_{n3}^2(\textbf{u}_2) | \tilde{s}_{n2}^3(\textbf{u}_1) - \tilde{s}_{n2}^3(\textbf{u}_2) |}{\tilde{s}_{n2}^3(\textbf{u}_1) \tilde{s}_{n2}^3(\textbf{u}_2)}.
    \end{align*}
    Now, $\tilde{s}_{n2}(\textbf{u})$ is, for all $\textbf{u} \in \mathbb{S}^{p - 1}$, lower bounded by the smallest eigenvalue of the sample covariance matrix, which by the continuity of the eigenvalues and the positive-definiteness of the covariance matrix converges almost surely to a positive constant. Moreover, we have
    \begin{align*}
        |\tilde{s}_{n3}(\textbf{u})| \leq \frac{1}{n} \sum_{i=1}^n | \textbf{u}' \tilde{\textbf{x}}_i |^3 \leq \frac{1}{n} \sum_{i=1}^n \| \textbf{x}_i - \bar{\textbf{x}} \|^3 \leq \frac{1}{n} \sum_{i=1}^n (\| \textbf{x}_i \| + \| \bar{\textbf{x}} \|)^3,
    \end{align*}
    which converges, by LLN, almost surely to a constant, and similar result can be shown for $|\tilde{s}_{n2}(\textbf{u})|$. Finally,
    \begin{align*}
        | \tilde{s}_{n3}^2(\textbf{u}_1) - \tilde{s}_{n3}^2(\textbf{u}_2) | \leq | \tilde{s}_{n3}(\textbf{u}_1) - \tilde{s}_{n3}(\textbf{u}_2) | \frac{2}{n} \sum_{i=1}^n \| \textbf{x}_i - \bar{\textbf{x}} \|^3,
    \end{align*}
    and
    \begin{align*}
        | \tilde{s}_{n3}(\textbf{u}_1) - \tilde{s}_{n3}(\textbf{u}_2) | &\leq \frac{1}{n} \sum_{i=1}^n | (\textbf{u}_1 - \textbf{u}_2)' \tilde{\textbf{x}}_i | | (\textbf{u}_1' \tilde{\textbf{x}}_i)^2 + \textbf{u}_1' \tilde{\textbf{x}}_i \textbf{u}_2' \tilde{\textbf{x}}_i + (\textbf{u}_2' \tilde{\textbf{x}}_i)^2 |\\
        &\leq \| \textbf{u}_1 - \textbf{u}_2 \| \frac{3}{n} \sum_{i=1}^n \| {\textbf{x}}_i - \bar{\textbf{x}} \|^3,
    \end{align*}
    and putting everything above together, we conclude that the Lipschitz continuity 4) holds. What remains to be verified is then condition 3), which can be shown similarly to 4) after recalling that Lipschitz continuity implies uniform continuity. Hence, the strong consistency of the estimator follows.

    Also the proof of the limiting normality has exactly the same steps as in the proof of Theorem \ref{theo:asnorm_blind} and we only provide the key steps and expressions, using the same notation as in the proof of Theorem \ref{theo:asnorm_blind}. The gradient of $\gamma_n$ is
    \begin{align*}
        \nabla \gamma_n (\textbf{u}) = \frac{6}{\tilde{s}_{n2}(\textbf{u})^{5/2}} \gamma_n(\textbf{u}) \{ \tilde{s}_{n2}(\textbf{u}) \tilde{\textbf{m}}_{n2}(\textbf{u}) - \tilde{s}_{n3}(\textbf{u}) \tilde{\textbf{m}}_{n1}(\textbf{u}) \},
    \end{align*}
    leading to the estimating equation $g_{n\gamma}(\textbf{u}_n) = \textbf{0}$, for $ g_{n\gamma}(\textbf{u}) := \tilde{s}_{n2}(\textbf{u}) \tilde{\textbf{m}}_{n2}(\textbf{u}) - \tilde{s}_{n3}(\textbf{u}) \tilde{\textbf{m}}_{n1}(\textbf{u}) $. The Jacobian of $g_{n\gamma}$ at $\textbf{u}_0$ satisfies (after simplification via the estimating equation)
     \begin{align*}
        \nabla g_{n\gamma}(\textbf{u}_0) \rightarrow_p - \frac{{s}_{3}(\textbf{u}_0)}{{s}_{2}(\textbf{u}_0)} {\textbf{m}}_{1}(\textbf{u}_0) {\textbf{m}}_{1}(\textbf{u}_0)' + 2 {s}_{2}(\textbf{u}_0) {\textbf{G}}_{1}(\textbf{u}_0) - {s}_{3}(\textbf{u}_0) {\textbf{G}}_{0}(\textbf{u}_0).
    \end{align*}
    By the formulas for ${\textbf{m}}_{k}(\textbf{u}_0)$ and ${\textbf{G}}_{k}(\textbf{u}_0)$, this limit equals $-s_3 \boldsymbol{\Sigma}^{1/2} (\textbf{I}_p - \textbf{w} \textbf{w}') \boldsymbol{\Sigma}^{1/2} $. Using the same trick as in the proof of Theorem \ref{theo:asnorm_blind} to make the Jacobian full rank (addition of $c_0 \textbf{h} (s_n \textbf{u}_n + \textbf{u}_0)' \sqrt{n}(s_n \textbf{u}_n - \textbf{u}_0) = 0 $ for suitably chosen $c_0$ to the Taylor expansion), we obtain the corresponding matrix $\textbf{G}$ to be
    \begin{align*}
        \textbf{G} = -s_3 \boldsymbol{\Sigma}^{1/2} (\textbf{I}_p - \textbf{w} \textbf{w}' + \textbf{w} \textbf{w}' \boldsymbol{\Sigma}^{-1}) \boldsymbol{\Sigma}^{1/2},
    \end{align*}
    with the inverse,
\begin{align*}
        \textbf{G}^{-1} = -
        \frac{1}{s_3} \left\{ \boldsymbol{\Sigma}^{-1} + \frac{1}{\| \boldsymbol{\theta} \|^2} \boldsymbol{\theta} \boldsymbol{\theta}' (\textbf{I}_p - \boldsymbol{\Sigma}^{-1})  \right\}.
    \end{align*}

    Moving to study the limiting distribution of $\sqrt{n} g_{n\gamma}$, we note that
    \begin{align*}
        \tilde{\textbf{m}}_{n2}(\textbf{u}) = {\textbf{m}}_{n2}(\textbf{u}) - 2 {s}_{n1}(\textbf{u}) \textbf{m}_{1}(\textbf{u}) - {s}_{2}(\textbf{u}) {\textbf{m}}_{n0}(\textbf{u}) + o_p(1/\sqrt{n}),
    \end{align*}
    and
    \begin{align*}
        \tilde{s}_{n3}(\textbf{u}) = {s}_{n3}(\textbf{u}) - 3 {s}_{2}(\textbf{u}) {s}_{n1}(\textbf{u}) + o_p(1/\sqrt{n}).
    \end{align*}
    With these, we expand $\sqrt{n}  g_{n\gamma}(\textbf{u}_0)$ to be,
    \begin{align}\label{eq:g_gamma_expansion}
        \sqrt{n} g_{n\gamma} =& \sqrt{n} ( {s}_{n2} - s_2 ) \textbf{m}_2 + s_2 \sqrt{n} (\textbf{m}_{n2} - \textbf{m}_2) - \sqrt{n} ( {s}_{n3} - s_3 ) \textbf{m}_1 - s_3 \sqrt{n} (\textbf{m}_{n1} - \textbf{m}_1)\nonumber\\
        &+ s_2 \textbf{m}_1 \sqrt{n} s_{n1} - s_2^2 \sqrt{n} \textbf{m}_{n0}.
    \end{align}
    Hence, by CLT, $\sqrt{n} g_{n\gamma}$ has a limiting normal distribution with the covariance matrix,
    \begin{align*}
        \boldsymbol{\Pi} = \mathrm{Cov}\{ (\textbf{u}_0'\textbf{x})^2 \textbf{m}_2 + s_2 (\textbf{u}_0'\textbf{x})^2 \textbf{x} -  (\textbf{u}_0'\textbf{x})^3 \textbf{m}_1 - s_3 (\textbf{u}_0'\textbf{x}) \textbf{x} + s_2 \textbf{m}_1 (\textbf{u}_0'\textbf{x}) - s_2^2 \textbf{x} \}.
    \end{align*}
    The covariance matrix has the following 36 terms.
    \begin{multicols}{2}
    \begin{itemize}
        \item[(1, 1):] $(s_4 - s_2^2) \textbf{m}_2 \textbf{m}_2' = \psi^2 s_3^2 (s_4 - s_2^2) \textbf{h} \textbf{h}'$.
        \item[(1, 2):] $s_2 (\textbf{m}_2 \textbf{m}_4' - s_2 \textbf{m}_2 \textbf{m}_2') = \psi^2 s_2 s_3 (s_5 - s_2 s_3) \textbf{h} \textbf{h}'$.
        \item[(1, 3):] $-(s_5 - s_2 s_3) \textbf{m}_2 \textbf{m}_1' = - \psi^2 s_2 s_3 (s_5 - s_2 s_3) \textbf{h} \textbf{h}'$.
        \item[(1, 4):] $-s_3(\textbf{m}_2 \textbf{m}_3' - s_2 \textbf{m}_2 \textbf{m}_1') = - \psi^2 s_3^2 (s_4 - s_2^2) \textbf{h} \textbf{h}'$.
        \item[(1, 5):] $s_2 s_3 \textbf{m}_2 \textbf{m}_1' = \psi^2 s_2^2 s_3^2 \textbf{h} \textbf{h}' $.
        \item[(1, 6):] $-s_2^2 \textbf{m}_2 \textbf{m}_2' = -\psi^2 s_2^2 s_3^2 \textbf{h} \textbf{h}'$.
        \item[(2, 1):] $\psi^2 s_2 s_3 (s_5 - s_2 s_3) \textbf{h} \textbf{h}'$.
        \item[(2, 2):] $ s_2^2 (\textbf{G}_4 - \textbf{m}_2 \textbf{m}_2') = s_2^2 s_4 (\boldsymbol{\Sigma} - \tau^{-1} \textbf{h} \textbf{h}') + \psi^2 s_2^2 (s_6 - s_3^2) \textbf{h} \textbf{h}'$.
        \item[(2, 3):] $-s_2 (\textbf{m}_5 \textbf{m}_1' - s_3 \textbf{m}_2 \textbf{m}_1') = -\psi^2 s_2^2 (s_6 - s_3^2) \textbf{h} \textbf{h}'$.
        \item[(2, 4):] $-s_2 s_3(\textbf{G}_3 - \textbf{m}_2 \textbf{m}_1') = -s_2 s_3^2 (\boldsymbol{\Sigma} - \tau^{-1} \textbf{h} \textbf{h}') - \psi^2 s_2 s_3 (s_5 - s_2 s_3) \textbf{h} \textbf{h}'$.
        \item[(2, 5):] $ s_2^2 \textbf{m}_3 \textbf{m}_1' = \psi^2 s_2^3 s_4 \textbf{h} \textbf{h}' $.
        \item[(2, 6):] $- s_2^3 \textbf{G}_2 = - s_2^4 (\boldsymbol{\Sigma} - \tau^{-1} \textbf{h} \textbf{h}') - \psi^2 s_2^3 s_4 \textbf{h} \textbf{h}'$.
        \item[(3, 1):] $- \psi^2 s_2 s_3 (s_5 - s_2 s_3) \textbf{h} \textbf{h}'$.
        \item[(3, 2):] $-\psi^2 s_2^2 (s_6 - s_3^2) \textbf{h} \textbf{h}'$.
        \item[(3, 3):] $ (s_6 - s_3^2) \textbf{m}_1 \textbf{m}_1' = \psi^2 s_2^2 (s_6 - s_3^2) \textbf{h}\textbf{h}' $.
        \item[(3, 4):] $s_3 (\textbf{m}_1 \textbf{m}_4' - s_3 \textbf{m}_1 \textbf{m}_1') = \psi^2 s_2 s_3 (s_5 - s_2 s_3) \textbf{h} \textbf{h}'$.
        \item[(3, 5):] $-s_2 s_4 \textbf{m}_1 \textbf{m}_1' = - \psi^2 s_2^3 s_4 \textbf{h} \textbf{h}'$.
        \item[(3, 6):] $s_2^2 \textbf{m}_1 \textbf{m}_3' = \psi^2 s_2^3 s_4 \textbf{h} \textbf{h}' $.
        \item[(4, 1):] $ - \psi^2 s_3^2 (s_4 - s_2^2) \textbf{h} \textbf{h}'$.
        \item[(4, 2):] $ -s_2 s_3^2 (\boldsymbol{\Sigma} - \tau^{-1} \textbf{h} \textbf{h}') - \psi^2 s_2 s_3 (s_5 - s_2 s_3) \textbf{h} \textbf{h}'$.
        \item[(4, 3):] $\psi^2 s_2 s_3 (s_5 - s_2 s_3) \textbf{h} \textbf{h}'$.
        \item[(4, 4):] $s_3^2 (\textbf{G}_2 - \textbf{m}_1 \textbf{m}_1') =  s_2 s_3^2 (\boldsymbol{\Sigma} - \tau^{-1} \textbf{h} \textbf{h}') + \psi^2 s_3^2 (s_4 - s_2^2) \textbf{h} \textbf{h}'$.
        \item[(4, 5):] $-s_2 s_3 \textbf{m}_2 \textbf{m}_1' = - \psi^2 s_2^2 s_3^2 \textbf{h} \textbf{h}'$.
        \item[(4, 6):] $s_2^2 s_3 \textbf{G}_1 = \psi^2 s_2^2 s_3^2 \textbf{h} \textbf{h}'$.
        \item[(5, 1):] $\psi^2 s_2^2 s_3^2 \textbf{h} \textbf{h}'$.
        \item[(5, 2):] $\psi^2 s_2^3 s_4 \textbf{h} \textbf{h}'$.
        \item[(5, 3):] $- \psi^2 s_2^3 s_4 \textbf{h} \textbf{h}'$.
        \item[(5, 4):] $- \psi^2 s_2^2 s_3^2 \textbf{h} \textbf{h}'$.
        \item[(5, 5):] $s_2^3 \textbf{m}_1 \textbf{m}_1' = \psi^2 s_2^5 \textbf{h} \textbf{h}' $.
        \item[(5, 6):] $ -s_2^3 \textbf{m}_1 \textbf{m}_1' = -\psi^2 s_2^5 \textbf{h} \textbf{h}' $.
        \item[(6, 1):] $-\psi^2 s_2^2 s_3^2 \textbf{h} \textbf{h}'$.
        \item[(6, 2):] $- s_2^4 (\boldsymbol{\Sigma} - \tau^{-1} \textbf{h} \textbf{h}') - \psi^2 s_2^3 s_4 \textbf{h} \textbf{h}'$.
        \item[(6, 3):] $\psi^2 s_2^3 s_4 \textbf{h} \textbf{h}'$.
        \item[(6, 4):] $\psi^2 s_2^2 s_3^2 \textbf{h} \textbf{h}'$.
        \item[(6, 5):] $ -\psi^2 s_2^5 \textbf{h} \textbf{h}' $.
        \item[(6, 6):] $s_2^4 \textbf{G}_0 = s_2^4 (\boldsymbol{\Sigma} - \tau^{-1} \textbf{h} \textbf{h}') + \psi^2 s_2^5 \textbf{h} \textbf{h}'$.
    \end{itemize}
    \end{multicols}
    Summing the terms gives $\boldsymbol{\Pi} = s_2 (s_2 s_4 - s_2^3 - s_3^2) (\boldsymbol{\Sigma} - \tau^{-1} \textbf{h} \textbf{h}')$. This yields the limiting covariance,
    \begin{align}
        \boldsymbol{\Psi}_\kappa =  \textbf{G}^{-1} \boldsymbol{\Pi} (\textbf{G}^{-1})' = \frac{s_2 (s_2 s_4 - s_2^3 - s_3^2)}{s_3^2} \left( \textbf{I}_p - \frac{\boldsymbol{\theta} \boldsymbol{\theta}'}{\| \boldsymbol{\theta} \|^2} \right) \boldsymbol{\Sigma}^{-1} \left( \textbf{I}_p - \frac{\boldsymbol{\theta} \boldsymbol{\theta}'}{\| \boldsymbol{\theta} \|^2} \right).
    \end{align}
    Finally, simplifying the constant in front shows that it equals
    \begin{align*}
        \frac{(1 + \beta \tau) ( 2 + 6 \beta \tau + \beta \tau^2)}{\tau^2 \beta^2 (1 - 4 \beta) \| \boldsymbol{\theta} \|^2}.
    \end{align*}
\end{proof}


\begin{proof}[Proof of Theorem \ref{theo:asnorm_blind_3}]
    Again, the strong consistency follows as in Theorem \ref{theo:asnorm_blind} and we omit its proof. For the limiting distribution, we give in the following the key steps of the proof (and use the same notation as in the proofs of Theorems \ref{theo:asnorm_blind} and \ref{theo:asnorm_blind_2}).

     The gradient of $\eta_n$ is
    \begin{align*}
        \nabla \eta_n (\textbf{u}) = \frac{1}{\tilde{s}_{n2}^3} [ 6 w_1 \gamma_n(\textbf{u}) \tilde{s}_{n2}^{1/2}(\textbf{u}) g_{n\gamma}(\textbf{u}) + 8 w_2 \{ \kappa_n(\textbf{u}) - 3 \} g_{n\kappa}(\textbf{u}) ],
    \end{align*}
    where $ g_{n\gamma}(\textbf{u}) = \tilde{s}_{n2}(\textbf{u}) \tilde{\textbf{m}}_{n2}(\textbf{u}) - \tilde{s}_{n3}(\textbf{u}) \tilde{\textbf{m}}_{n1}(\textbf{u}) $ and $g_{n\kappa}(\textbf{u}) = \tilde{s}_{n2}(\textbf{u}) \tilde{\textbf{m}}_{n3}(\textbf{u}) - \tilde{s}_{n4}(\textbf{u}) \tilde{\textbf{m}}_{n1}(\textbf{u}) $ were used in the proofs of Theorems \ref{theo:asnorm_blind_2} and \ref{theo:asnorm_blind}, respectively. Thus, $\textbf{u}_n$ solves the estimating equation $g_{n\eta}(\textbf{u}_n) = \textbf{0}$, where $ g_{n\eta}(\textbf{u}) := 3 w_1 \gamma_n(\textbf{u}) \tilde{s}_{n2}^{1/2}(\textbf{u}) g_{n\gamma}(\textbf{u}) + 4 w_2 \{ \kappa_n(\textbf{u}) - 3 \} g_{n\kappa}(\textbf{u})$. The Jacobian of $g_{n\eta}$ at $\textbf{u}_0$ satisfies
    \begin{align*}
        \nabla g_{n \eta}(\textbf{u}_0) =& 3 w_1 [ \nabla \{ \gamma_n(\textbf{u}_0) \tilde{s}_{n2}^{1/2}(\textbf{u}_0) \} g_{n\gamma}(\textbf{u}_0)' + \gamma_n(\textbf{u}_0) \tilde{s}_{n2}^{1/2}(\textbf{u}_0) \nabla g_{n\gamma}(\textbf{u}_0) ]\\
        &+ 4 w_2 [ \nabla \{ \kappa_n(\textbf{u}_0) - 3 \} g_{n\kappa}(\textbf{u}_0)' + \{ \kappa_n(\textbf{u}_0) - 3 \}\nabla g_{n\kappa}(\textbf{u}_0)].
    \end{align*}
    Recalling now that $ \textbf{m}_{k}(\textbf{u}_0) = \psi s_{k + 1} \textbf{h} $ and $\textbf{G}_{k}(\textbf{u}_0) = s_k (\boldsymbol{\Sigma} - \tau^{-1} \textbf{h} \textbf{h}') + \psi^2 s_{k + 2} \textbf{h} \textbf{h}' $, where $\psi = \| \boldsymbol{\theta} \| \tau^{-1}$, LLN now gives that $ g_{n\gamma}(\textbf{u}_0) \rightarrow_p \textbf{0} $ and $ g_{n\kappa}(\textbf{u}_0) \rightarrow_p \textbf{0} $, implying that
    \begin{align*}
        \nabla g_{n \eta}(\textbf{u}_0) \rightarrow_p -s_2^{-2}\{ 3 w_1 s_2 s_3^2 + 4 w_2 (s_4 - 3 {s}_2^2)^2 \} \boldsymbol{\Sigma}^{1/2} (\textbf{I}_p - \textbf{w} \textbf{w}') \boldsymbol{\Sigma}^{1/2}.
    \end{align*}
    Completing now this matrix to full rank through the unit length constraint on $\textbf{u}_n$ (as in the proofs of Theorems \ref{theo:asnorm_blind_2} and \ref{theo:asnorm_blind}), we now obtain that,
    \begin{align*}
        \textbf{G}^{-1} = \frac{-s_2^2}{3 w_1 s_2 s_3^2 + 4 w_2 (s_4 - 3 {s}_2^2)^2} \left\{ \boldsymbol{\Sigma}^{-1} + \frac{1}{\| \boldsymbol{\theta} \|^2} \boldsymbol{\theta} \boldsymbol{\theta}' (\textbf{I}_p - \boldsymbol{\Sigma}^{-1})  \right\}.
    \end{align*}
    We then derive the limiting distribution of $\sqrt{n} g_{n\eta}(\textbf{u}_0) = 3 w_1 \sqrt{n} \tilde{s}_{n2}^{1/2}(\textbf{u}_0) g_{n\gamma}(\textbf{u}_0) + 4 w_2 \sqrt{n} g_{n\kappa}(\textbf{u}_0)$. Recalling that the population version satisfies $ 3 w_1 \gamma(\textbf{u}_0) {s}_{2}^{1/2}(\textbf{u}_0) g_{\gamma}(\textbf{u}_0) + 4 w_2 \{ \kappa(\textbf{u}_0) - 3 \} g_{\kappa}(\textbf{u}_0) = 0$, we get the expansion,
    \begin{align*}
        \sqrt{n} g_{n\eta} =& 3 w_1 \{ \sqrt{n} ( \gamma_n \tilde{s}_{n2}^{1/2} - \gamma s_2^{1/2} ) (s_2 \textbf{m}_2 - s_3 \textbf{m}_1) + s_2^{-1} s_3 \sqrt{n} g_{n\gamma} \}\\
        &+ 4 w_2 \{ \sqrt{n} ( \kappa_n - \kappa ) (s_2 \textbf{m}_3 - s_4 \textbf{m}_1) + s_2^{-2} (s_4 - 3 s_2^2) \sqrt{n} g_{n\kappa} \} + o_p(1),
    \end{align*}
    where $\sqrt{n} g_{n\kappa}$ has the expansion given in \eqref{eq:g_kappa_expansion}. Now, $s_2 \textbf{m}_2 - s_3 \textbf{m}_1 = \textbf{0}$ and $s_2 \textbf{m}_3 - s_4 \textbf{m}_1 = \textbf{0}$, implying that
    \begin{align*}
        \sqrt{n} g_{n\eta} = 3 w_1 s_2^{-1} s_3 \sqrt{n} g_{n\gamma} + 4 w_2 s_2^{-2} (s_4 - 3 s_2^2) \sqrt{n} g_{n\kappa} + o_p(1),
    \end{align*}
    where $ \sqrt{n} g_{n\gamma} $ has the expansion given in \eqref{eq:g_gamma_expansion}. Consequently, by CLT, $\sqrt{n} g_{n\eta}$ has a limiting normal distribution. By the proof of Theorem \ref{theo:asnorm_blind}, the limiting covariance matrix of $4 w_2 s_2^{-2} (s_4 - 3 s_2^2) \sqrt{n} g_{n\kappa}$ is $16 w_2^2 s_2^{-3} (s_4 - 3 s_2^2)^2 (s_2 s_6 - s_2 s_3^2 - s_4^2) (\boldsymbol{\Sigma} - \tau^{-1} \textbf{h} \textbf{h}')$ and, by the proof of Theorem \ref{theo:asnorm_blind_2}, the limiting covariance matrix of $ 3 w_1 s_2^{-1} s_3 \sqrt{n} g_{n\gamma} $ is $ 9 w_1^2 s_2^{-1} s_3^2 (s_2 s_4 - s_2^3 - s_3^2) (\boldsymbol{\Sigma} - \tau^{-1} \textbf{h} \textbf{h}') $. Thus, the limiting covariance matrix of $\sqrt{n} g_{n\eta}$ is $ \{ 9 w_1^2 s_2^{-1} s_3^2 (s_2 s_4 - s_2^3 - s_3^2) + 16 w_2^2 s_2^{-3} (s_4 - 3 s_2^2)^2 (s_2 s_6 - s_2 s_3^2 - s_4^2) \} (\boldsymbol{\Sigma} - \tau^{-1} \textbf{h} \textbf{h}') + 24 w_1 w_2 s_2^{-3} s_3 (s_4 - 3 s_2^2) \mathrm{Cov}(\textbf{y}_1, \textbf{y}_2) $, where
    \begin{align*}
        \textbf{y}_1 &:= (\textbf{u}_0'\textbf{x})^2 \textbf{m}_2 + s_2 (\textbf{u}_0'\textbf{x})^2 \textbf{x} -  (\textbf{u}_0'\textbf{x})^3 \textbf{m}_1 - s_3 (\textbf{u}_0'\textbf{x}) \textbf{x} + s_2 (\textbf{u}_0'\textbf{x}) \textbf{m}_1 - s_2^2 \textbf{x},\\
        \textbf{y}_2 &:= (\textbf{u}_0'\textbf{x})^2 \textbf{m}_3 + s_2 (\textbf{u}_0'\textbf{x})^3 \textbf{x} -  (\textbf{u}_0'\textbf{x})^4 \textbf{m}_1 - s_4 (\textbf{u}_0'\textbf{x}) \textbf{x} + s_3 (\textbf{u}_0'\textbf{x}) \textbf{m}_1 - s_2 s_3 \textbf{x}.
    \end{align*}
    The matrix $\mathrm{Cov}(\textbf{y}_1, \textbf{y}_2)$ consists of the following 36 terms:
    \begin{multicols}{2}
    \begin{itemize}
        \item[(1, 1):] $\psi^2 s_3 s_4 (s_4 - s_2^2) \textbf{h} \textbf{h}'$.
        \item[(1, 2):] $\psi^2 s_2 s_3 (s_6 - s_2 s_4) \textbf{h} \textbf{h}'$.
        \item[(1, 3):] $-\psi^2 s_2 s_3 (s_6 - s_2 s_4) \textbf{h} \textbf{h}'$.
        \item[(1, 4):] $-\psi^2 s_3 s_4 (s_4 - s_2^2) \textbf{h} \textbf{h}'$.
        \item[(1, 5):] $\psi^2 s_2 s_3^3 \textbf{h} \textbf{h}'$.
        \item[(1, 6):] $-\psi^2 s_2 s_3^3 \textbf{h} \textbf{h}'$.
        \item[(2, 1):] $\psi^2 s_2 s_4 (s_5 - s_2 s_3) \textbf{h} \textbf{h}'$.
        \item[(2, 2):] $s_2^2 s_5 (\boldsymbol{\Sigma} - \tau^{-1} \textbf{h} \textbf{h}') + \psi^2 s_2^2 (s_7 - s_3 s_4) \textbf{h} \textbf{h}'$.
        \item[(2, 3):] $-\psi^2 s_2^2 (s_7 - s_3 s_4) \textbf{h} \textbf{h}'$.
        \item[(2, 4):] $-s_2 s_3 s_4 (\boldsymbol{\Sigma} - \tau^{-1} \textbf{h} \textbf{h}') - \psi^2 s_2 s_4 (s_5 - s_2 s_3) \textbf{h} \textbf{h}'$.
        \item[(2, 5):] $\psi^2 s_2^2 s_3 s_4 \textbf{h} \textbf{h}'$.
        \item[(2, 6):] $-s_2^3 s_3 (\boldsymbol{\Sigma} - \tau^{-1} \textbf{h} \textbf{h}') - \psi^2 s_2^2 s_3 s_4 \textbf{h} \textbf{h}'$.
        \item[(3, 1):] $-\psi^2 s_2 s_4 (s_5 - s_2 s_3) \textbf{h} \textbf{h}'$.
        \item[(3, 2):] $-\psi^2 s_2^2 (s_7 - s_3 s_4) \textbf{h} \textbf{h}'$.
        \item[(3, 3):] $\psi^2 s_2^2 (s_7 - s_3 s_4) \textbf{h} \textbf{h}'$.
        \item[(3, 4):] $\psi^2 s_2 s_4 (s_5 - s_2 s_3) \textbf{h} \textbf{h}'$.
        \item[(3, 5):] $-\psi^2 s_2^2 s_3 s_4 \textbf{h} \textbf{h}'$.
        \item[(3, 6):] $\psi^2 s_2^2 s_3 s_4 \textbf{h} \textbf{h}'$.
        \item[(4, 1):] $-\psi^2 s_3 s_4 (s_4 - s_2^2) \textbf{h} \textbf{h}'$.
        \item[(4, 2):] $-s_2 s_3 s_4 (\boldsymbol{\Sigma} - \tau^{-1} \textbf{h} \textbf{h}') - \psi^2 s_2 s_3 (s_6 - s_2 s_4) \textbf{h} \textbf{h}'$.
        \item[(4, 3):] $\psi^2 s_2 s_3 (s_6 - s_2 s_4) \textbf{h} \textbf{h}'$.
        \item[(4, 4):] $s_2 s_3 s_4 (\boldsymbol{\Sigma} - \tau^{-1} \textbf{h} \textbf{h}') + \psi^2 s_3 s_4 (s_4 - s_2^2) \textbf{h} \textbf{h}'$.
        \item[(4, 5):] $-\psi^2 s_2 s_3^3 \textbf{h} \textbf{h}'$.
        \item[(4, 6):] $\psi^2 s_2 s_3^3 \textbf{h} \textbf{h}'$.
        \item[(5, 1):] $\psi^2 s_2^2 s_3 s_4 \textbf{h} \textbf{h}'$.
        \item[(5, 2):] $\psi^2 s_2^3 s_5 \textbf{h} \textbf{h}'$.
        \item[(5, 3):] $-\psi^2 s_2^3 s_5 \textbf{h} \textbf{h}'$.
        \item[(5, 4):] $-\psi^2 s_2^2 s_3 s_4 \textbf{h} \textbf{h}'$.
        \item[(5, 5):] $\psi^2 s_2^4 s_3 \textbf{h} \textbf{h}'$.
        \item[(5, 6):] $-\psi^2 s_2^4 s_3 \textbf{h} \textbf{h}'$.
        \item[(6, 1):] $-\psi^2 s_2^2 s_3 s_4 \textbf{h} \textbf{h}'$.
        \item[(6, 2):] $-s_2^3 s_3 (\boldsymbol{\Sigma} - \tau^{-1} \textbf{h} \textbf{h}') - \psi^2 s_2^3 s_5 \textbf{h} \textbf{h}'$.
        \item[(6, 3):] $\psi^2 s_2^3 s_5 \textbf{h} \textbf{h}'$.
        \item[(6, 4):] $\psi^2 s_2^2 s_3 s_4 \textbf{h} \textbf{h}'$.
        \item[(6, 5):] $-\psi^2 s_2^4 s_3 \textbf{h} \textbf{h}'$.
        \item[(6, 6):] $s_2^3 s_3 (\boldsymbol{\Sigma} - \tau^{-1} \textbf{h} \textbf{h}') + \psi^2 s_2^4 s_3 \textbf{h} \textbf{h}'$.
    \end{itemize}
    \end{multicols}
    The sum of the 36 terms is $s_2 (s_2 s_5 - s_2^2 s_3 - s_3 s_4) (\boldsymbol{\Sigma} - \tau^{-1} \textbf{h} \textbf{h}')$. Hence, the limiting covariance matrix of $\sqrt{n} g_{n\eta}$ is $ \{ 9 w_1^2 s_2^{-1} s_3^2 (s_2 s_4 - s_2^3 - s_3^2) + 24 w_1 w_2 s_2^{-2} s_3 (s_4 - 3 s_2^2) (s_2 s_5 - s_2^2 s_3 - s_3 s_4) + 16 w_2^2 s_2^{-3} (s_4 - 3 s_2^2)^2 (s_2 s_6 - s_2 s_3^2 - s_4^2) \} (\boldsymbol{\Sigma} - \tau^{-1} \textbf{h} \textbf{h}') $. Consequently, the limiting covariance of $\sqrt{n} (\textbf{u}_n - \textbf{u}_0)$ is
    \begin{align*}
        \boldsymbol{\Psi}_\eta =& \left\{\frac{9 w_1^2 s_2^{3} s_3^2 (s_2 s_4 - s_2^3 - s_3^2) + 24 w_1 w_2 s_2^{2} s_3 (s_4 - 3 s_2^2) (s_2 s_5 - s_2^2 s_3 - s_3 s_4)}{\{ 3 w_1 s_2 s_3^2 + 4 w_2 (s_4 - 3 {s}_2^2)^2 \}^2}\right.\\
        &+ \left.\frac{16 w_2^2 s_2 (s_4 - 3 s_2^2)^2 (s_2 s_6 - s_2 s_3^2 - s_4^2)}{\{ 3 w_1 s_2 s_3^2 + 4 w_2 (s_4 - 3 {s}_2^2)^2 \}^2}\right\} \left( \textbf{I}_p - \frac{\boldsymbol{\theta} \boldsymbol{\theta}'}{\| \boldsymbol{\theta} \|^2} \right) \boldsymbol{\Sigma}^{-1} \left( \textbf{I}_p - \frac{\boldsymbol{\theta} \boldsymbol{\theta}'}{\| \boldsymbol{\theta} \|^2} \right).
    \end{align*}
    Using now the expressions for $s_2, s_3, s_4, s_6$ in the proof of Theorem \ref{theo:asnorm_blind} and the analogously obtainable formula $s_5 = \| \boldsymbol{\theta} \|^{-5} (\alpha_1 - \alpha_2) \beta \tau^4 \{ (1 - 2\beta) \tau + 10 \}$, the factor in front of the covariance matrix simplifies to $ C_\eta (1 + \beta \tau)/(\| \boldsymbol{\theta} \|^2 \beta)$, where
    \begin{align*}
        C_\eta =& \frac{ 9 w_1^2 (1 + \beta \tau)^2 (1 - 4\beta) (2 + 6 \beta \tau + \beta \tau^2) + 24 w_1 w_2 \tau^{2} (1 + \beta \tau) (1 - 4 \beta) \beta (1 - 6 \beta) (6 + \tau)}{ \beta \tau^{2} \{ 3 w_1 (1 + \beta \tau) (1 - 4\beta) + 4 w_2 \tau (1 - 6 \beta)^2 \}^2 } \\
        &+ \frac{16 w_2^2 \tau (1 - 6 \beta)^2 \Delta }{ \beta \tau^{2} \{ 3 w_1 (1 + \beta \tau) (1 - 4\beta) + 4 w_2 \tau (1 - 6 \beta)^2 \}^2 }.
    \end{align*}
    and $\Delta := 6 + 24 \beta \tau + 9 \beta (1 - 2 \beta) \tau^2 + \beta (1 - 3 \beta) \tau^3 $.

\end{proof}

\begin{proof}[Proof of Lemma \ref{lem:inner_product}]
    The limiting distribution of $ \{ \sqrt{n}(s_n \textbf{u}_n - \boldsymbol{\theta}/\| \boldsymbol{\theta} \|) \}' \{ \sqrt{n}(s_n \textbf{u}_n - \boldsymbol{\theta}/\| \boldsymbol{\theta} \|) \} $ is the same as that of $ \textbf{z}' \boldsymbol{\Psi} \textbf{z} $. The first claim now follows by observing that,
    \begin{align*}
        \{ \sqrt{n}(s_n \textbf{u}_n - \boldsymbol{\theta}/\| \boldsymbol{\theta} \|) \}' \{ \sqrt{n}(s_n \textbf{u}_n - \boldsymbol{\theta}/\| \boldsymbol{\theta} \|) \} &= 2 n ( 1 - s_n \textbf{u}_n' \boldsymbol{\theta}/\| \boldsymbol{\theta} \|).
    \end{align*}
    Finally, $\mathrm{E}(\textbf{z}' \boldsymbol{\Psi} \textbf{z}) = \mathrm{tr}\{ \boldsymbol{\Psi} \mathrm{E}(\textbf{z} \textbf{z}') \} = \mathrm{tr}(\boldsymbol{\Psi})$.
\end{proof}

\begin{proof}[Proof of Lemma \ref{lem:PCA_fisher}]
    We first show that i) implies ii). The positive-definiteness of $\boldsymbol{\Sigma} $ in conjunction with the relation $\boldsymbol{\Sigma} \textbf{h} = \phi \textbf{h}$ gives that $\boldsymbol{\Sigma}^{-1} \textbf{h} = \phi^{-1} \textbf{h}$. Consequently,
    \begin{align*}
        \mathrm{Cov}(\textbf{x}) \frac{\boldsymbol{\theta}}{\| \boldsymbol{\theta} \|} = (\boldsymbol{\Sigma} + \beta \textbf{h} \textbf{h}') \frac{\boldsymbol{\Sigma}^{-1} \textbf{h}}{\| \boldsymbol{\Sigma}^{-1} \textbf{h} \| } = \frac{\phi}{\| \textbf{h} \|} (1 + \beta \tau) \textbf{h} = \phi  (1 + \beta \tau) \frac{\boldsymbol{\theta}}{\| \boldsymbol{\theta} \|}.
    \end{align*}
    Hence, $\pm \boldsymbol{\theta}/\| \boldsymbol{\theta} \|$ are the unique leading unit-length eigenvectors of $\mathrm{Cov}(\textbf{x})$ if the second-to-largest eigenvalue of $\mathrm{Cov}(\textbf{x})$ is smaller than $ \phi (1 + \beta \tau) $, the eigenvalue corresponding to $\pm \boldsymbol{\theta}/\| \boldsymbol{\theta} \|$.

    To see that ii) implies i), denote the eigenvalue of $ \mathrm{Cov}(\textbf{x}) $ corresponding to $ \pm \boldsymbol{\theta}/\| \boldsymbol{\theta} \| $ by $\rho$. Then, $ (\boldsymbol{\Sigma} + \beta \textbf{h} \textbf{h}') \boldsymbol{\theta} = \rho \boldsymbol{\theta}$ or, equivalently,
    \begin{align*}
        \boldsymbol{\Sigma} \textbf{h} = \frac{\rho}{1 + \beta \tau} \textbf{h},
    \end{align*}
    showing that $\textbf{h}$ is indeed an eigenvector of $\boldsymbol{\Sigma}$ corresponding to the eigenvalue $\phi := \rho/(1 + \beta \tau)$. Finally, since $\pm \boldsymbol{\theta}/\| \boldsymbol{\theta} \|$ are the unique leading unit length eigenvectors of $ \mathrm{Cov}(\textbf{x}) $, we have $\phi_2\{ \mathrm{Cov}(\textbf{x}) \} < \rho = \phi ( 1 + \beta \tau )$, concluding the proof.
\end{proof}

\begin{proof}[Proof of Theorem \ref{theo:asnorm_PCA}]
    The proof of the strong consistency is done similarly as in Theorem~\ref{theo:asnorm_blind} and we omit it. For the limiting normality we again, without loss of generality, assume that $\textbf{x}$ has zero mean, implying that $\textbf{x} \sim \alpha_1 \mathcal{N}_p(-\alpha_2 \textbf{h}, \boldsymbol{\Sigma}) + \alpha_2 \mathcal{N}_p(\alpha_1 \textbf{h}, \boldsymbol{\Sigma})$, where $\textbf{h} = \boldsymbol{\mu}_2 - \boldsymbol{\mu}_1$.

Let $\textbf{u}_1 := \boldsymbol{\theta}/\| \boldsymbol{\theta} \| $, where $\boldsymbol{\theta} = \boldsymbol{\Sigma}^{-1} \textbf{h}$, and recall that it is a leading eigenvector of $\mathrm{Cov}(\textbf{x}) = \boldsymbol{\Sigma} + \beta \textbf{h} \textbf{h}'$. Denote any set of the remaining $p - 1$ eigenvectors by $\textbf{u}_2, \ldots , \textbf{u}_p$ and the corresponding eigenvalues by $ \phi =: \phi_1 > \phi_2 \geq \cdots \geq \phi_p > 0 $.

The gradient of the Lagrangian corresponding to the extraction of the leading unit length eigenvector of $\textbf{C}_n$ is
\begin{align*}
    2 \textbf{C}_n \textbf{u} - 2 \lambda_n \textbf{u},
\end{align*}
where $\lambda_n$ is the Lagrangian multiplier. The gradient vanishes at $s_n \textbf{u}_n$, allowing us to solve the value of $\lambda_n = \textbf{u}_n' \textbf{C}_n \textbf{u}_n$ by multiplying the gradient equation from left with $\textbf{u}_n$. Plugging the multiplier back in gives $(\textbf{I}_p - \textbf{P}_{n}) \textbf{C}_n s_n \textbf{u}_{n} = \textbf{0}$, where $\textbf{P}_{n} := \textbf{u}_{n} \textbf{u}_{n}'$. This equation is equivalent to the following equality,
\begin{align} \label{eq:asymp_1}
     \textbf{A}_{n}  \sqrt{n}(s_n \textbf{u}_n - \textbf{u}_1) = - (\textbf{I}_p - \textbf{P}_{n}) \sqrt{n} \{ \textbf{C}_n - \mathrm{Cov}(\textbf{x}) \} \textbf{u}_1,
\end{align}
where
\begin{align*}
    \textbf{A}_{n} := (\textbf{I}_p - \textbf{P}_{n}) \textbf{C}_n - s_n \phi_1 (\textbf{u}_{n}' \textbf{u}_1) \textbf{I}_p + \phi_1 s_n \textbf{u}_1 \textbf{u}_{n}',
\end{align*}
$\phi_1 = \textbf{u}_1' \mathrm{Cov}(\textbf{x}) \textbf{u}_1 = \| \boldsymbol{\theta} \|^{-2} \tau (1 + \beta \tau)$ and $\tau = \textbf{h}' \boldsymbol{\Sigma}^{-1} \textbf{h} $. Now, by the strong consistency $s_{n} \textbf{u}_{n} \rightarrow \textbf{u}_1$, we have $\textbf{P}_n \rightarrow_p \textbf{u}_1 \textbf{u}_1' =: \textbf{P}$. Consequently,
\begin{align*}
    \textbf{A}_{n} \rightarrow_p (\textbf{I}_p - \textbf{P}) \mathrm{Cov}(\textbf{x}) - \phi_1 \textbf{I}_p + \phi_1 \textbf{u}_1 \textbf{u}_{1}' = \mathrm{Cov}(\textbf{x}) - \phi_1 \textbf{I}_p. 
\end{align*}

Observe then that we have the identity $\textbf{B}_n \sqrt{n}(s_n \textbf{u}_n - \textbf{u}_1) = \textbf{0}$ where $\textbf{B}_n := \textbf{u}_1 \textbf{u}_1' + s_n \textbf{u}_1 \textbf{u}_{n}' \rightarrow_p 2 \textbf{u}_1 \textbf{u}_1' $. As $\sqrt{n} \{ \textbf{C}_n - \mathrm{Cov}(\textbf{x}) \} = \mathcal{O}_p(1)$, summing the previous equation and \eqref{eq:asymp_1}, yields,
\begin{align*}
    \left( \textbf{A}_n + \textbf{B}_n \right) \sqrt{n}(s_n \textbf{u}_n - \textbf{u}_1) = - (\textbf{I}_p - \textbf{P}) \sqrt{n} \{ \textbf{C}_n - \mathrm{Cov}(\textbf{x}) \} \textbf{u}_1 + o_p(1),
\end{align*}
where $\textbf{A}_n + \textbf{B}_n \rightarrow_p 2 \textbf{u}_1 \textbf{u}_1' + \sum_{j=2}^p (\phi_j - \phi_1) \textbf{u}_j \textbf{u}_j'$ and $\phi_j - \phi_1 < 0$ for all $j = 2, \ldots , p$. Hence, by Slutsky's theorem, the limiting distribution of $ \sqrt{n}(s_n \textbf{u}_n - \textbf{u}_1) $ is that of
\begin{align*}
  & - \left( \frac{1}{2} \textbf{u}_1 \textbf{u}_1' + \sum_{j=2}^p \frac{1}{\phi_j - \phi_1} \textbf{u}_j \textbf{u}_j' \right) (\textbf{I}_p - \textbf{P}) \sqrt{n} \{ \textbf{C}_n - \mathrm{Cov}(\textbf{x}) \} \textbf{u}_1\\
  =& - \left( \sum_{j=2}^p  \frac{1}{\phi_j - \phi_1} \textbf{u}_j \textbf{u}_j' \right) \sqrt{n} \{ \textbf{C}_n - \mathrm{Cov}(\textbf{x}) \} \textbf{u}_1.
\end{align*}
Observing that $ (1/n) \sum_i (\textbf{x}_i - \bar{\textbf{x}}) (\textbf{x}_i - \bar{\textbf{x}})' = (1/n) \sum_i \textbf{x}_i \textbf{x}_i' + o_p(1/\sqrt{n})$, the limiting covariance matrix of $ \sqrt{n}(s_n \textbf{u}_n - \textbf{u}_1) $ is hence
\begin{align*}
    \boldsymbol{\Psi}_{\mathrm{PCA}} = \left( \sum_{j=2}^p  \frac{1}{\phi_j - \phi_1} \textbf{u}_j \textbf{u}_j' \right) \mathrm{Cov}\{  (\textbf{u}_1' \textbf{x}) \textbf{x} \} \left( \sum_{j=2}^p  \frac{1}{\phi_j - \phi_1} \textbf{u}_j \textbf{u}_j' \right).
\end{align*}
In the notation of Theorem \ref{theo:asnorm_blind}, we have $\mathrm{Cov}\{  (\textbf{u}_1' \textbf{x}) \textbf{x} \} = \textbf{G}_2 - \textbf{m}_1 \textbf{m}_1' = s_2 (\boldsymbol{\Sigma} - \tau^{-1} \textbf{h} \textbf{h}') + \psi^2 ( s_4 - s_2^2) \textbf{h} \textbf{h}' $, where $ s_2 = \| \boldsymbol{\theta} \|^{-2} \tau ( 1 + \beta \tau) $ and $s_4/s_2^2 = \kappa(\textbf{u}_1) = \kappa(\boldsymbol{\theta})$, yielding the result.

\end{proof}

\section{Additional simulation results}\label{app_simu}
In this section, we give supporting plots as a supplementary material to claims made and plots presented in the article. Simulations and the corresponding plots are done using R 3.6.1 \citep{R} together with R packages  ICtest \citep{ICtest}, mvtnorm \citep{mvtnorm}, MASS \citep{MASS}, GGally \citep{GGally}, ggpubr \citep{ggpubr}, dplyr \citep{dplyr}, tidyr \citep{tidyr} and RColorBrewer \citep{RColorBrewer}.\\

Figures~\ref{fig:as_fig_10} and \ref{fig:as_fig_15} show the standard deviation of maximal similarity index $s_n\textbf{u}_n'\boldsymbol\theta/\|\boldsymbol\theta\|$ where $\textbf{u}_n$ is one of PP estimators discussed in the article, as a function of the Mahalanobis distance between the group means $\tau$ and mixing proportion $\alpha_1$, for sample sizes $n\in\{500,1000,2000,4000\}$ and $n\in\{8000,16000,32000\}$ respectively.

Figure~\ref{fig:as_fig_12} shows a scatter matrix plot of the \textit{finance} data set from the R-package \textit{Rmodmix}, where the point in the plot is being colored red if the company is being bankrupt, and blue otherwise, as well as the marginal densities for both groups which are given at the diagonal.

Figure~\ref{fig:as_fig_13} shows boxplots of the projection scores of the \textit{finance} data set from the R-package \textit{Rmodmix} along the PP directions based on PCA, LDA, kurtosis, skewness and hybrid estimator $\eta_n(\cdot,w_1)$, for $w_1=0.1,0.2,\dots,0.9$ for healthy and bankrupted companies.
\begin{figure}[ht]
\centering
    \includegraphics[width=1\textwidth]{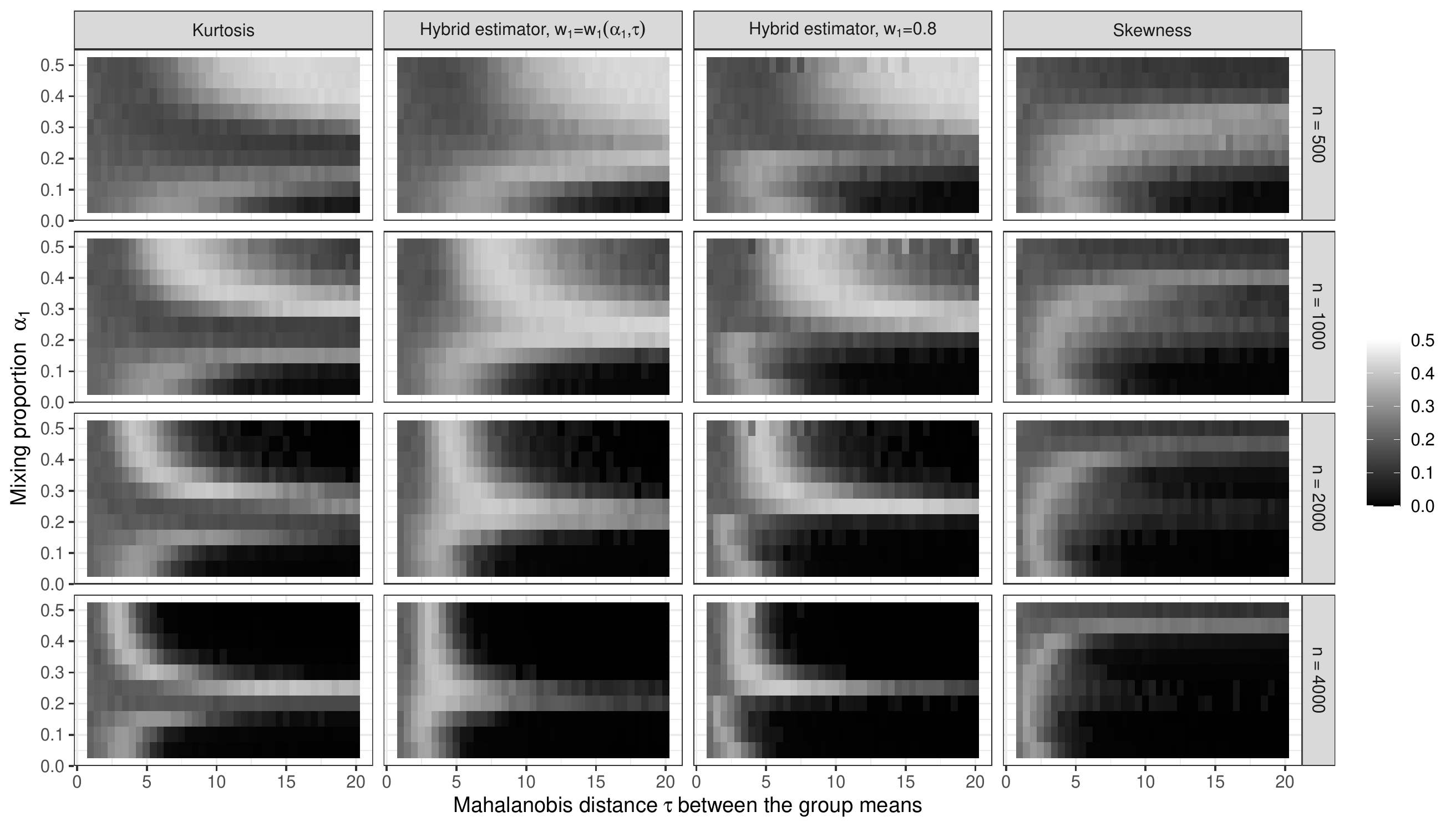}
    \caption{Standard deviation of the MSI  $s_n\textbf{u}_n'\boldsymbol\theta/||\boldsymbol\theta||$ as a function of Mahalanobis distance $\tau$ between the group means and mixing proportion $\alpha_1$, where $\textbf{u}_n$ is one of the four estimators discussed above.}
    \label{fig:as_fig_10}
\end{figure}

\begin{figure}[ht]
\centering
    \includegraphics[width=1\textwidth]{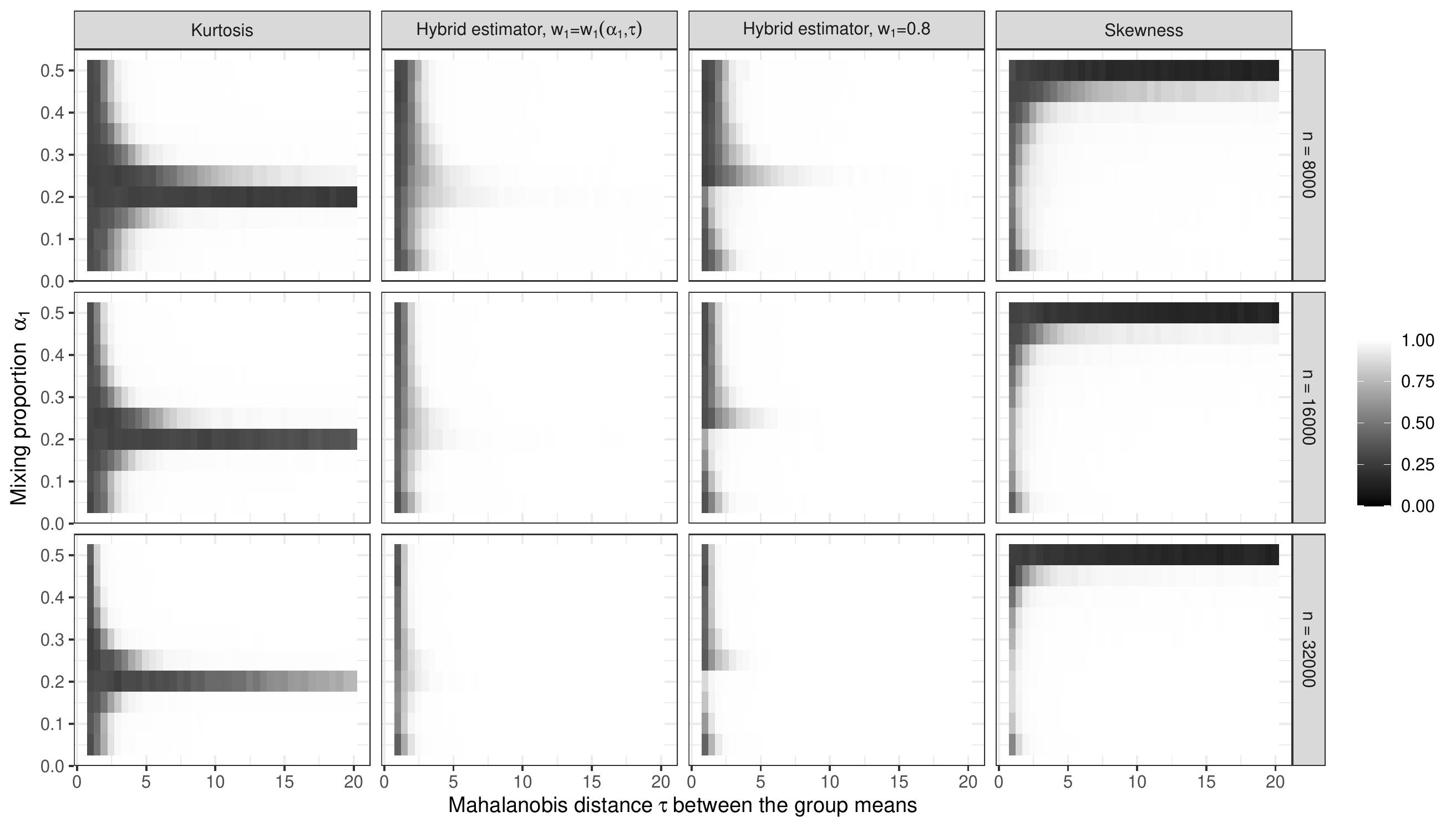}
    \caption{Average values of the MSI  $s_n\textbf{u}_n'\boldsymbol\theta/||\boldsymbol\theta||$ as a function of Mahalanobis distance $\tau$ between the group means and mixing proportion $\alpha_1$, where $\textbf{u}_n$ is one of the four estimators discussed above.}
    \label{fig:as_fig_14}
\end{figure}

\begin{figure}[H]
\centering
    \includegraphics[width=1\textwidth]{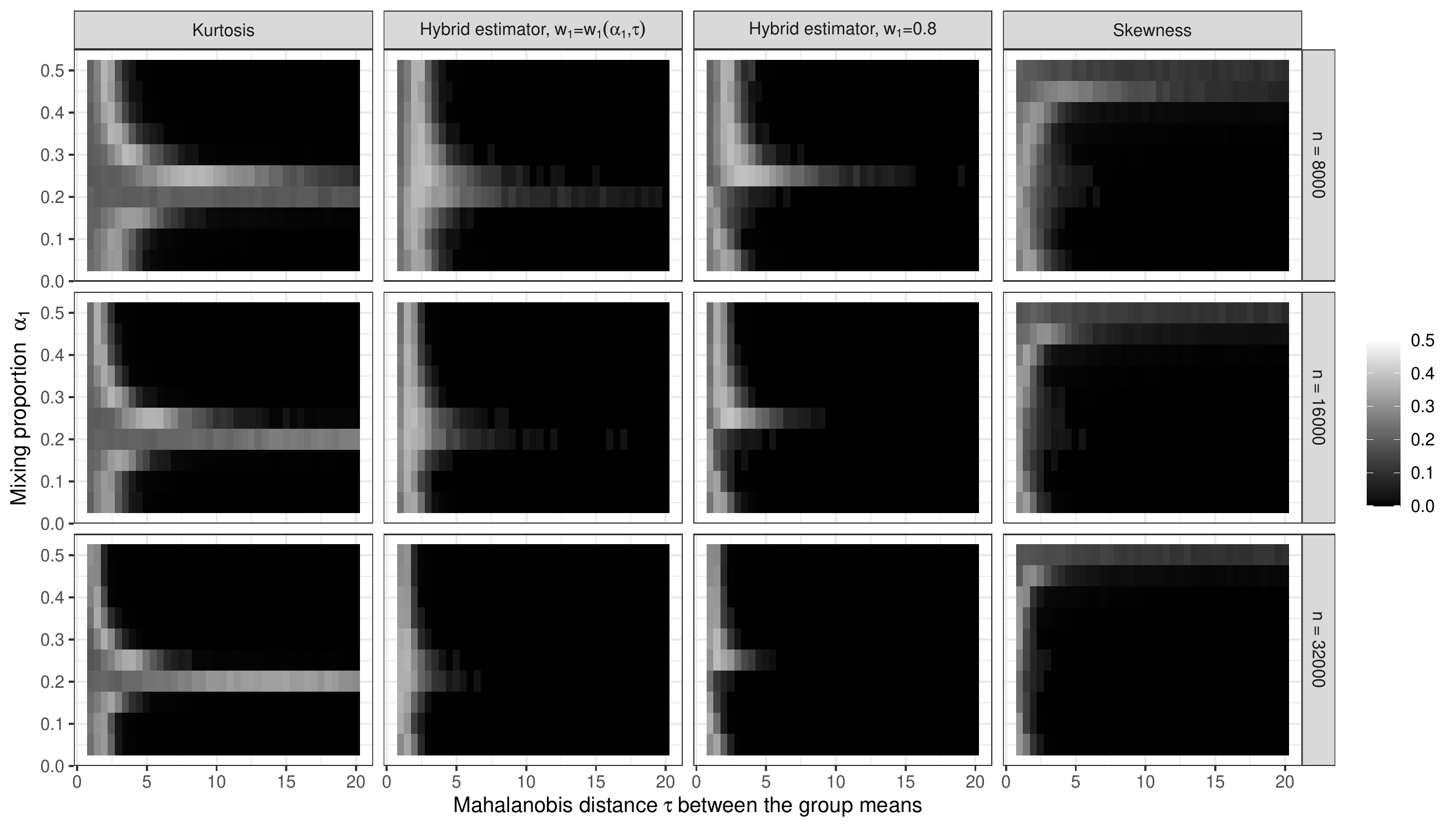}
    \caption{Standard deviation of the MSI  $s_n\textbf{u}_n'\boldsymbol\theta/||\boldsymbol\theta||$ as a function of Mahalanobis distance $\tau$ between the group means and mixing proportion $\alpha_1$, where $\textbf{u}_n$ is one of the four estimators discussed above.}
    \label{fig:as_fig_15}
\end{figure}

\begin{figure}[H]
\centering
    \includegraphics[width=1\textwidth]{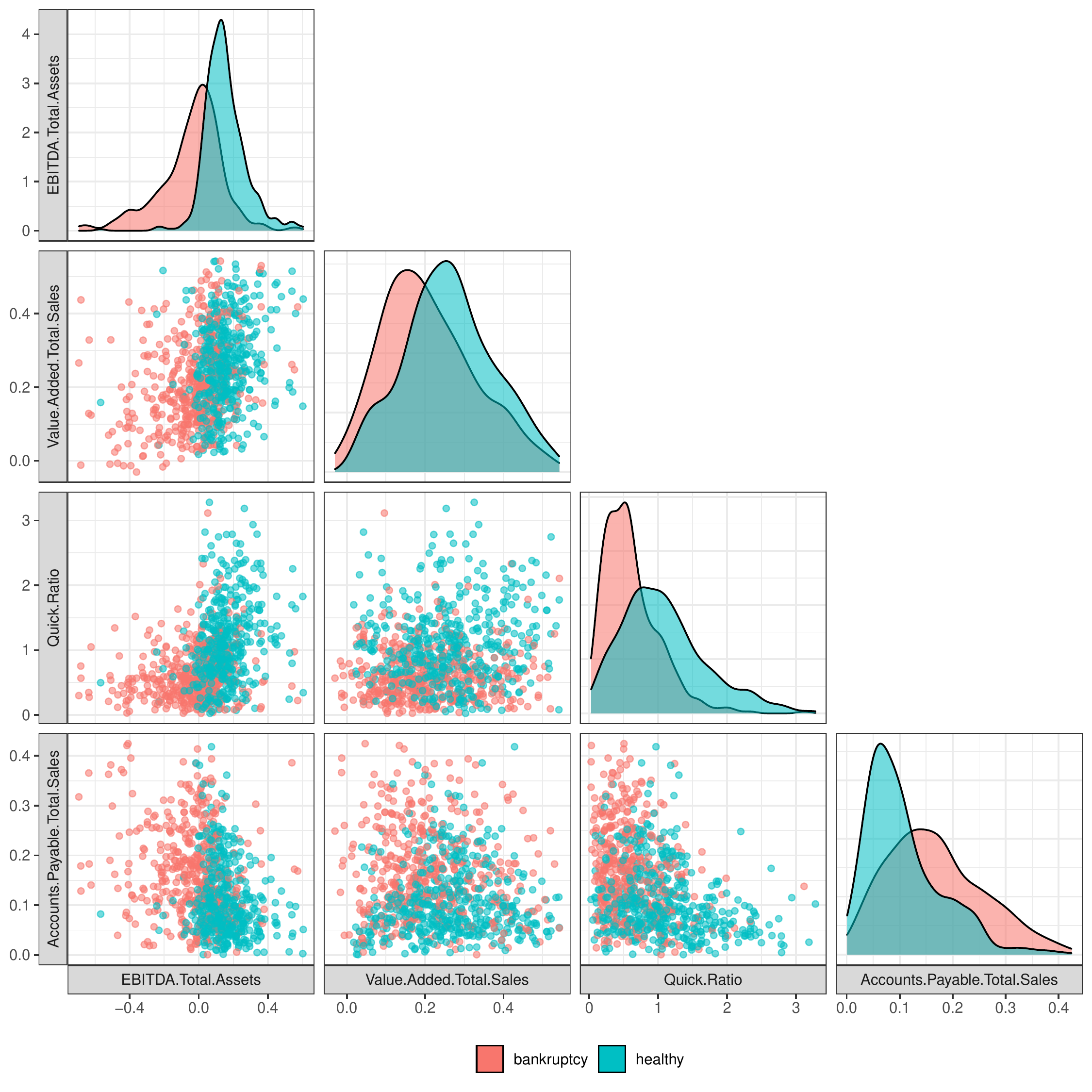}
    \caption{Scatter matrix plot of the \textit{finance} data set, where point is colored red
     if the company is being bankrupt, and blue otherwise.  Marginal densities for both groups are given at the diagonal.}
    \label{fig:as_fig_12}
\end{figure}

\begin{figure}[H]
\centering
    \includegraphics[width=1\textwidth]{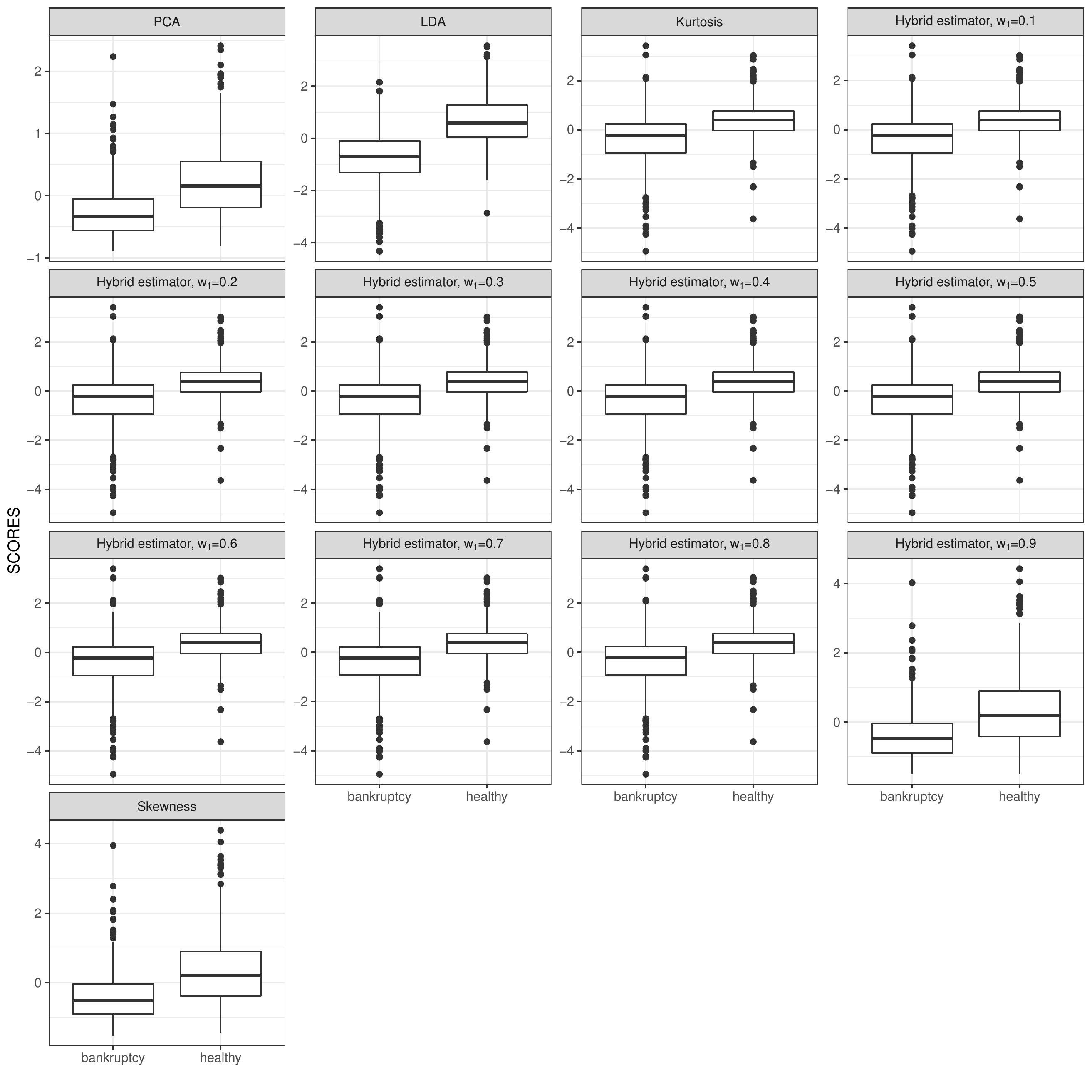}
    \caption{Plot shows boxplots of the projection scores of the \textit{finance} data along the directions based on PCA, LDA, as well as the PP directions obtained by maximizing $(\kappa_n-3)^2$, $\gamma_n^2$ and $\eta_n(\cdot,w_1)$, for $w_1=0.1,0.2,\dots,0.9$ for healthy and bankrupted companies.}
    \label{fig:as_fig_13}
\end{figure}

\newpage
\bibliographystyle{agsm}
\bibliography{refs}

\end{document}